\documentclass[11pt]{amsart} 

\usepackage[utf8]{inputenc} 

\usepackage{lipsum}

\usepackage[margin=1in]{geometry} 

\usepackage{graphicx} 

\usepackage[english]{babel}

\usepackage{booktabs} 
\usepackage{array} 
\usepackage{paralist} 
\usepackage{verbatim} 
\usepackage{subfig} 

\usepackage{caption}

\usepackage{amsfonts,amsmath,amssymb,bbm,amsthm,amsbsy,amscd}
\usepackage{hyperref}
\usepackage{tikz-cd}
\usepackage{mathtools}
\usepackage{bm}
\usepackage{pinlabel}

\usepackage{color}
\usepackage{pdfcolmk}

\usepackage{tikz}
\usepackage{tikz-cd}

\usepackage{IEEEtrantools}

\newenvironment{equ*}[1]{\begin{IEEEeqnarray*}{#1}}{\end{IEEEeqnarray*}}
\numberwithin{equation}{section}

\usepackage{bbm} 

\usepackage{cleveref}

\newenvironment{frcseries}{\fontfamily{frc}\selectfont}{}

\makeatletter                                                   
\newtheorem*{rep@theorem}{\rep@title}
\newcommand{\newreptheorem}[2]{%
	\newenvironment{rep#1}[1]{%
		\def\rep@title{#2 \ref{##1}}%
		\begin{rep@theorem}}%
		{\end{rep@theorem}}}
\makeatother

\theoremstyle{theorem}
\newtheorem{theorem}{Theorem}

\newtheorem*{strong_gap_theorem}{Theorem~\ref{theorem: gaps graph of groups hyperbolic}$'$}

\newtheorem{thm}{Theorem}[section]
\newtheorem{lemma}[thm]{Lemma}

\crefname{lemma}{Lemma}{Lemmata}
\newtheorem{prop}[thm]{Proposition}
\newtheorem{cor}[thm]{Corollary}
\newtheorem{claim}[thm]{Claim}

\newtheorem*{thm*}{Theorem}
\newtheorem*{lemma*}{Lemma}
\newtheorem*{prop*}{Proposition}
\newtheorem*{corr*}{Corrolary}
\newtheorem*{claim*}{Claim}

\theoremstyle{remark}
\newtheorem{rmk}[thm]{Remark}

\newtheorem{quest}[thm]{Question}

\newtheorem*{rmk*}{Remark}
\newtheorem*{conj*}{Conjecture}
\newtheorem*{quest*}{Question}

\theoremstyle{definition}
\newtheorem{defn}[thm]{Definition}
\newtheorem{exmp}[thm]{Example}

\newtheorem*{defn*}{Definition}
\newtheorem*{exmp*}{Example}

\newreptheorem{theorem}{Theorem}
\newreptheorem{corollary}{Corollary}
\newreptheorem{proposition}{Proposition}

\newcommand{\R}{\mathbb{R}}

\newcommand{\Q}{\mathbb{Q}}
\newcommand{\Z}{\mathbb{Z}}
\newcommand{\N}{\mathbb{N}}
\newcommand{\Hbb}{\mathbb{H}}
\newcommand{\Ebb}{\mathbb{E}}
\newcommand{\Qc}{\mathcal{Q}}

\newcommand{\Isom}{\mathrm{Isom}}
\newcommand{\PSLtwo}{\mathrm{PSL}_2}
\newcommand{\PSLtwotilde}{\widetilde{\mathrm{PSL}}_2}
\newcommand{\ug}{\underline{g}}
\newcommand{\uga}{\underline{\gamma}}

\newcommand{\att}{\texttt{a}}
\newcommand{\btt}{\texttt{b}}
\newcommand{\xtt}{\texttt{x}}

\newcommand{\sign}{\textrm{sign}}

\newcommand{\Acl}{\mathcal{A}}

\newcommand{\Trm}{\mathrm{T}}
\newcommand{\Image}{\mathrm{Im}}
\newcommand{\eurm}{\mathrm{eu}}

\newcommand{\area}{\mathrm{area}}
\newcommand{\len}{\mathrm{length}}
\newcommand{\Sthin}{S_\mathrm{thin}}
\newcommand{\Sthick}{S_\mathrm{thick}}

\newcommand{\BS}{\mathrm{BS}}

\newcommand{\Homeo}{\mathrm{Homeo}}

\newcommand{\rot}{\mathrm{rot}}

\newcommand{\cl}{\mathrm{cl}}
\newcommand{\scl}{\mathrm{scl}}
\newcommand{\conv}{\mathrm{conv}}

\newcommand{\inj}{\hookrightarrow}

\newcommand{\col}{\colon}
\newcommand{\defeq}{\vcentcolon=}

\title{Spectral gap of scl in graphs of groups and $3$-manifolds}

\author{Lvzhou Chen}
\address{Department of Mathematics\\ Purdue University\\ West Lafayette, Indiana, USA}
\email[L.~Chen]{lvzhou@purdue.edu}

\author{Nicolaus Heuer}
\address{Tudor Investment Corporation\\ Cambridge, England, UK}
\email[N.~Heuer]{nicolaus.heuer.maths@gmail.com}

\begin{document}

\begin{abstract}
Stable commutator length $\scl_G(g)$ of an element $g$ in a group $G$ is an invariant for group elements sensitive to the geometry and dynamics of $G$.

For any group $G$ acting on a tree, we prove a sharp bound $\scl_G(g)\ge 1/2$ for any $g$ acting without fixed points,
provided that the stabilizer of each edge is \emph{relatively torsion-free} in its vertex stabilizers. 
The sharp gap becomes $1/2-1/n$ if the edge stabilizers are \emph{$n$-relatively torsion-free} in vertex stabilizers.
We also compute $\scl_G$ for elements acting with a fixed point.

This implies many such groups have a \emph{spectral gap}, that is, there is a constant $C > 0$ such that either $\scl_G(g) \geq C$ or $\scl_G(g)=0$. New examples include the fundamental group of any $3$-manifold using the JSJ decomposition, though the gap must depend on the manifold. We also obtain the optimal spectral gap of graph products of group without $2$-torsion.


We prove these statements by characterizing maps of surfaces to a suitable $K(G,1)$. 
For groups acting on trees, we also construct explicit quasimorphisms and apply Bavard's duality to give a different proof of our spectral gap theorem under stronger assumptions.



\end{abstract}

	\maketitle
	
	\tableofcontents
	
	\section{Introduction}
Let $G$ be a group and let $G' = [G,G]$ be its commutator subgroup. For an element $g \in G'$ we define the \emph{commutator length} ($\cl_G(g)$) of $g$ in $G$ as 
$$
\cl_G(g) \defeq \min \{ n \mid \exists x_1, \ldots, x_n, y_1, \ldots, y_n \in G: g = [x_1,y_1] \cdots [x_n, y_n] \},
$$ 
and define the \emph{stable commutator length} ($\scl_G(g)$) of $g$ in $G$ 
 as 
 $$
 \scl_G(g) \defeq \lim_{n \to \infty} \frac{\cl_G(g^n)}{n}.
 $$
 We extend $\scl_G$ to an invariant on the whole group by setting $\scl_G(g) = \scl_G(g^N)/N$ if $g^N \in G'$ for some $N \in \Z_+$ and $\scl_G(g) = \infty$ otherwise. 
Stable commutator length ($\scl$) arises naturally in 
geometry, topology and dynamics. See \cite{Cal:sclbook} for an introduction to stable commutator length.

We say that a group $G$ has a \emph{spectral gap in scl} if there is a constant $C > 0$ such that for every $g \in G$ either $\scl_G(g) \geq C$ or $\scl_G(g)=0$.
We are mainly interested in the case where the gap is nontrivial, meaning that $\scl_G$ does not vanish on the commutator subgroup $[G,G]$.

Nontrivial lower bounds of scl can serve as an obstruction for homomorphisms due to monotonicity of scl. For instance a spectral gap theorem for mapping class groups by \cite{BBF} implies that any homomorphism from an irreducible lattice of a higher rank semisimple Lie group to a mapping class group has finite image, originally a theorem of Farb--Kaimanovich--Masur \cite{homrigidityKM,homrigidityFM}; See Theorem \ref{thm: FarbKaimanovichMasur} for more details.
Sharp spectral gap results have also been used \cite{gentorsion} to show nonexistence of the so-called generalized torsion in certain groups; See also Section \ref{subsec: uniform bounds}.

Many classes of groups have spectral gaps, including free groups, word-hyperbolic groups, mapping class groups of closed surfaces, and right-angled Artin groups. See Subsection \ref{subsec:scl basic} for a list of known results.

In this article we study \emph{sharp} spectral gaps of groups acting on trees without inversion. 
By the work of Bass--Serre, such groups may be algebraically decomposed  into \emph{graphs of groups} built from their edge and vertex stabilizers; See Subsection \ref{subsec:graph of groups}. 
Basic examples of groups acting on trees are amalgamated free products and HNN extensions.
Many classes of groups have a natural graph of groups structure associated to them. Examples include the JSJ decomposition of $3$-manifolds and HNN-hierarchy of one-relator groups,
as well as the decomposition of graph products into amalgamated free products.

An element acting on a tree is called \emph{elliptic} if it stabilizes some vertex and \emph{hyperbolic} otherwise. 
We will discuss the stable commutator length of both types of elements.

We say that a pair of a group $G$ and a subgroup $H \le G$ is \emph{$n$-relatively torsion-free} ($n$-RTF) if there is no $1 \leq k < n$, $g \in G \setminus H $ and $\{ h_i \}_{1 \leq i \leq k} \subset H$ such that
$$
g h_1 \cdots g h_k = 1_G,
$$
and simply \emph{relatively torsion-free} if we can take $n = \infty$; See Definition \ref{defn: n-rft}.
We say that $H$ is \emph{left relatively convex} if there is a $G$-invariant order on the cosets $G/H = \{ g H \mid g \in G \}$ where $G$ acts on the left.
Every left relatively convex subgroup is relatively torsion-free; See Lemma \ref{lemma: left relatively convex implies RTF}.
\begin{theorem}[Theorems \ref{thm: $n$-RTF gap, weak version} and \ref{thm:left relatively convex graph of groups}] \label{theorem: gaps graph of groups hyperbolic}
Let $G$ be a group acting on a tree such that the stabilizer of every edge is $n$-RTF in the stabilizers of the its vertices.
If $g \in G$ is hyperbolic, then
\begin{eqnarray*}
\scl_G(g) &\geq & \frac{1}{2} - \frac{1}{n} \mbox{, if $n \in \N$ and} \\
\scl_G(g) &\geq  & \frac{1}{2} \mbox{, if $n = \infty$.}
\end{eqnarray*}
If the stabilizer of every edge lies left relatively convex in the stabilizers of its vertices, then there is an explicit homogeneous quasimorphism $\phi$ (see Subsection \ref{subsubsec:scl via quasimorphisms} for the definition) on $G$ such that $\phi(g) \geq 1$ and $D(\phi) \leq 1$.
\end{theorem}
Our estimates are sharp, strengthening the estimates in \cite{CFL16} and generalizing all other spectral gap results for graphs of groups known to the authors \cite{Chen:sclfpgap,DH91,Heuer}. See Subsection \ref{subsubsec:scl via admissible surfaces} for a stronger version that gives the estimates for individual elements under weaker assumptions.

The stable commutator length generalizes to \emph{chains}, i.e. linear combinations of elements; See Subsection \ref{subsec:scl basic}. We show how to compute $\scl_G$ of chains of elliptic elements in terms of the stable commutator length of vertex groups.
\begin{theorem}[Theorem \ref{thm: vert scl compute}]
Let $G$ be a group acting on a tree with vertex stabilizers $\{G_v\}$ and let $c_v$ be a chain of elliptic elements in $G_v$.
Then 
$$\scl_G(\sum_v c_v)=\inf\sum_v \scl_{G_v}(c'_v),$$
where each $c'_v$ is a chain of elliptic elements in $G_v$, and the infimum is taken over all collections $\{c'_v\}$ of chains obtained from adjusting $\{c_v\}$ by chains of elements in edge stabilizers.
\end{theorem}
See Theorem \ref{thm: vert scl compute} for a precise statement. For example if $G = A \star_{\Z} B$ where $\Z$ is generated by $t$ then we show that $\scl_G(t)$ is the minimum of $\scl_{A}(t)$ and $\scl_{B}(t)$; See Theorem \ref{thm:scl edge for amalgam}.

We apply our results to obtain spectral gaps of $3$-manifold groups using the JSJ decomposition and geometrization theorem.
\begin{theorem}[Theorem \ref{thm:3mfdgap}]\label{theorem: 3 manifold gap intro}
	For any closed oriented connected $3$-manifold $M$, there is a constant $C(M)>0$ such that for any $g\in \pi_1(M)$ we have either $\scl_{\pi_1 (M)}(g)\ge C(M)$ or $\scl_{\pi_1(M)}(g)=0$.
\end{theorem}

The gap $C(M)$ must depend on $M$; See Example \ref{exmp: surgery example}. However, we classify elements with $\scl_{\pi_1(M)}(g)=0$ and describe those with $\scl_{\pi_1(M)}(g)<1/48$ in Theorem \ref{thm: small scl in 3mfds}. See Subsection \ref{subsec: intro, 3mfd} for more details.

As another application, we obtain spectral gaps for graph products of groups; See Subsection \ref{subsec: Graph product in intro} and Theorem \ref{theorem: gap for graph product} below.

\subsection{Method}
Let $G$ be a group and $g\in G$. There are two equivalent definitions of $\scl_G(g)$ that lead to two very different approaches to proving lower bounds of $\scl_G(g)$.

\subsubsection{Stable commutator length via quasimorphisms}\label{subsubsec:scl via quasimorphisms}
The first one uses Bavard's duality and \emph{homogeneous quasimorphisms}.
A \emph{quasimorphism} on a group $G$ is a map $\phi \col G \to \R$ such that the \emph{defect} $D(\phi)\defeq\sup_{g,h\in G} |\phi(g)+\phi(h)-\phi(gh)|$ is finite.
A quasimorphism is \emph{homogeneous} if it restricts to a homomorphism on each cyclic subgroup.
Bavard's Duality Theorem \ref{thm:Bavards duality} asserts that
$$
\scl_G(g) = \sup_{\phi} \frac{\phi(g)}{2 D(\phi)},
$$
where the supremum ranges over all homogeneous quasimorphisms. 

Using this point of view, to establish a lower bound of $\scl_G(g)$, it suffices to construct \emph{one} good homogeneous quasimorphism for the given element $g \in G$.
This has become a common approach for proving spectral gap theorems \cite{CF:sclhypgrp,BBF,CFL16,Heuer}. 
However, it is very difficult to construct homogeneous quasimorphisms to prove \emph{sharp} bounds.
One recent new idea is to construct certain nice maps (called \emph{letter quasimorphisms}) from $G$ to free groups that allow us to pull back nice quasimorphisms on free groups to ones on $G$. 
This was first used by the second author \cite{Heuer} for amalgamated free products with left relatively convex edge groups.

Here we generalize this approach to graphs of groups, in particular HNN extensions; See Subsection \ref{subsec:extremal qm for lrc subgroups}.
This allows us to construct the desired quasimorphisms (Theorem \ref{thm:left relatively convex graph of groups}) as stated in the second part of Theorem \ref{theorem: gaps graph of groups hyperbolic}.

\subsubsection{Stable commutator length via admissible surfaces} \label{subsubsec:scl via admissible surfaces}
A very different approach to proving lower bounds of $\scl_G(g)$ (or more generally computing it) uses the topological definition of scl via admissible surfaces.

Let $X$ be a topological space with fundamental group $G$ and let $\gamma \col S^1 \to X$ be a loop representing $g \in G$. An oriented surface map $f \col S \to X$ is called \emph{admissible} for $g$ of degree $n(f)>0$, if there is a covering map $\partial f: \partial S\to S^1$ of total degree $n(f)$ such that the diagram
	\begin{center}
		\begin{tikzcd}
			\partial S \arrow[r, "\partial f"] \arrow[d, hook] 	& S^1 \arrow[d, "\gamma"]\\
			S 			\arrow[r,"f"] 									& X
		\end{tikzcd}
	\end{center}
	commutes. 
It is known \cite[Section 2.1]{Cal:sclbook} that
$$
\scl_G(g)=  \inf_{(f,S)} \frac{-\chi^-(S)}{2n(f)},
$$
where the infimum ranges over all admissible surfaces and where $\chi^-$ is the Euler characteristic ignoring sphere and disk components. 

From this point of view, to establish a lower bound of $\scl_G(g)$, we need a uniform lower bound over \emph{all} admissible surfaces.
This seemingly more difficult approach is the main point of view we use in this article.
Actually, the graphs of groups structure of the underlying group $G$ allows us to simplify and decompose admissible surfaces into smaller pieces that are easy to understand.

A graph of groups $G$ has a standard realization $X$ that has fundamental group $G$ and contains \emph{vertex and edge spaces} corresponding to the vertex and edge groups. Each hyperbolic element $g$ is represented by some loop $\gamma$ that cyclically visits finitely many vertex spaces, each time entering the vertex space from one adjacent edge space and exiting from another. A \emph{backtrack} of $\gamma$ at a vertex space $X_v$ is a time when $\gamma$ enters and exits $X_v$ from the same edge space $X_e$. If $\gamma$ is pulled \emph{tight}, each such a backtrack gives rise to a winding number $g_{e,v}\in G_v\setminus G_e$, where $G_v$ and $G_e$ are the vertex and edge groups corresponding to $X_v$ and $X_e$. See Subsection \ref{subsec: surfs in graphs of groups} for more details.

For a subgroup $H\le G$ and $k\ge 2$, an element $g\in G\setminus H$ \emph{has order $\ge n$} rel $H$ if for all $k<n$ and $h_1,\ldots, h_k\in H$ we have
$$gh_1\ldots gh_k\neq 1_G.$$
Then $H$ is $n$-RTF if and only if each $g\in G\setminus H$  has order $\ge n$ rel $H$.

In Section \ref{sec:scl on hyperbolic elts} we prove the following stronger version of Theorem \ref{theorem: gaps graph of groups hyperbolic}.
\begin{strong_gap_theorem}[Theorem \ref{thm: $n$-RTF gap, strong version}]
	Let $G$ be a graph of groups and let $3\le n\le\infty$. Suppose $g$ is represented by a tight loop $\gamma$ so that the winding number $g_{e,v}$ associated to any backtrack at a vertex space $X_v$ through an edge space $X_e$ has order $\ge n$ rel $G_e$ in $G_v\setminus G_e$. Then 
	$$\scl_G(g)\ge \frac{1}{2}-\frac{1}{n}.$$
\end{strong_gap_theorem}

The proof is based on a \emph{linear programming duality method} that we develop to uniformly estimate the Euler characteristic of all admissible surfaces in $X$ in \emph{normal form}. 
The normal form is obtained by cutting the surface along edge spaces and simplifying the resulting surfaces, similar to the one in \cite{Chen:sclBS}. 
The linear programming duality method is a generalization of the argument for free products by the first author \cite{Chen:sclfpgap}.

\subsection{Graph products}\label{subsec: Graph product in intro}
Let $\Gamma$ be a simple and not necessarily connected graph with vertex set $V$ and let $\{G_v\}_{v \in V}$ be a collection of groups. The \emph{graph product} $G_\Gamma$ is the quotient of the free product $\star_{v \in V} G_v$ subject to the relations $[g_u,g_v]$ for any  $g_u\in G_u$ and $g_v\in G_v$ such that $u,v$ are adjacent vertices. 

Several classes of non-positively curved groups are graph products including right-angled Artin groups; See Example \ref{exmp:graph products}. Each vertex of the graph induces a splitting of the graph product as an amalgam; See Lemma \ref{lemma: graph products splitting}. 
Using such a structure, we show:
	\begin{theorem}[Theorem \ref{thm: gap for graph products, element-wise statement}]\label{theorem: gap for graph product}
		Let $G_\Gamma$ be a graph product. Suppose $g=g_{1}\cdots g_{m}\in G_\Gamma$ ($m\ge 1$) is in cyclically reduced form and there is some $3\le n\le \infty$ such that $g_{i}\in G_{v_i}$ has order at least $n$ for all $1\le i\le k$. Then either
		$$\scl_{G_\Gamma}(g)\ge \frac{1}{2}-\frac{1}{n},$$
		or $\Gamma$ contains a complete subgraph $\Lambda$ with vertex set $\{v_1,\ldots, v_m\}$. In the latter case, we have
		$$\scl_{G_\Gamma}(g)=\scl_{G_\Lambda}(g)=\max \scl_{G_i}(g_i).$$
	\end{theorem}
The estimate is sharp: For $g_v\in G_v$ of order $n\ge2$ and $g_u\in G_u$ of order $m\ge 2$ with $u$ not adjacent to $v$ in $\Gamma$, we have $\scl_{G_\Gamma}([g_u,g_v])=\frac{1}{2}-\frac{1}{\min(m,n)}$; See Remark \ref{rmk: graph product gap sharp}. In particular, for a collection of groups with a uniform spectral gap and without $2$-torsion, their graph product has a spectral gap.

In the special case of right-angled Artin groups, this provides a new proof of the sharp $1/2$ gap \cite{Heuer} that is topological in nature.

\subsection{$3$-manifold groups}\label{subsec: intro, 3mfd}

Let $G$ be the fundamental group of a closed oriented connected $3$-manifold $M$. The prime decomposition of $M$ canonically splits $G$ as a free product. For each \emph{non-geometric} prime factor, the corresponding free factor has the structure of a graph of groups by the JSJ decomposition, where each vertex group is the fundamental group of a geometric $3$-manifold by the geometrization theorem. 

Using this structure, we prove Theorem \ref{theorem: 3 manifold gap intro} for any $3$-manifold group in Section \ref{sec: spectral gap 3 mnfd}. This positively answers a question that Genevieve Walsh asked about the existence of spectral gaps of $3$-manifolds after a talk by Joel Louwsma at an AMS sectional meeting in 2017.

Although the spectral gap in Theorem \ref{theorem: 3 manifold gap intro} cannot be uniform, its proof implies that elements with scl less than $1/48$ must take certain special forms, and it allows us to classify elements with zero scl; See Theorem \ref{thm: small scl in 3mfds}. Besides, as is suggested by Michael Hull, prime $3$-manifolds only with hyperbolic pieces in the JSJ decomposition have finitely many conjugacy classes with scl strictly less than $1/48$ (Corollary \ref{cor: finitely many conj class with small scl}).

For hyperbolic elements in prime factors, we also have a uniform gap $1/48$ using the acylindricity of the action and a gap theorem of Clay--Forester--Louwsma \cite[Theorem 6.11]{CFL16}. This gap can be improved to $1/6$ unless the prime factor contains in its geometric decomposition either the twisted $I$-bundle over a Klein bottle or a Seifert fibered space over a hyperbolic orbifold that contains no cone points of even order. This is accomplished by using geometry to verify the $3$-RTF condition in Theorem \ref{theorem: gaps graph of groups hyperbolic}.

The proof of Theorem \ref{theorem: 3 manifold gap intro} relies heavily on estimates of scl in vertex groups (Theorem \ref{thm: vertex scl for non-geometric prime factors}). On the one hand this uses a simple estimate in terms of \emph{relative} stable commutator length (Lemma \ref{lemma: simple estimate of scl in vertex groups}) together with generalized versions of earlier gap results of hyperbolic groups \cite{Cal:sclhypmfd,CF:sclhypgrp}. On the other hand, this relies on our characterization of scl in edge groups (Corollary \ref{cor: scl edge compute}), where the simple estimate above is useless.

\subsection*{Organization of the paper}
This article is organized as follows. In Section \ref{sec:background} we recall basic or well known results on stable commutator length and its relative version,
including quasimorphisms and Bavard's duality.
In Section \ref{sec:scl on hyperbolic elts} we develop a linear programming duality method to estimate $\scl$ of hyperbolic elements in graphs of groups and prove Theorem \ref{theorem: gaps graph of groups hyperbolic}. 
Subsection \ref{subsec: n-RTF} includes a discussion on the crucial $n$-RTF conditions and key examples.
Subsection \ref{subsec:extremal qm for lrc subgroups} is devoted to the construction of explicit quasimorphisms that prove our spectral gap Theorem \ref{theorem: gaps graph of groups hyperbolic} under the stronger left relatively convex assumption.
In Section \ref{sec: scl vertex and edge group} we compute stable commutator length for edge group elements in graphs of groups.
Finally, we apply our results to obtain spectral gaps in graph products (Section \ref{sec:spectral gap graph products}) and $3$-manifolds (Section \ref{sec: spectral gap 3 mnfd}).

\subsection*{Acknowledgments}
	We would like to thank Martin Bridson, Danny Calegari and Jason Manning for many helpful suggestions and encouragement. We also thank Lei Chen, Max Forester, Michael Hull, Qixiao Ma, Ben McReynolds, Genevieve Walsh and Yiwen Zhou for useful conversations. We are grateful to the anonymous referee for careful reading and valuable suggestions that improve the exposition.

	\section{Background} \label{sec:background}

	\subsection{Stable commutator length} \label{subsec:scl basic}
	\begin{defn}\label{def: scl alg}
		Let $G$ be a group and $G'$ its commutator subgroup. Each element $g\in G'$ may be written as a product of commutators $g=[a_1,b_1]\cdots[a_k,b_k]$. The smallest such $k$ is called the \emph{commutator length of $g$} and denoted by $\cl_G(g)$. The \emph{stable commutator length} is the limit
		$$\scl_G(g)\defeq \lim_{n \to \infty} \frac{\cl_G(g^n)}{n}.
		$$
	\end{defn}
It is easy to see that $\cl_G$ is subadditive and thus that the limit above always exists.
	
Stable commutator length can be equivalently defined and generalized using \emph{admissible surfaces}. 

	An integral (rational, or real resp.) \emph{chain} $c=\sum c_ig_i$ is a finite formal sum of group elements $g_1,\ldots, g_m\in G$ with integral (rational, or real resp.) coefficients $c_i$. 
	
	Let $X$ be a space with fundamental group $G$ and let $c=\sum c_ig_i$ be a rational chain. Represent each $g_i$ by a loop $\gamma_i: S^1_i\to X$. An \emph{admissible surface} of degree $n(f)\ge1$ is a map $f:S\to X$ from a compact oriented surface $S$ with boundary $\partial S$ such that the following diagram commutes and $\partial f[\partial S]=n(f)\sum_i c_i[S^1_i]$,
	\begin{center}
		\begin{tikzcd}
			\partial S \arrow[r, "\partial f"] \arrow[d, hook] 	& \sqcup S^1_i \arrow[d, "\sqcup \gamma_i"]\\
			S 			\arrow[r,"f"] 									& X
		\end{tikzcd}
	\end{center}
	where $\partial S\inj S$ is the inclusion map. 

Such surfaces exist when the rational chain $c$ is null-homologous, i.e. $[c]=0\in H_1(G;\Q)$. 
An admissible surface $S$ is \emph{monotone} if $\partial f$ is a covering map of positive degree on every boundary component of $S$.
	
	
	For any \emph{connected} orientable compact surface $S$, let $\chi^-(S)$ be $\chi(S)$ unless $S$ is a disk or a sphere, in which case we set $\chi^-(S)=0$. If $S$ is disconnected, we define $\chi^-(S)$ as the sum of $\chi^-(\Sigma)$ over all components $\Sigma$ of $S$. Equivalently, $\chi^-(S)$ is the Euler characteristic of $S$ after removing disk and sphere components.
	
	\begin{defn}\label{def: scl top}
		For a null-homologous rational chain $c=\sum c_ig_i$ in $G$ as above, its stable commutator length is defined as
		$$\scl_G(c)=\inf_S \frac{-\chi^-(S)}{2n(f)},$$
		where the infimum is taken over all admissible surfaces $S$, where $n(f)$ is the corresponding degree. 
		If $c$ is nontrivial in the first (rational) homology, we make the convention that $\scl_G(c)=+\infty$.
	\end{defn}
	
	By \cite[Proposition 2.13]{Cal:sclbook}, the infimum remains the same if we restrict to monotone admissible surfaces,
	For the rest of this paper we mostly use monotone admissible surfaces.
	If a chain is a single element $c = g \in G'$, then	this topological definition agrees with Definition \ref{def: scl alg}.

	Let $C_1(G)$ be the space of real chains in $G$, i.e. the $\R$-vector space with basis $G$. 
	Let $H(G)$ be the subspace spanned by all elements of the form $g^n - ng$ for $g \in G$ and $n \in \Z$ and 
	$hgh^{-1}-g$ for $g, h \in G$.
	There is a well defined linear map $h_G:C_1^H(G)\to H_1(G;\R)$ sending each chain to its homology class, where $C_1^H(G):=C_1(G)/H(G)$.
	We denote the kernel by $B_1^H(G)$. Definition~\ref{def: scl top} extends uniquely to null-homologous real chains by continuity:
	\begin{defn}\label{def: scl of real chain}
		For any null-homologous real chain $c=\sum c_i g_i\in C_1(G)$, there are null-homologous rational chains $c'=\sum c'_ig_i$, $c'_i\in\Q$ with each $c'_i$ arbitrarily close to $c_i$.
		As $c'_i\in\Q$ tends to $c_i$ for all $i$, for such chains $\scl_G(c')$ converges to a unique number, which agrees with $\scl_G(c)$ if $c$ is itself rational and is defined to be $\scl_G(c)$ if $c$ is irrational.
	\end{defn}
	
	Scl vanishes on $H(G)$ and induces a semi-norm on $B_1^H(G)$. See \cite[Chapter 2]{Cal:sclbook}. 
	Sometimes, $\scl$ is a genuine norm, for example if $G$ is word-hyperbolic \cite{CF:sclhypgrp}.
	
	Here are some basic properties of scl that easily follow from the definitions.
	\begin{lemma}\label{lemma: basic prop of scl}
		\leavevmode
		\begin{enumerate}
			\item (Stability) $\scl_G(g^n)=n\cdot\scl_G(g)$;
			\item (Monotonicity) For any homomorphism $\phi:G\to H$ and any chain $c\in C_1(G)$, we have $\scl_G(c)\ge \scl_H(\phi(c))$;
			\item (Retract) If a subgroup $H\le G$ is a retract, i.e. there is a homomorphism $r:G\to H$ with $r|_H=id$, then $\scl_H(c)=\scl_G(c)$ for all chains $c\in C_1(H)$;
			\item (Direct product) For $a\in A'$ and $b\in B'$ in the direct product $G=A\times B$, we have $\cl_G(a,b)=\max\{\cl_A(a),\cl_B(b)\}$ and $\scl_G(a,b)=\max\{\scl_A(a),\scl_B(b)\}$.\label{item: basic prop direct prod}
		\end{enumerate}
	\end{lemma}
	
	Note that finite-order elements can be removed from a chain without changing scl. Thus we will often assume elements in chains to have infinite order.
	
	\begin{defn} \label{defn: spectral gap}
		A group $G$ has a \emph{spectral gap} $C>0$ if for any $g\in G$ either $\scl_G(g)\ge C$ or $\scl_G(g)=0$. If in addition, the case $\scl_G(g)=0$ only occurs when $g$ is torsion, we say $G$ has a \emph{strong spectral gap} $C$.
	\end{defn}
	
	Many classes of groups are known to have a spectral gap.
	\begin{enumerate}
		\item $G$ trivially has a spectral gap $C$ for any $C>0$ if $\scl_G$ vanishes on $B_1^H(G)$, which is the case if $G$ is amenable \cite[Theorem 2.47]{Cal:sclbook} or an irreducible lattice in a semisimple Lie group of higher rank \cite{BurgerMonod,BurgerMonod02}. The gap is strong if $G$ is abelian;
		\item $G$ has a strong spectral gap $1/2$ if $G$ is residually free \cite{DH91}.
		\item $\delta$-hyperbolic groups have gaps depending on the number of generators and $\delta$ \cite{CF:sclhypgrp}. The gap is strong if the group is also torsion-free.
		\item Any finite index subgroup of the mapping class group of a (possibly punctured) closed surface has a spectral gap \cite{BBF}. Moreover, each such mapping class group contains a finite index (torsion-free) subgroup $G$ that has a strong spectral gap $\epsilon>0$.
		\item All right-angled Artin groups have a strong gap $1/2$ \cite{Heuer}; See \cite{RAAGgap1,RAAGgap2} for earlier weaker estimates.
		\item All Baumslag--Solitar groups have a gap $1/12$ \cite{CFL16}.
	\end{enumerate}
	Our Theorem \ref{theorem: 3 manifold gap intro} adds all $3$-manifold groups to the list.
	
	The gap property is essentially preserved under taking free products.
	\begin{lemma}[Clay--Forester--Louwsma]\label{lemma: freeprod preserves gap}
		Let $G=\star_\lambda G_\lambda$ be a free product. Then for any $g\in G$ not conjugate into any free factor, we have either $\scl_G(g)=0$ or $\scl_G(g)\ge 1/12$. Moreover, $\scl_G(g)=0$ if and only if $g$ is conjugate to $g^{-1}$. Thus if the groups $G_\lambda$ have a uniform spectral gap $C>0$, then $G$ has a gap $\min\{C,1/12\}$.
	\end{lemma}
	\begin{proof}
		If $g$ is not conjugate into any free factor, then it either satisfies the so-called well-aligned condition in \cite{CFL16} or is conjugate to its inverse. For such $g$, it follows from \cite[Theorem 6.9]{CFL16} that either $\scl_G(g)\ge 1/12$ or $\scl_G(g)=0$, corresponding to the two situations. Assuming the free factors have a uniform gap $C$, if an element $g\in G$ is conjugate into some free factor $G_\lambda$, then $\scl_G(g)=\scl_{G_\lambda}(g)\ge C$ since $G_\lambda$ is a retract of $G$.
	\end{proof}
	The constant $1/12$ is optimal in general, but can be improved if there is no torsion of small order. See \cite{Chen:sclfpgap} or \cite{IK}.
	
	Many other groups have a uniform positive lower bound on \emph{most} elements. They often satisfy a spectral gap in a relative sense, which we introduce in Subsection \ref{subsec: rel scl}.
	
	One can obtain rigidity results using spectral gap theorems. The following proof is not original but we cannot find it in the literature.
	\begin{thm}[Farb--Kaimanovich--Masur \cite{homrigidityKM,homrigidityFM}]\label{thm: FarbKaimanovichMasur}
		For any irreducible higher rank lattice $\Gamma$ and a mapping class group $\mathrm{Mod}(S)$, any homomorphism $h:\Gamma\to \mathrm{Mod}(S)$ has finite image.
	\end{thm}
	\begin{proof}
		By \cite[Theorem B]{BBF}, there is $\epsilon>0$ and a finite index subgroup $G\le\mathrm{Mod}(S)$ such that $\scl_G(g)\ge\epsilon$ for all $g\neq id\in G$.
		We may replace $\Gamma$ by its finite index subgroup $\Lambda\defeq h^{-1} G$, for which the same assumption holds.
		So we may assume without loss of generality that $h(\Gamma)\subset G$.
		By theorems of Burger--Monod \cite{BurgerMonod,BurgerMonod02} and the fact that higher rank lattices have finite abelianization, we have $\scl_\Gamma(\gamma)= 0$ for all $\gamma\in \Gamma$.
		Hence by monotonicity, $\scl_G(h(\gamma))\le \scl_\Gamma(\gamma)=0$, so $\scl_G(h(\gamma))=0$ and $h(\gamma)=id$ for all $\gamma\in \Gamma$.
		Since we passed to a finite index subgroup at the beginning, we only conclude that $h$ has finite image in general.		
	\end{proof}

	\subsection{Quasimorphisms and Bavard's duality} \label{subsec:QM and Bavard}
	Stable commutator length is dual to the so-called homogeneous quasimorphisms.
	A \emph{quasimorphism} on a group $G$ is a map $\phi: G \to \R$ for which
	$$
	D(\phi)\defeq \sup_{g,h\in G} |\phi(g) + \phi(h) - \phi(gh) | <\infty,
	$$
	and $D(\phi)$ is called the \emph{defect} of $\phi$.
	A quasimorphism $\phi: G \to \R$ is \emph{homogeneous} if for all $g \in G$ and $n \in \Z$ we have $\phi(g^n) = n \cdot \phi(g)$.
	Homogeneous quasimorphisms form a vector space $\Qc(G)$ under pointwise addition and $\R$-scalar multiplication.
	Homomorphisms to $\R$ are trivial examples of homogeneous quasimorphisms, which form a vector subspace $H^1(G)\le \Qc(G)$.
	
	We collect some well known properties for homogeneous quasimorphisms:
	\begin{prop}[\cite{Cal:sclbook}]\label{prop: split hqm on comm elem}
		Let $\phi: G \to \R$ be a homogeneous quasimorphism. Then $\phi$ is constant on each conjugacy class and restricts to a homomorphism on each amenable subgroup $H \le G$.
		In particular, $\phi(gh)=\phi(g)+\phi(h)$ if $g$ and $h$ commute.
	\end{prop}

	We can pullback a homogeneous quasimorphism $\phi\in \Qc(G)$ by any homomorphism $f:H\to G$ via $f^*\phi\defeq \phi\circ f\in \Qc(H)$. 
	It follows from the definition that $D(f^*\phi)\le D(\phi)$.
 
%
%

	Bavard's duality theorem \cite{bavard} provides the connection to stable commutator length. 
	Here we state the generalized version for chains \cite[Theorem 2.79]{Cal:sclbook}.
	\begin{thm}[\protect{Bavard's duality}] \label{thm:Bavards duality}
		Let $c$ be a null homologous chain in a group $G$. Then
		\[
		\scl(c) = \sup_{\phi} \frac{|\phi(c)|}{2 D(\phi)},
		\]
		where the supremum is taken over all homogeneous quasimorphisms $\phi\in\Qc(G)$, 
		and $\phi(c)$ is defined by linearity.
	\end{thm}
	Equivalently, this shows that the quotient space of $B_1(G)$ by the subspace of chains with zero scl, equipped the induced scl norm, has dual space isomorphic to $\Qc(G)/H^1(G)$ with norm $2D(\cdot)$; See \cite[Sections 2.4 and 2.5]{Cal:sclbook}.
	
	For a fixed $g \in G$ the supremum is achieved by an \emph{extremal} quasimorphism $\phi$ \cite[Proposition 2.88]{Cal:sclbook}. However, extremal quasimorphisms are notoriously hard to construct.
	For free groups, explicit constructions of extremal quasimorphisms are known for words with scl value $1/2$ but not in general; See \cite{Heuer} and also \cite{calegari:extremal, CFL16}.
	
	Note that, by Bavard's duality, we obtain a lower bound $\scl(c)\ge C$ once we construct a homogeneous quasimorphism $\phi\in \Qc(G)$ with $D(\phi)\le 1$ and $\phi(c)\ge 2C$. In this case, we say the quasimorphism $\phi$ \emph{detects} the lower bound $C$. Similarly, one may construct a family of quasimorphisms to obtain a spectral gap $C$, in which case we say such quasimorphisms detect the spectral gap $C$.
	
	\subsection{Relative stable commutator length}\label{subsec: rel scl}
	
	We will use relative stable commutator length to state our results in the most natural and the strongest form. It was informally mentioned or implicitly used in \cite{CF:sclhypgrp,CFL16,IK}, and it was formalized and shown to be useful in scl computations in \cite{Chen:sclBS}.
	
	\begin{defn} \label{defn:relative scl}
		Let $\{G_\lambda\}_{\lambda\in \Lambda}$ be a collection of subgroups of $G$. Let $C_1(\{G_\lambda\})$ be the subspace of $C_1(G)$ consisting of chains of the form $\sum_\lambda c_\lambda$ with $c_\lambda\in C_1(G_\lambda)$, where all but finitely many $c_\lambda$ vanish in each summation. 
		
		For any chain $c\in C_1(G)$, define its \emph{relative stable commutator length} to be
		$$\scl_{(G,\{G_\lambda\})}(c)\defeq \inf \{\scl_{G}(c+c') : c'\in C_1(\{G_\lambda\})\}.$$
	\end{defn}
	
	Let $H_1(\{G_\lambda\})\le H_1(G;\R)$ be the subspace of homology classes represented by chains in $C_1(\{G_\lambda\})$. Recall that we have a linear map $h_G:C_1^H(G)\to H_1(G)$ taking chains to their homology classes. Denote $B_1^H(G,\{G_\lambda\})\defeq h_G^{-1}H_1(\{G_\lambda\})$, which contains $B_1^H(G)$ as a subspace. Then $\scl_{(G,\{G_\lambda\})}$ is finite on $B_1^H(G,\{G_\lambda\})$ and is a semi-norm.
	
	The following basic properties of relative scl will be used later.
	\begin{lemma}\label{lemma: basic prop of rel scl}
		Let $G$ be a group and $\{G_\lambda\}$ be a collection of subgroups.
		\begin{enumerate}
			\item $\scl_G(c)\ge \scl_{(G,\{G_\lambda\})}(c)$ for any $c\in C_1^H(G)$.
			\item If $g^n$ is conjugate into some $G_\lambda$ for some integer $n\neq0$, then $\scl_{(G,\{G_\lambda\})}(g)=0$.
			\item (Stability) For any $g\in G$, we have $\scl_{(G,\{G_\lambda\})}(g^n)=n\cdot\scl_{(G,\{G_\lambda\})}(g)$.
			\item (Monotonicity) Let $\phi:G\to H$ be a homomorphism such that $\phi(G_\lambda)\subset H_\lambda$ for a collection of subgroups $H_\lambda$ of $H$, then for any $c\in C_1(G)$ we have $$\scl_{(G,\{G_\lambda\})}(c)\ge\scl_{(H,\{H_\lambda\})}(\phi(c)).$$
		\end{enumerate}
	\end{lemma}
	
	For rational chains, relative scl can be computed using \emph{relative admissible surfaces}, which are admissible surfaces possibly with extra boundary components in $\{G_\lambda\}$. This is \cite[Proposition 2.9]{Chen:sclBS} stated as Lemma \ref{lemma: rel admissible} below.
	
	\begin{defn}
	Let  $c\in B_1^H(G,\{G_\lambda\})$ be a rational chain. A surface $S$ together with a specified collection of boundary components $\partial_0\subset \partial S$ is called \emph{relative admissible} for $c$ of degree $n>0$ if $\partial_0$ represents $[nc]\in C_1^H(G)$ and every other boundary component of $S$ represents an element conjugate into $G_\lambda$.
	\end{defn}
	\begin{lemma}[\cite{Chen:sclBS}]\label{lemma: rel admissible}
		For any rational chain $c\in B_1^H(G,\{G_\lambda\})$, we have
		$$\scl_{(G,\{G_\lambda\})}(c)=\inf \frac{-\chi^-(S)}{2n},$$
		where the infimum is taken over all relative admissible surfaces for $c$.
	\end{lemma}
	
	The Bavard duality Theorem \ref{thm:Bavards duality} naturally generalizes to relative stable commutator length.
	\begin{lemma}[Relative Bavard's Duality]\label{lemma: Bavard duality for rel scl}
		For any chain $c\in B_1^H(G,\{G_\lambda\})$, we have
		$$\scl_{(G,\{G_\lambda\})}(c)=\sup \frac{f(c)}{2D(f)},$$
		where the supremum is taken over all homogeneous quasimorphisms $f$ on $G$ that vanish on $C_1(\{G_\lambda\})$.
	\end{lemma}
	\begin{proof}
		This essentially follows from Bavard's duality and the following standard fact: For a normed vector space $(W,|\cdot|)$ and its dual $(W^*,|\cdot|_*)$, the quotient $V=W/U$ by a closed subspace $U$ with the induced norm had dual space isometrically isomorphic to the subspace of $W^*$ consisting of functionals that vanish on $U$.
		
		Recall that $\mathcal{Q}(G)$ is the space of homogeneous quasimorphisms on $G$. Let $N(G)$ be the subspace of $B_1(G)$ where $\scl$ vanishes. Then the quotient $B_1(G)/N(G)$ with induced $\scl$ becomes a normed vector space. Denote the quotient map by $\pi: B_1(G)\to B_1(G)/N(G)$. As we mentioned earlier, Bavard's duality Theorem \ref{thm:Bavards duality} shows that the dual space of $B_1(G)/N(G)$ is exactly $\mathcal{Q}(G)/H^1(G)$ equipped with the norm $2D(\cdot)$. Then scl further induces a norm $\|\cdot\|$ on the quotient space $V$ of $B_1(G)/N(G)$ by the closure of $\pi(C_1(\{G_\lambda\})\cap B_1(G))$. By definition we have $\|\bar{c}\|=\scl_{(G,\{G_\lambda\})}(c)$ for any $c\in B_1(G)$, where $\bar{c}$ is the image in $V$. By the standard fact above, the dual space of $V$ is naturally isomorphic to the subspace of $\mathcal{Q}(G)/H^1(G)$ consisting of linear functionals that vanish on $C_1(\{G_\lambda\})\cap B_1(G)$. Any $\bar{f}\in \mathcal{Q}(G)/H^1(G)$ with this vanishing property can be represented by some $f\in \mathcal{Q}(G)$ that vanishes on $C_1(\{G_\lambda\})$. This proves the assertion assuming $c\in B_1^H(G)$. The general case easily follows since any $c\in B_1^H(G,\{G_\lambda\})$ can be replaced by $c+c'\in B_1^H(G)$ for some $c'\in C_1(\{G_\lambda\})$ without changing both sides of the equation.
	\end{proof}
	
	\begin{defn}
		For a collection of subgroups $\{G_\lambda\}$ of $G$ and a positive number $C$, we say $(G, \{G_\lambda\})$ has a \emph{strong relative spectral gap} $C$ if either $\scl_{(G,\{G_\lambda\})}(g)\ge C$ or $\scl_{(G,\{G_\lambda\})}(g)=0$ for all $g\in G$, where the latter case occurs if and only if $g^n$ is conjugate into some $G_\lambda$ for some $n\neq 0$.
	\end{defn}
	
	Some previous work on spectral gap properties of scl can be stated in terms of or strengthened to strong relative spectral gap. Here are two results that we will use in Section \ref{sec: spectral gap 3 mnfd}.
	
	\begin{thm}\label{thm: relative gap restate}
		\leavevmode
		\begin{enumerate}
			\item \cite[Theorem A$'$]{CF:sclhypgrp} Let $G$ be $\delta$-hyperbolic with symmetric generating set $S$. Let $a$ be an element with $a^n\neq ba^{-n}b^{-1}$ for all $n\neq 0$ and all $b\in G$. Let $\{a_i\}$ be a collection of elements with translation lengths bounded by $T$. Suppose $a^n$ is not conjugate into any $G_i\defeq\langle a_i\rangle$ for any $n\neq 0$, then there is $C=C(\delta,|S|,T)>0$ such that $\scl_{(G,\{G_i\})}(a)\ge C$.\label{item: rel gap restate 1}
			\item \cite[Theorem C]{Cal:sclhypmfd} Let $M$ be a compact $3$-manifold with tori boundary (possibly empty). Suppose the interior of $M$ is hyperbolic with finite volume. Then $(\pi_1 M,\pi_1 \partial M)$ has a strong relative spectral gap $C(M)>0$, where $\pi_1 \partial M$ is the collection of peripheral subgroups.\label{item: rel gap restate 2}
		\end{enumerate}
	\end{thm}
	\begin{proof}
		Part (\ref{item: rel gap restate 1}) is an equivalent statement of the original theorem \cite[Theorem A$'$]{CF:sclhypgrp} in view of Lemma \ref{lemma: Bavard duality for rel scl}.		
		Part (\ref{item: rel gap restate 2}) is stated stronger than the original form \cite[Theorem C]{Cal:sclhypmfd} but can be proved in the same way. See Theorem \ref{thm: hyp rel gap} for the detailed proof.
	\end{proof}
	
	In addition, our Theorem \ref{thm: $n$-RTF gap, weak version} immediately implies the following strong relative spectral gap results, which are stronger than the original statements cited.
	\begin{thm}
		\leavevmode
		\begin{enumerate}
			\item \cite[Theorem 3.1]{Chen:sclfpgap} Let $n\ge 3$ and let $G=\star_{\lambda} G_\lambda$ be a free product where $G_\lambda$ has no $k$-torsion for all $k<n$. Then $(G,\{G_\lambda\})$ has a strong relative spectral gap $\frac{1}{2}-\frac{1}{n}$.
			\item \cite[Theorem 6.3]{Heuer} Suppose we have inclusions of groups $C\inj A$ and $C\inj B$ such that both images are left relatively convex subgroups (see Definition \ref{defn: left relatively convex}). Let $G=A\star_{C} B$ be the associated amalgam. Then $(G,\{A,B\})$ has a strong relative spectral gap $\frac{1}{2}$.
		\end{enumerate}
		
	\end{thm}
	
	\subsection{Bounded cohomology and circle actions} \label{subsec:circle actions and euler class}
	
	We recall the close connection between quasimorphisms and the second bounded cohomology,
	with an emphasis on the rotation quasimorphism and the (bounded) Euler class associated to a group acting on the circle.
	See \cite{Frigerio}, \cite{ghys} and also \cite{cicle_quasimorph_modern} for a thorough treatment of such topics.
	
	Let $G$ be a group and let $V$ be either $\Z$ or $\R$. The \emph{bounded cohomology with coefficients in $V$}, denoted as $H^*_b(G;V)$, is the homology of the (co)chain complex $(C^n_b(G,V),\delta^n)$, where we think of $C^n_b(G,V)$ as the set of $V$-valued bounded functions on $G^n$ in the inhomogeneous resolution setting.
	There is a natural comparison map $c: H^n_b(G;V)\to H^n(G;V)$ by treating a bounded cocycle just as a usual cocycle.
	
	The coboundary $\delta\phi$ of any homogeneous quasimorphism $\phi$ on $G$ represents a bounded cohomology class in $H^2_b(G;\R)$, which is trivial if $\phi\in H^1(G;\R)$.
	This defines a linear map $\delta: \Qc(G)/H^1(G;\R)\to H^2_b(G;\R)$, which fits into the following exact sequence \cite[Theorem 2.50]{Cal:sclbook}:
	$$0\to \Qc(G)/H^1(G;\R) \stackrel{\delta}{\to} H^2_b(G;\R) \stackrel{c}{\to} H^2(G;\R).$$
	
	For an exact sequence of group homomorphisms $K\to G\to H\to 1$, we also have the following exact sequence \cite{Bouarich}:
	$$0\to H^2_b(H,\R) \to H^2_b(G,\R)\to H^2_b(K,\R).$$
	The same holds replacing $H^2_b(\cdot)$ by $\Qc(\cdot)$; See also \cite[Theorem 2.49 and Remark 2.90]{Cal:sclbook}.
	
	Now we turn to orientation-preserving homeomorphisms of the circle.
	There is a \emph{bounded Euler class} $\eurm_b^{\Z}\in H^2_b(\Homeo^+(S^1);\Z)$, 
	whose image under the comparison map is the ordinary Euler class $\eurm^{\Z}\in H^2(\Homeo^+(S^1);\Z)$.
	Identify $S^1$ with $\R/\Z$ and let $\tau(x)=x+1$ be a generator of the deck transformation. 
	Then all lifts of homeomorphisms in $\Homeo^+(S^1)$ to the universal cover $\R$ form a group
	$$
	\Homeo^+_\Z(\R) \defeq \{ h\in \Homeo^+(\R) \mid h\circ \tau =\tau\circ h \}.
	$$
	This gives a central extension
	$$1\to \Z \to \Homeo^+_\Z(\R) \stackrel{\pi}{\to} \Homeo^+(S^1)\to 1,$$
	where the central subgroup $\Z$ is generated by $\tau$. The Euler class associated to this central extension is exactly $\eurm^{\Z}$.
	For our purposes, we only care about the image $\eurm^{\R}\in H^2(\Homeo^+(S^1);\R)$ (resp. $\eurm_b^{\R}\in H^2_b(\Homeo^+(S^1);\R)$) of $\eurm^{\Z}$ (resp. $\eurm_b^{\Z}$) by a change of coefficient $\Z\to \R$.
	
	There is a \emph{rotation quasimorphism} $\rot\in\Qc(\Homeo^+_\Z(\R))$ defined as
	$$
	\rot(h)\defeq \lim_{n \to \infty} \frac{h^n(x)}{n}.
	$$
	The limit exists and it is independent of the reference point $x \in \R$; See \cite{Ghys:circle}.
	\begin{thm}[\protect{\cite[Theorem 2.43]{Cal:sclbook}, \cite{Ghys:circle}}] \label{thm:rot homogeneous qm classical}
		The map $\rot$ is a homogeneous quasimorphism of defect $1$, and $\Qc(\Homeo^+_\Z(\R))$ is one-dimensional, spanned by $\rot$.
		Moreover, for any $h\in \Homeo^+_\Z(\R)$, we have $\rot(h)\in \Z$ if and only if the image $\pi(h)\in \Homeo^+(S^1)$ acts with a fixed point on $S^1$.
	\end{thm}
	
	The rotation quasimorphism is related to the bounded Euler class by the equation $\delta \rot = \pi^*\eurm_b^{\R} \in H^2_b(\Homeo^+_\Z(\R);\R)$.

	Given a group action $\rho: G\to \Homeo^+(S^1)$, the associated bounded (resp. ordinary) Euler class with real coefficient is $\eurm_b(\rho)\defeq\rho^*\eurm_b^{\R}$ (resp. $\eurm(\rho)\defeq\rho^*\eurm^{\R}$). 
	If we have a central extension $1\to \Z\to \hat{G}\stackrel{p}{\to} G\to 1$ and an action $\hat{\rho}:\hat{G}\to \Homeo^+_{\Z}(\R)$ with $\pi\circ \hat{\rho}= \rho \circ p$,	
	then by naturality and the aforementioned equation, we have 
	\begin{equation}\label{eqn: rot and eu}
		\delta \hat{\rho}^*\rot = \hat{\rho}^* \delta \rot =\hat{\rho}^* \pi^* \eurm_b^\R = p^* \rho^* \eurm_b^\R =p^* \eurm_b(\rho).
	\end{equation}
	In addition, the pullback $\hat{\rho}^*\rot\in \Qc(\hat{G})$ of the rotation quasimorphism has $D(\hat{\rho}^*\rot)\le D(\rot)=1$.
	
	\subsection{Graphs of groups} \label{subsec:graph of groups}
	Let $\Gamma$ be a connected graph with vertex set $V$ and edge set $E$. Each edge $e\in E$ is oriented with origin $o(e)$ and terminus $t(e)$. Denote the same edge with opposite orientation by $\bar{e}$, which provides an involution on $E$ satisfying $t(\bar{e})=o(e)$ and $o(\bar{e})=t(e)$. 
	
	A \emph{graph of groups} with underlying graph $\Gamma$ is a collection of \emph{vertex groups} $\{G_v\}_{v\in V}$ and \emph{edge groups} $\{G_e\}_{e\in E}$ with $G_e=G_{\bar{e}}$, as well as injections $t_e: G_e\inj G_{t(e)}$ and $o_e: G_e\inj G_{o(e)}$ satisfying $t_{\bar{e}}=o_e$. Let $(X_v,b_v)$ and $(X_e,b_e)$ be pointed $K(G_v,1)$ and $K(G_e,1)$ spaces respectively, and denote again by $t_e,o_e$ the maps between spaces inducing the given homomorphisms $t_e,o_e$ on fundamental groups. Let $X$ be the space obtained from the disjoint union of $\sqcup_{e\in E} X_e\times[-1,1]$ and $\sqcup_{v\in V} X_v$ by gluing $X_e\times\{1\}$ to $X_{t(e)}$ via $t_e$ and identifying $X_e\times\{s\}$ with $X_{\bar{e}}\times\{-s\}$ for all $s\in[-1,1]$ and $e\in E$. We refer to $X$ as the \emph{standard realization} of the graphs of groups and denote its fundamental group by $G=\mathcal{G}(\Gamma,\{G_v\},\{G_e\})$. This is called the \emph{fundamental group of the graph of groups}. When there is no danger of ambiguity, we will simply refer to $G$ as the graph of groups.
	
	In practice, we will choose a preferred orientation for each unoriented edge $\{e,\bar{e}\}$ by working with $e$ and ignoring $\bar{e}$.
	
	\begin{exmp}\label{exmp: realization}
		\leavevmode
		\begin{enumerate}
		
			\item Let $\Gamma$ be the graph with a single vertex $v$ and an edge $\{e,\bar{e}\}$ connecting $v$ to itself. Let $G_e\cong G_v\cong \Z$. Fix nonzero integers $m,\ell$, and let the edge inclusions $o_e,t_e: G_e\inj G_v$ be given by $o_e(1)=m$ and $t_e(1)=\ell$. Let the edge space $X_e$ and vertex space $X_v$ be circles $S^1_e$ and $S^1_v$ respectively. Then the standard realization $X$ is obtained by gluing the two boundary components of a cylinder $S^1_e\times[-1,1]$ to the circle $S^1_v$ wrapping around $m$ and $\ell$ times respectively. See the left of Figure \ref{fig: BS}. The fundamental group is the \emph{Baumslag--Solitar group} $\BS(m,\ell)$, which has presentation
			$$\BS(m,\ell)=\langle a,t\ |\ a^{m}=ta^{\ell}t^{-1}\rangle.$$
			In general, with the same graph $\Gamma$, for any groups $G_e=C$ and $G_v=A$ together with two inclusions $t_e,o_e:C\inj A$, the corresponding graph of groups is the \emph{HNN extension} $G=A\star_{C}$.
			\begin{figure}
				\labellist
				\small 
				\pinlabel $m$ at -10 70
				\pinlabel $\ell$ at 150 70
				\pinlabel $\Gamma=$ at -60 210
				\pinlabel $X=$ at -60 70
				\pinlabel $v$ at 73 245
				\pinlabel $e$ at 73 177
				\pinlabel $X_v=S^1_v$ at 72 150
				\pinlabel $X_e\times[-1,1]=S^1_e\times[-1,1]$ at 72 -15
				
				\pinlabel $m$ at 235 50
				\pinlabel $\ell$ at 388 50
				\pinlabel $v_1$ at 278 225
				\pinlabel $v_2$ at 353 225
				\pinlabel $e$ at 315 202
				\pinlabel $X_{v_1}$ at 255 150
				\pinlabel $X_{v_2}$ at 377 150
				\pinlabel $X_e\times[-1,1]$ at 315 -15
				\endlabellist
				\centering
				\includegraphics[scale=0.6]{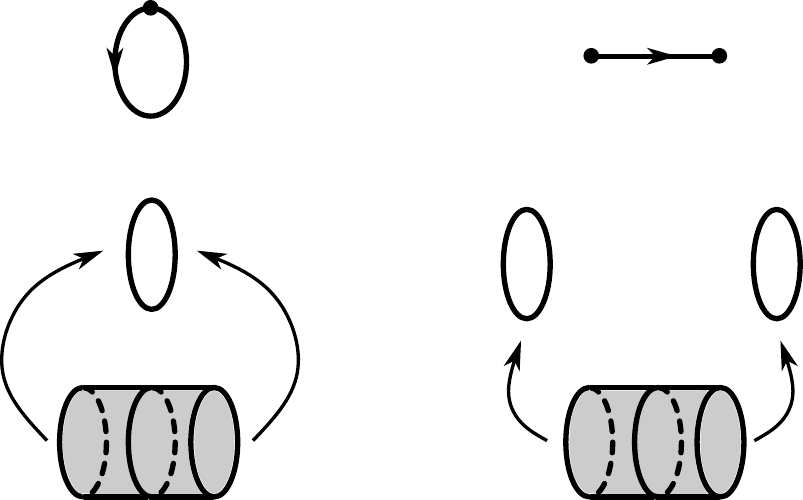}
				\vspace{10pt}
				\caption{On the left we have the underlying graph $\Gamma$ and the standard realization $X$ of $\BS(m,\ell)$; on the right we have the graph $\Gamma$ and the realization $X$ for the amalgam $\Z\star_{\Z}\Z$ associated to $\Z\overset{\times m}{\longrightarrow} \Z$ and $\Z\overset{\times \ell}{\longrightarrow} \Z$.}\label{fig: BS}
			\end{figure}
			\item Similarly, if we let $\Gamma$ be the graph with a single edge $\{e,\bar{e}\}$ connecting two vertices $v_1=o(e)$ and $v_2=t(e)$, the graph of groups associated to two inclusions $t_e: G_e\to G_{v_1}$ and $o_e:G_e\to G_{v_2}$ is the \emph{amalgam} $G_{v_1}\star_{G_e} G_{v_2}$. See the right of Figure \ref{fig: BS} for an example where all edge and vertex groups are $\Z$.
	\end{enumerate}
	\end{exmp}
	
	In general, each connected subgraph of $\Gamma$ gives a graph of groups, whose fundamental group \emph{injects} into $G$, from which we see that each separating edge of $\Gamma$ splits $G$ as an amalgam and each non-separating edge splits $G$ as an HNN extension. Hence $G$ arises as a sequence of amalgamations and HNN extensions.
	
	It is a fundamental result of the Bass--Serre theory that there is a correspondence between groups acting on trees (without inversions) and graphs of groups, where vertex and edge stabilizers correspond to vertex and edge groups respectively. See \cite{Serre} for more details about graphs of groups and their relation to groups acting on trees. 
	
	In the standard realization $X$, the homeomorphic images of $X_e\times\{0\}\cong X_e$ and $X_v$ are called an \emph{edge space} and a \emph{vertex space} respectively. The image of $X_v\sqcup (\sqcup_{e: t(e)=v}X_{e}\times [0,1) )$ deformation retracts to $X_v$. We refer to its completion $N(X_v)$ as the \emph{thickened vertex space}; See Figure \ref{fig: thickenedvertsp}.
	
	\begin{figure}
		\labellist
		\small 
		\pinlabel $m$ at 5 270
		\pinlabel $\ell$ at 165 270
		\pinlabel $X_v$ at 110 325
		\pinlabel $X=$ at -50 270
		
		\pinlabel $N(X_v)=$ at -50 70
		\endlabellist
		\centering
		\includegraphics[scale=0.6]{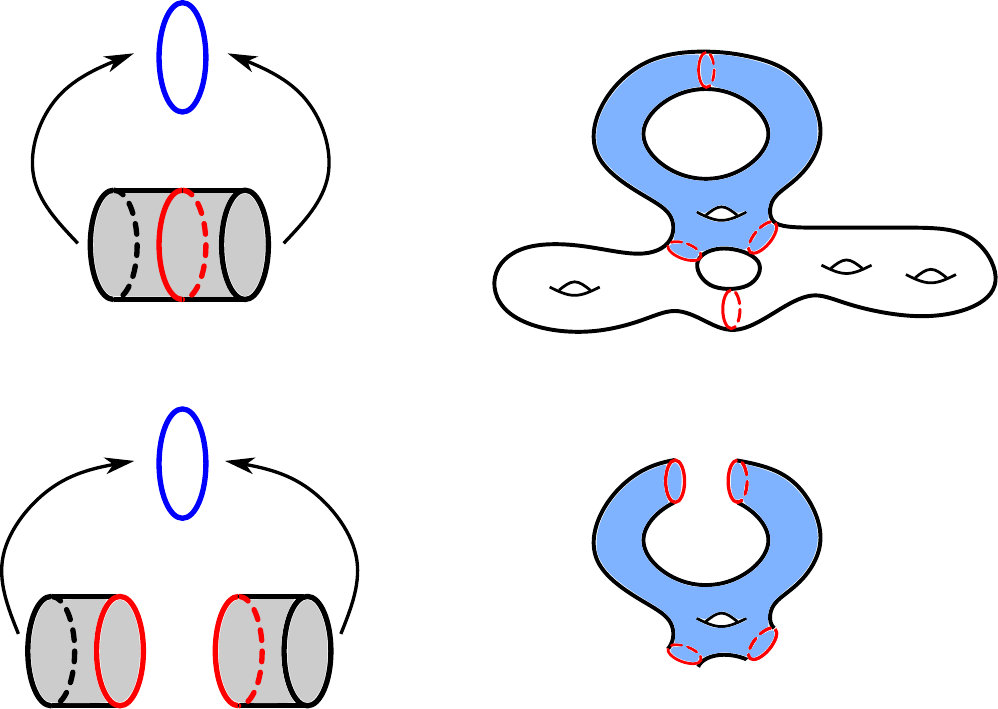}
		\vspace{10pt}
		\caption{On the bottom left we depict the thickened vertex space $N(X_v)$ for the unique vertex space $X_v$ in the realization $X$ (upper left) of $\BS(m,\ell)$; on the right depict a thickened vertex space in a more general situation, where the red circles are the edge spaces that we cut along.}\label{fig: thickenedvertsp}
	\end{figure}
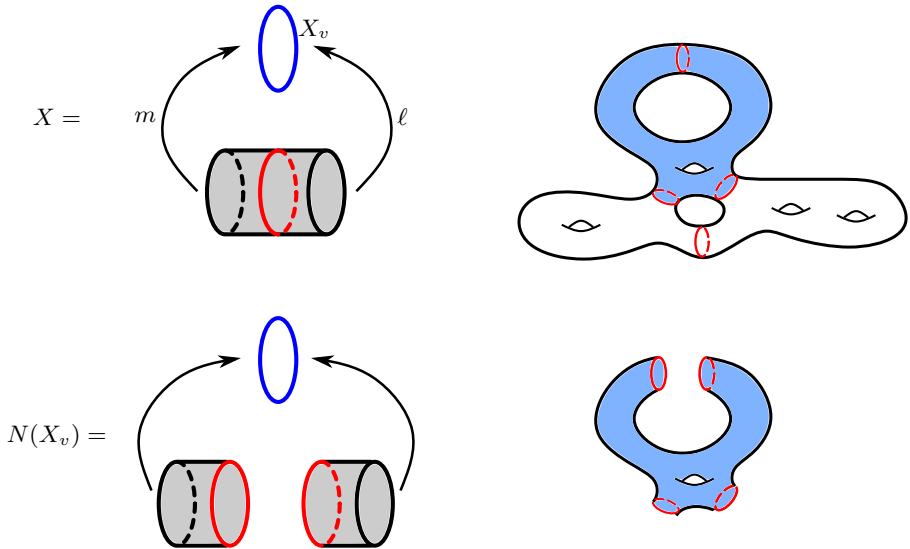
	
	Free homotopy classes of loops in $X$ fall into two types. \emph{Elliptic} loops are those admitting a representative supported in some vertex space, and such a representative is called a \emph{tight} elliptic loop. Loops of the other type are called \emph{hyperbolic}. We can deform any hyperbolic loop $\gamma$ so that, for each $s\in (-1,1)$ and $e\in X_e$, $\gamma$ is either disjoint from $X_e\times\{s\}$ or intersects it only at $\{b_e\}\times\{s\}$ transversely. For such a representative, the edge spaces cut $\gamma$ into finitely many arcs, each supported in some thickened vertex space $N(X_v)$. The image of each arc $\alpha$ in $N(X_v)$ under the deformation retraction $N(X_v)\to X_v$ becomes a based loop in $X_v$ and thus represents an element $w(\alpha)\in G_v$. If some arc $\alpha$ enters and leaves $X_v$ via the same end of an edge $e$ with $t(e)=v$ and $w(\alpha)\in t_e(G_e)$, then we say $\gamma$ trivially backtracks at $\alpha$ and can push $\alpha$ off $X_v$ to further simplify $\gamma$; See Figure \ref{fig: backtrack}. After finitely many such simplifications, we may assume that $\gamma$ does not trivially backtrack. Refer to such a representative as a \emph{tight} hyperbolic loop, which exists in each hyperbolic homotopy class by the procedure above.
	
	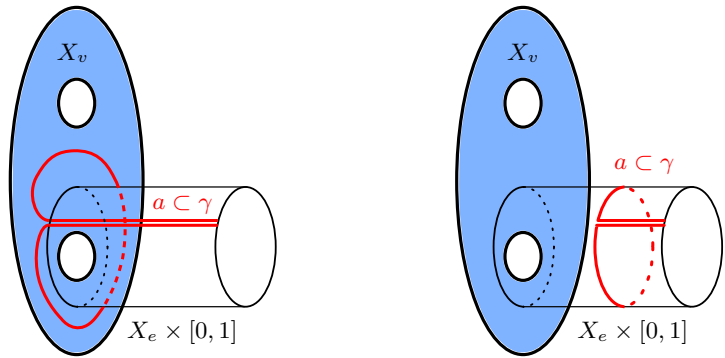
\begin{figure} 
		\centering
		
		\labellist
		\small \hair 2pt
		\pinlabel $X_v$ at 40 190
		\pinlabel \textcolor{red}{$a\subset\gamma$} at 108 95
		\pinlabel $X_e\times[0,1]$ at 108 15
		\pinlabel $X_v$ at 320 190
		\pinlabel \textcolor{red}{$a\subset\gamma$} at 395 120
		\pinlabel $X_e\times[0,1]$ at 388 15
		\endlabellist
		
		\includegraphics[scale=0.6]{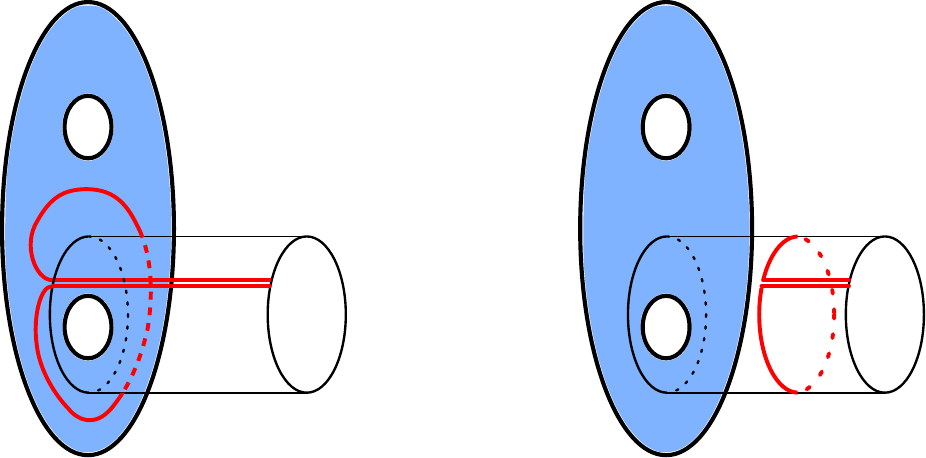}
		\caption{A loop $\gamma$ trivially backtracks at an arc $a$ supported in the thickened vertex space $N(X_v)$ as shown on the left. It can be pushed off the vertex space $X_v$ by a homotopy as shown on the right.}
		\label{fig: backtrack}
	\end{figure}
	
	On the group theoretic side, an element $g\in G$ is elliptic (resp. hyperbolic) if it is represented by an elliptic (resp. hyperbolic) loop in $X$, which we usually choose to be tight. Then an element is elliptic if and only if it is conjugate into some vertex group. 
	
	

	\section{Scl of hyperbolic elements} \label{sec:scl on hyperbolic elts}

	\subsection{Surfaces in graphs of groups}\label{subsec: surfs in graphs of groups}
	In this subsection, we investigate surfaces in graphs of groups and their normal forms following \cite{Chen:sclBS}, which will be used to estimate (relative) scl.
	
	Let $G$ be a graph of groups. With the setup in Subsection \ref{subsec:graph of groups}, let $X$ be the standard realization of $G$. Let $\ug=\{g_i,i\in I\}$ be a finite collection of \emph{infinite-order} elements indexed by $I$ and let $c=\sum r_i g_i$ be a rational chain with $r_i\in \Q_{>0}$.  Let $\uga=\{\gamma_i,i\in I\}$ be tight loops representing elements in $\ug$. Recall that the edge spaces cut the hyperbolic tight loops into arcs, and denote by $A_v$ the collection of arcs supported in the thickened vertex space $N(X_v)$.
	
	Let $S$ be any aspherical (monotone) admissible surface for $c$. Put $S$ in general position so that it is transverse to all edge spaces. Then the preimage $F$ of the union of edge spaces is a collection of disjoint embedded proper arcs and loops. Up to homotopy and compression that simplifies $S$, every loop in $F$ represents a nontrivial conjugacy class in some edge group.
	
	Now cut $S$ along $F$ into subsurfaces. Then each component $\Sigma$ is a surface (possibly with corner) supported in some thickened vertex space $N(X_v)$. Boundary components of $\Sigma$ fall into two types (see Figure \ref{fig: polygonal boundary}): 
	\begin{enumerate}
		\item \emph{Loop boundary}: these are boundary components containing no corners. Each such boundary is either a loop in $F$ or a loop in $\partial S$ that winds around an elliptic tight loop in $\uga$;
		\item \emph{Polygonal boundary}: these are boundary components containing corners. Each such boundary is necessarily divided into segments by the corners, such that the segments alternate between arcs in $A_v$ and proper arcs in $F$ (called \emph{turns}).
	\end{enumerate}
	\begin{figure}
		\labellist
		\small 
		\pinlabel $\beta_1$ at -10 70
		\pinlabel $\Sigma_1$ at 120 60
		\pinlabel $\beta_2$ at 210 25
		\pinlabel $\beta_3$ at 210 105
		\pinlabel $\beta_4$ at 400 35
		\pinlabel $\Sigma_2$ at 330 65
		
		\pinlabel $a_{i_1}$ at 330 15
		\pinlabel $a_{i_2}$ at 365 80
		\pinlabel $a_{i_3}$ at 285 80
		\endlabellist

		\centering
		\includegraphics[scale=0.6]{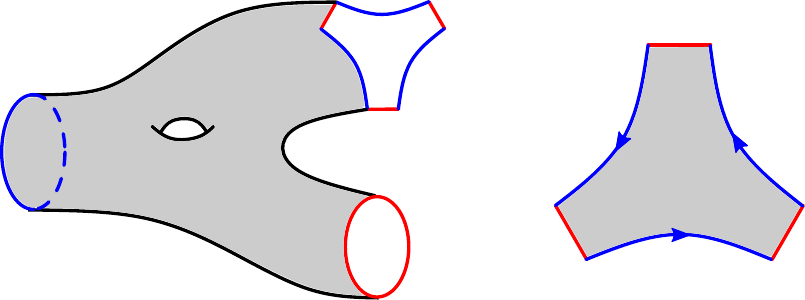}
		\caption{Here are two possible components $\Sigma_1$ and $\Sigma_2$ of $S$ supported in some $N(X_v)$, where the blue parts are supported in $\uga$ and the red is part of $F$. The component $\Sigma_1$ has three boundary components: $\beta_1$ winds around an elliptic loop in $\uga$, the loop $\beta_2$ lies in $F$, and $\beta_3$ is a polygonal boundary. The component $\Sigma_2$ is a disk with polygonal boundary $\beta_4$, on which we have arcs $a_{i_1},a_{i_2},a_{i_3}$ in cyclic order.}\label{fig: polygonal boundary}
	\end{figure}
	
	Note that any component $\Sigma$ with positive Euler characteristic must be a disk with polygonal boundary since elements in $\ug$ have infinite-order and loops in $F$ are nontrivial in edge spaces.
	
	For $\Sigma$ as above, we say a turn has \emph{type} $(a_1,w,a_2)$ for some $a_1,a_2\in A_v$ and $w\in G_e$ for some $e$ adjacent to $v$ if it travels from $a_1$ to $a_2$ as a based loop supported on $X_e$ representing $w$, referred to as the \emph{winding number} of the turn. We say two turn types $(a_1,w,a_2)$ and $(a'_1,w',a'_2)$ are \emph{paired} if they can be glued together, i.e. if there is some edge $e$ with $w,w'\in G_e$, $a_1,a_2\in A_{o(e)}$ and $a'_1,a'_2\in A_{t(e)}$ such that $w^{-1}=w'$ and $a_1$ (resp. $a'_1$) is followed by $a'_2$ (resp. $a_2$) on $\uga$. See Figure \ref{fig: pair}.
	
	\begin{figure}
		\labellist
		\small \hair 2pt
		
		\pinlabel $a_1$ at 38 40
		\pinlabel $a_2$ at 38 110
		\pinlabel $a'_1$ at 93 110
		\pinlabel $a'_2$ at 93 40
		\pinlabel $w$ at 48 75
		\pinlabel $w^{-1}$ at 93 75
		
		\pinlabel $a_1$ at 238 20
		\pinlabel $a_2$ at 238 130
		\pinlabel $a'_1$ at 333 130
		\pinlabel $a'_2$ at 333 20
		\pinlabel \textcolor{blue}{$\gamma_j$} at 328 95
		\pinlabel \textcolor{blue}{$\gamma_i$} at 328 55
		
		\pinlabel $X_e$ at 288 10
		
		\endlabellist
		\centering
		\includegraphics[scale=0.7]{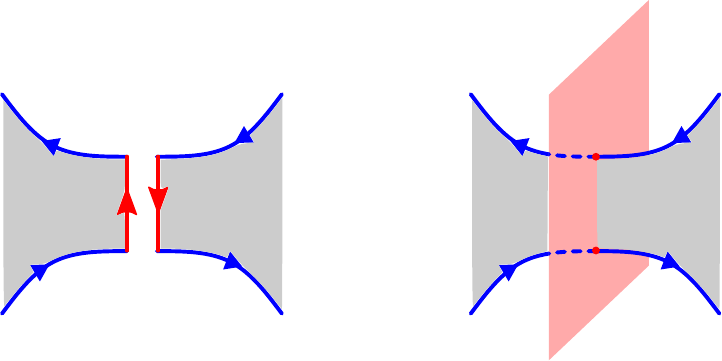}
		\caption{Two paired turns on an edge space $X_e$, with arcs $a_1,a'_2\subset\gamma_i$ and $a'_1,a_2\subset\gamma_j$}\label{fig: pair}
	\end{figure}
	
	For each vertex $v$, let $S_v$ be the union of components $\Sigma$ supported in $N(X_v)$ obtained from cutting $S$ along $F$. Note that there is an obvious pairing of turns in $\sqcup S_v$: two turns are paired if they are glued together in $S$. Let $t_{(a_1,w,a_2)}$ be the number of turns of type $(a_1,w,a_2)$ that appear in $\sqcup S_v$ \emph{divided by} $k$ if $S$ is a degree $k$ admissible surface. Then we obviously have the following \newline
	\textbf{Gluing condition: }
	\begin{equation}\label{eqn: gluding condition}
	 t_{(a_1,w,a_2)}=t_{(a'_1,w',a'_2)} \text{ for any paired turn types } (a_1,w,a_2) \text{ and } (a'_1,w',a'_2).
	\end{equation}
	
	Let $i(\alpha)\in I$ be the index such that the arc $\alpha$ lies on $\gamma_{i(\alpha)}$. Recall that $r_i$ is the coefficient of $\gamma_i$ in the chain $c$. Then we also have the following\newline
	\textbf{Normalizing condition: } 
	\begin{equation}\label{eqn: normalizing condition}
	\sum_{a_2,w} t_{(a_1,w,a_2)}=r_{i(a_1)} \text{ for any } a_1 \text{ and } \sum_{a_1,w} t_{(a_1,w,a_2)}=r_{i(a_2)} \text{ for any } a_2.
	\end{equation}
	
	\begin{defn}\label{def: normal form}
		We say an admissible surface $S$ decomposed in the above way satisfying the gluing and normalizing conditions is in its \emph{normal form}.
	\end{defn}
	
	The discussion above shows that any (monotone) admissible surface for the chain $c$ can be put in normal form after ignoring sphere components, applying homotopy and compression, during which $-\chi^-(S)$ does not increase. 
	
	Note that if $S$ is already in normal form, then the sum of $\chi(\Sigma)$ over all components differs from $\chi(S)$ by the number of proper arcs in $F$, which is half of the total number of turns. Thus in general, for any $S$ we have 
	
	\begin{equation}\label{eqn: Euler char computation}
	\frac{-\chi^-(S)}{k}\ge \frac{1}{2}\sum_{(a_1,w,a_2)} t_{(a_1,w,a_2)}+\sum\frac{ -\chi (\Sigma)}{k}=\frac{1}{2}\sum_i r_i|\gamma_i|+\sum\frac{-\chi (\Sigma)}{k}
	\end{equation}
	where the last equality follows from the normalizing condition (\ref{eqn: normalizing condition}) and $|\gamma_i|$ denotes the number of arcs on the hyperbolic tight loop $\gamma_i$ and is $0$ for elliptic $\gamma_i$. Since the quantity $\frac{1}{2}\sum_i r_i|\gamma_i|$ is purely determined by the chain $c$, the key is to estimate $\sum\frac{-\chi (\Sigma)}{k}$.
	
	This for example yields a simple estimate of $scl_G$ of chains supported in vertex groups in terms of relative scl.
	\begin{lemma}\label{lemma: simple estimate of scl in vertex groups}
		For any vertex $v$ and any chain $c\in C_1(G_v)$, we have $$\scl_G(c)\ge\scl_{(G_v,\{G_e\}_{t(e)=v})}(c).$$
	\end{lemma}
	\begin{proof}
		We assume $c$ to be null-homologous in $G$, since otherwise the result is trivially true. By continuity, we may further assume $c$ to be a rational chain after an arbitrarily small change. Represent $c$ by elliptic tight loops supported in $X_v$. Suppose $f:S\to X$ is an admissible surface for $c$ of degree $k$ in normal form. Let $S_v$ be the disjoint union of components supported in $N(X_v)$. Since there are no hyperbolic loops in our chain, each component of $S_v$ has non-positive Euler characteristic. So equation (\ref{eqn: Euler char computation}) implies
		$$\frac{-\chi^-(S)}{k}\ge\frac{-\chi(S_v)}{k}=\frac{-\chi^-(S_v)}{k}.$$
		Note that $S_v$ is admissible of degree $k$ for $c$ in $G_v$ relative to the nearby edge groups. It follows that $-\chi^-(S_v)/k\ge 2\cdot \scl_{(G_v,\{G_e\}_{t(e)=v})}(c)$. Combining this with the inequality above, the conclusion follows from Lemma \ref{lemma: rel admissible} since $S$ is arbitrary.
	\end{proof}
	
	\subsection{Lower bounds from linear programming duality}\label{subsec: bounds from LP duality}
	We introduce a general strategy to obtain lower bounds of scl in graphs of groups relative to vertex groups using the idea of linear programming duality. This has been used by the first author in the special case of free products to obtain uniform sharp lower bounds of scl \cite{Chen:sclfpgap} and in the case of Baumslag--Solitar groups to compute scl of certain families of chains \cite{Chen:sclBS}.
	
	Consider a rational chain $c=\sum r_i g_i$ with $r_i\in \Q_{>0}$ and $\ug=\{g_i,i\in I\}$ consisting of finitely many \emph{hyperbolic} elements represented by tight loops $\uga=\{\gamma_i,i\in I\}$. With notation as in the previous subsection, for each turn type $(a_1,w,a_2)$, we assign a \emph{non-negative cost} $\$(a_1,w,a_2)\ge0$. 
	
	Suppose $S$ is an admissible surface of degree $k$ for $c$ relative to vertex groups. Then $S$ is by definition admissible (in the absolute sense) of degree $k$ for some rational chain $c'=c+c_{ell}$ where $c_{ell}$ is a rational chain of elliptic elements, each of which can be assumed to have infinite order. Hence the normal form discussed in the previous subsection applies to $S$.
	
	The assignment above induces by linearity a \emph{non-negative} cost for each component of $\Sigma$ in the decomposition of $S$ in normal form. To be more specific, recall that there are two kinds of boundary components of $\Sigma$, loop boundaries and polygonal boundaries. The induced cost of each polygonal boundary is the sum of costs of its turns, and the cost of $\Sigma$ is the sum of costs of all polygonal boundaries.
	
	
	\begin{lemma}\label{lemma: duality estimate}
		Let $S$ be any relative admissible surface for $c$ of degree $k$ in normal form. With the $t_{(a_1,w,a_2)}$ notation (normalized number of turns) in the estimate (\ref{eqn: Euler char computation}), if every disk component $\Sigma$ in the normal form of $S$ has cost at least $1$, then the normalized total cost
		$$\sum_{(a_1,w,a_2)} t_{(a_1,w,a_2)}\$(a_1,w,a_2)\ge \frac{1}{2}\sum r_i|\gamma_i|+\frac{\chi^-(S)}{k}.$$
	\end{lemma}
	\begin{proof}
		By (\ref{eqn: Euler char computation}), it suffices to prove
		$$\sum\frac{\chi (\Sigma)}{k}\le \sum_{(a_1,w,a_2)} t_{(a_1,w,a_2)}\$(a_1,w,a_2).$$
		Note that, for each component $\Sigma$ in the normal form of $S$, either $\chi(\Sigma)\le 0$ which does not exceed its cost, or $\Sigma$ is a disk component and $\chi(\Sigma)= 1$ which is also no more than its cost by our assumption. The desired estimate follows by summing up these inequalities and dividing by $k$.
	\end{proof}
	
	In light of Lemma \ref{lemma: duality estimate}, to get lower bounds of scl relative to vertex groups, the strategy is to come up with suitable cost assignments $\$(a_1,w,a_2)$ such that
	\begin{enumerate}
		\item every possible disk component has cost at least $1$; and
		\item one can use the gluing condition (\ref{eqn: gluding condition}) and normalizing condition (\ref{eqn: normalizing condition}) to bound the quantity $\sum_{(a_1,w,a_2)} t_{(a_1,w,a_2)}\$(a_1,w,a_2)$ from above by a constant.
	\end{enumerate}
	
	
	\subsection{Uniform lower bounds}\label{subsec: uniform bounds}
	Now we use the duality method above to prove sharp uniform lower bounds of scl in graphs of groups relative to vertex groups. The results are subject to some local conditions introduced as follows.
	
	\begin{defn} \label{defn: n-rft}
		Let $H$ be a subgroup of $G$. For $2\le k< \infty$, an element $g\in G\setminus H$ \emph{has order $k$ rel $H$} if for some $h_i\in H$ we have
		\begin{equation}\label{eqn: relative torsion}
		gh_1 \ldots gh_k=id.
		\end{equation}
		For $3\le n\le \infty$, we say $g$ is has order $\ge n$ rel $H$ if $g\in G\setminus H$ does not have order $k$ rel $H$ for all $2\le k<n$. Similarly, we say the subgroup $H$ is \emph{$n$-relatively torsion-free} (\emph{$n$-RTF}) in $G$ if each element $g\in G\setminus H$ has order $\ge n$ rel $H$.
		For clarity, sometimes we say the pair $(G,H)$ is $n$-RTF.
	\end{defn}
	By definition, $m$-RTF implies $n$-RTF whenever $m\ge n$.
	
	\begin{exmp}
		If $H$ is normal in $G$, then $g$ has order $k$ rel $H$ if and only if its image in $G/H$ is a $k$-torsion. In this case, $H$ is $n$-RTF if and only if $G/H$ contains no $k$-torsion for all $k< n$, and in particular, $H$ is $\infty$-RTF if and only if $G/H$ is torsion-free. Concretely, the subgroup $H=6\Z$ in $G=\Z$ is not $n$-RTF for all $n\ge3$ since $z^3$ has order $2$ rel $H$, where $z$ is a generator of $\Z$. More generally, the subgroup $H=m\Z$ is $p_m$-RTF if $m$ is odd and $p_m$ is the smallest prime factor of $m\in\Z_+$.
		
		For a less trivial example, let $S$ be an orientable closed surface of positive genus and let $g\in \pi_1(S)$ be an element represented by a simple closed curve. Then the cyclic subgroup $\langle g\rangle$ is $\infty$-RTF in $\pi_1(S)$. One can see this from either Lemma \ref{lemma: left relatively convex implies RTF} or Lemma \ref{lemma: inherits RTF in graphs of groups} below.
	\end{exmp}
	
	The equation (\ref{eqn: relative torsion}) can be rewritten as
	\begin{equation}\label{eqn: relative torsion form 2}
	g\cdot \tilde{h}_1g\tilde{h}_1^{-1}\cdot\ldots \cdot\tilde{h}_{k-1}g\tilde{h}_{k-1}^{-1}=\tilde{h}_k^{-1},
	\end{equation}
	where $\tilde{h}_i\defeq h_1\ldots h_i$. This is closely related to the notion of generalized $k$-torsion.
	\begin{defn}
		For $k\ge 2$, an element $g\neq id\in G$ is a \emph{generalized $k$-torsion} if
		$$g_1gg_1^{-1}\ldots g_kgg_k^{-1}=id$$
		for some $g_i\in G$.
	\end{defn}
	If equation (\ref{eqn: relative torsion}) holds with $\tilde{h}_k=h_1h_2\ldots h_k=id$, then $g$ is a generalized $k$-torsion.
	
	It is observed in \cite[Theorem 2.4]{gentorsion} that a generalized $k$-torsion cannot have scl exceeding $1/2-1/k$ for a reason similar to Proposition \ref{prop: rel gap implies $n$-RTF} below. On the other hand, it is well known and easy to note that the existence of any generalized torsion is an obstruction for a group $G$ to be bi-orderable, i.e. to admit a total order on $G$ that is invariant under left and right multiplications.
	
	More properties of the $n$-RTF condition can be found in Subsection \ref{subsec: n-RTF}. Now we turn to the relation between the $n$-RTF condition and scl estimates.
	
	The $n$-RTF condition is closely related to lower bounds of relative scl.
	
	\begin{prop}\label{prop: rel gap implies $n$-RTF}
		Let $H$ be a subgroup of $G$. If $$\scl_{(G,H)}(g)\ge \frac{1}{2}-\frac{1}{2n},$$ then $g\in G$ has order $\ge n$ rel $H$.
	\end{prop}
	\begin{proof}
		Suppose equation (\ref{eqn: relative torsion form 2}) holds for some $k\ge2$. Then this gives rise to an admissible surface $S$ in $G$ for $g$ of degree $k$ relative to $H$, where $S$ is a sphere with $k+1$ punctures: $k$ of them each wraps around $g$ once, and the other maps to $\tilde{h}_k$. This implies 
		$$\frac{1}{2}-\frac{1}{2n}\le \scl_{(G,H)}(g)\le \frac{-\chi(S)}{2k}=\frac{1}{2}-\frac{1}{2k},$$ and thus $k\ge n$.
	\end{proof}
	
	Conversely, the $n$-RTF condition implies a lower bound for relative scl in the case of graphs of groups.
	\begin{thm}\label{thm: $n$-RTF gap, strong version}
		Let $G=\mathcal{G}(\Gamma,\{G_v\},\{G_e\})$ be a graph of groups. Let $\gamma$ be a tight loop cut into arcs $a_1,\ldots, a_L$ by the edge spaces, where $a_i$ is supported in a thickened vertex space $N(X_{v_i})$ and $v_1,e_1,\ldots, v_L,e_L$ form a loop in $\Gamma$ with $o(e_i)=v_i$ and $t(e_i)=v_{i+1}$, indices taken mod $L$. Suppose for some $n\ge3$, whenever $e_{i-1}=\bar{e}_i$, the winding number $w(a_i)\in G_{v_i}$ of $a_i$ has order $\ge n$ rel $o_{e_i}(G_{e_i})$. Then 
		$$\scl_{(G,\{G_v\})}(g)\ge \frac{1}{2}-\frac{1}{n}.$$
	\end{thm}
	\begin{proof}
		Let $S$ be a (monotone) relative admissible surface for $\gamma$ of degree $k$. Put $S$ in its normal form with components $\Sigma$. We follow the strategy and notation in Subsection \ref{subsec: bounds from LP duality} and assign costs in a way that does not depend on the winding numbers of turns. That is, for any $1\le i,j\le L$, the cost $\$(a_i,w,a_j)=\$_{ij}$ where
		$$\$_{ij}\defeq\left\{ \begin{array}{ll}
		1-\frac{1}{n}, & \text{if } i<j\\
		\frac{1}{n}, & \text{if } i\ge j.
		\end{array} \right. $$
		Let $t_{ij}=\sum_w t_{(a_i,w,a_j)}$ be the (normalized) total number of turns of the form $(a_i,w,a_j)$.
		
		For any disk component $\Sigma$, let $\sigma_1=a_{i_1},\ldots, \sigma_s=a_{i_s}$ be the arcs of $\gamma$ on the polygonal boundary of $\Sigma$ in cyclic order; See the disk component in Figure \ref{fig: polygonal boundary}. There are two cases:
		\begin{enumerate}
			\item $\sigma_1=\cdots=\sigma_s=a_i$ for some $i$. Then we necessarily have $e_{i-1}=\bar{e}_i$, and $a_i\in G_{v_i}\backslash o_{e_i}(G_{e_i})$ since $\gamma$ is tight. For each $1\le j\le s$, let $w_j\in o_{e_i}(G_{e_i})$ be the winding number of the turn from $\sigma_j$ to $\sigma_{j+1}$, where the index is taken mod $s$. Since $\Sigma$ is a disk, we have $w(a_i)w_1\cdots w(a_i)w_s=id\in G_{v_i}$. Then we must have $s\ge n$ since $w(a_i)\in G_{v_i}$ has order $\ge n$ rel $o_{e_i}(G_{e_i})$ by assumption, and thus the cost
			$$\$(\Sigma)=s\cdot \$_{ii}=\frac{s}{n}\ge 1.$$

			\item Some $\sigma_j\neq \sigma_{j'}$. Then there are $1\le m\neq M\le s$ such that $i_{m}<i_{m+1}$ and $i_{M}>i_{M+1}$, where subscripts are interpreted mod $s$. Hence the cost
			$$\$(\Sigma)\ge \$_{m,m+1}+\$_{M,M+1}=\left(1-\frac{1}{n}\right)+\left(\frac{1}{n}\right)= 1.$$
		\end{enumerate}
		
		In summary, we always have $\$(\Sigma)\ge1$ for any disk component $\Sigma$. Hence by Lemma \ref{lemma: duality estimate}, we have
		
		$$\frac{-\chi^-(S)}{2k}\ge \frac{L}{4}-\frac{1}{2}\sum_{ij} \$_{ij}t_{ij}$$ since $|\gamma|=L$.
		
		On the other hand, for any $1\le i,j\le L$, we have $t_{ij}=t_{j-1,i+1}$ by the gluing condition (\ref{eqn: gluding condition}) and $\sum_i t_{ij}=\sum_{j} t_{ij}=1$ by the normalizing condition (\ref{eqn: normalizing condition}), indices taken mod $L$. We also have $t_{i,i+1}=0$ since $\gamma$ is tight and intersects edge spaces transversely. Thus
		
		$$\sum_{i,j} \$_{ij}t_{ij}=\sum_{i,j} \frac{1}{n}t_{ij}+\left(1-\frac{2}{n}\right)\sum_{i<j} t_{ij}=\frac{L}{n}+\left(1-\frac{2}{n}\right)\sum_{i<j} t_{ij}.$$
		and
		\begin{eqnarray*}
			2\sum_{i<j} t_{ij}&=&\sum_{1\le i<j\le L} t_{ij}+\sum_{1\le i<j\le L} t_{j-1,i+1}\\
			&=& \sum_{\substack{1\le i\le L-1 \\ 2\le j\le L}} t_{ij}+\sum_{1\le i\le L-1} t_{i,i+1}\\
			&=&\left[\sum_{1\le i,j\le L}t_{ij}-\sum_i t_{i1}-\sum_j t_{Lj} \right]+\left[0\right]\\
			&=&L-2,\\
		\end{eqnarray*}
		where the first two equalities can be visualized in Figure \ref{fig: staircasesum}.
		
		\begin{figure}
			\labellist
			\small \hair 2pt
			
			\pinlabel $\times 2$ at 65 185
			\pinlabel $=$ at 105 165
			\pinlabel $+$ at 225 165

			\pinlabel $=$ at 105 45
			\pinlabel $+$ at 225 45
			
			\endlabellist
			\centering
			\includegraphics[scale=0.6]{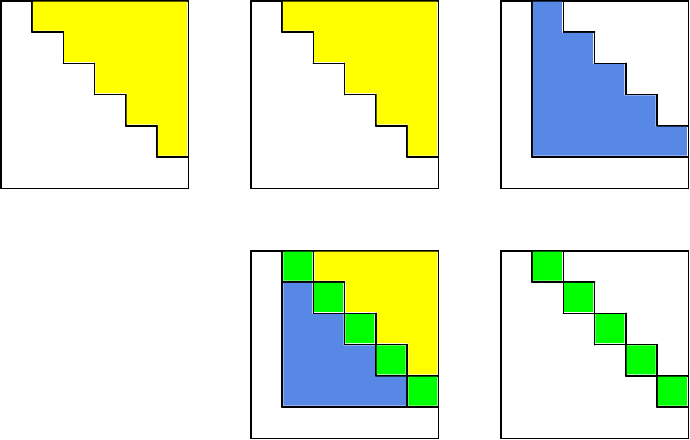}
			\caption{Visualization of the summation in the case $L=6$, where the first equality uses the gluing condition $t_{ij}=t_{j-1,i+1}$}\label{fig: staircasesum}
		\end{figure}
		
		Putting the equations above together, we have 
		$$\sum_{ij} \$_{ij}t_{ij}=\frac{L}{n}+\left(1-\frac{2}{n}\right)\frac{L-2}{2}=\frac{L}{2}-\left(1-\frac{2}{n}\right),$$
		and
		$$\frac{-\chi^-(S)}{2k}\ge \frac{L}{4}-\frac{1}{2}\sum_{ij} \$_{ij}t_{ij}=\frac{1}{2}-\frac{1}{n}$$
		for any relative admissible surface $S$. Thus the conclusion follows from Lemma \ref{lemma: rel admissible}.
	\end{proof}
	
	\begin{thm}\label{thm: $n$-RTF gap, weak version}
		Let $G=\mathcal{G}(\Gamma,\{G_v\},\{G_e\})$ be a graph of groups. If for some $3\le n\le \infty$ the inclusion of each edge group into vertex group is $n$-RTF, then 
		$$\scl_G (g)\ge \scl_{(G,\{G_v\})}(g)\ge \frac{1}{2}-\frac{1}{n}$$
		for any hyperbolic element $g\in G$.
		
		More generally, under the same assumptions, the same inequality also holds for scl relative to some collections of subgroups in $G$ below other than $\{G_v\}$. Let $\{\Gamma_\lambda\}_{\lambda\in\Lambda}$ be a collection of pairwise disjoint connected subgraphs of $\Gamma$, each of which defines (up to conjugation) a subgroup $G_\lambda$ of $G$ as the graph of groups associated to $\Gamma_\lambda$. Then under the same assumptions above on edge groups, unless $g\in G$ is conjugate into some $G_\lambda$, we have
		$$\scl_G (g)\ge \scl_{(G,\{G_\lambda\})}(g)\ge \frac{1}{2}-\frac{1}{n}.$$
		Hence $(G,\{G_\lambda\}_{\lambda\in\Lambda})$ has a strong relative spectral gap $\frac{1}{2}-\frac{1}{n}$.		
		
%
	\end{thm}
	\begin{proof}
		Each hyperbolic element $g$ is represented by a tight loop $\gamma$ satisfying the assumptions of Theorem \ref{thm: $n$-RTF gap, strong version} since the inclusions of edge groups are $n$-RTF. This implies the special case where $\{\Gamma_\lambda\}_{\lambda\in\Lambda}$ is the set of vertices of $\Gamma$. 
		
		To see the general case, by possibly adding some subgraphs each consisting of a single vertex, we assume each vertex is contained in some $\Gamma_\lambda$. By collapsing each $\Gamma_\lambda$ to a single vertex, we obtain a new splitting of $G$ as a graph of groups where $\{G_\lambda\}$ is the new collection of vertex groups. Each new edge group is some $G_e$ for some edge $e$ of $\Gamma$ connecting two $\Gamma_\lambda$'s. Let $\Gamma_\lambda$ be the subgraph containing the vertex $v=t(e)$, then $(G_\lambda,G_v)$ is $n$-RTF by Lemma \ref{lemma: inherits RTF in graphs of groups} in the next subsection and $(G_v,G_e)$ is $n$-RTF by assumption. Thus $(G_\lambda,G_e)$ is also $n$-RTF by Lemma \ref{lemma: subgroup inherits RTF}. Therefore the edge group inclusions in this new splitting also satisfy the $n$-RTF condition, so the general case follows from the special case proved above.
	\end{proof}
	
	As we will see below in Lemma \ref{lemma: left relatively convex implies RTF}, left relatively convexity implies $\infty$-RTF. Thus in the case of an amalgam $G=A\star_C B$, Theorem \ref{thm: $n$-RTF gap, weak version} implies \cite[Theorem 6.3]{Heuer}, which is the main input to obtain gap $1/2$ in all right-angled Artin groups in \cite{Heuer}. We will discuss similar applications in Section \ref{sec:spectral gap graph products}.
	
	For the moment, let us consider the case of Baumslag--Solitar groups.
	\begin{cor}
		Let $\BS(m,\ell)\defeq \langle a,t\ |\ a^m=ta^\ell t^{-1}\rangle$ be the Baumslag--Solitar group, where $|m|,|\ell|\ge 2$. Let $p_m$ and $p_\ell$ be the smallest prime factors of $|m|$ and $|\ell|$ respectively. Then $\BS(m,\ell)$ has strong spectral gap $1/2-1/\min(p_m,p_\ell)$ relative to the subgroup $\langle a\rangle$ if $m,l$ are both odd. This estimate is sharp since for $g=a^{m/p_m}ta^{\ell/p_\ell}t^{-1}a^{-m/p_m}ta^{-\ell/p_\ell}t^{-1}$ we have
		$$\scl_{\BS(m,\ell)}(g)=\scl_{(\BS(m,\ell),\langle a \rangle)}(g)=1/2-1/\min(p_m,p_\ell).$$ 
	\end{cor}
	\begin{proof}
		Let $n=\min(p_m,p_\ell)$. The Baumslag--Solitar group $\BS(m,\ell)$ is the HNN extension associated to the inclusions $\Z\overset{\times m}{\to} \Z$ and $\Z\overset{\times \ell}{\to}\Z$, which are both $n$-RTF. Thus the strong relative spectral gap follows from Theorem \ref{thm: $n$-RTF gap, weak version}. The example achieving the lower bound follows from \cite[Corollary 3.12]{Chen:sclBS}, \cite[Proposition 5.6]{Chen:sclfp} and \cite[Proposition 2.11]{Chen:sclBS}.
	\end{proof}
	
	If at least one of $m$ and $\ell$ is even, the word $g$ above has scl value $0$, and thus one cannot have a strong relative spectral gap. However, we do have a (relative) spectral gap $1/12$ by \cite[Theorem 7.8]{CFL16}, which is sharp for example when $p_m=2$ and $p_\ell=3$. When $p_m=2$ and $p_\ell\ge 3$, the smallest known positive scl in $\BS(m,\ell)$ is $1/4-1/2p_\ell$ achieved by the word $a^{m/p_m}ta^{\ell/p_\ell}t^{-1}$.
	\begin{quest}
		If $p_m=2$ and $p_\ell\ge 3$, does $\BS(m,\ell)$ have (relative) spectral gap $1/4-1/2p_\ell$?
	\end{quest}
	
	\subsection{The $n$-RTF condition}\label{subsec: n-RTF}
	
	The goal of this subsection is to investigate the $n$-RTF condition that plays an important role in Theorem \ref{thm: $n$-RTF gap, strong version} and Theorem \ref{thm: $n$-RTF gap, weak version}.
	
	Let us start with some basic properties.
	\begin{lemma}\label{lemma: subgroup inherits RTF}
		Suppose we have groups $K\le H\le G$. 
		\begin{enumerate}
			\item If $(G,K)$ is $n$-RTF, then so is $(H,K)$;\label{item: subgroup inherits RTF (1)}
			\item If $g\in G\setminus H$ has order $\ge n$ rel $H$, then $g$ is also has order $\ge n$ rel $K$; \label{item: subgroup inherits RTF (1.5)}
			\item If both $(G,H)$ and $(H,K)$ are $n$-RTF, then so is $(G,K)$.\label{item: subgroup inherits RTF (2)}
		\end{enumerate}
	\end{lemma}
	\begin{proof}
		(\ref{item: subgroup inherits RTF (1)}) and (\ref{item: subgroup inherits RTF (1.5)}) are clear from the definition. As for (\ref{item: subgroup inherits RTF (2)}), suppose $gk_1\ldots g k_i=id$ for some $1\le i\le n$ where each $k_j\in K$. Then $g\in H$ since $(G,H)$ is $n$-RTF, from which we get $g\in K$ since $(H,K)$ is $n$-RTF.
	\end{proof}
	
	The $n$-RTF condition is closely related to orders on groups.
	\begin{defn}\label{defn: left relatively convex}
		A subgroup $H$ is \emph{left relatively convex} in $G$ if there is a total order on the left cosets $G/H$ that is $G$-invariant, i.e. $gg_1H\prec gg_2H$ for all $g$ if $g_1H\prec g_2H$.
	\end{defn}
	The definition does not require $G$ to be left-orderable, i.e. $G$ may not have a total order invariant under the left $G$-action. Actually, if $H$ is left-orderable, then $H$ is left relatively convex in $G$ if and only if $G$ has a left $G$-invariant order $\prec$ such that $H$ is convex, i.e. $h\prec g\prec h'$ for some $h,h'\in H$ implies $g\in H$. Many examples and properties of left relatively convex subgroups are discussed in \cite{LRC}. 
	
	\begin{exmp}[\cite{LRC}]
		Let $G$ be a surface group, a pure braid group or a subgroup of some right-angled Artin group. Let $H$ be any maximal cyclic subgroup of $G$, that is, there is no cyclic subgroup of $G$ strictly containing $H$. Then $H$ is left relatively convex in $G$.
	\end{exmp}
	
	The $n$-RTF conditions share similar properties with the left relatively convex condition, and they are weaker.
	\begin{lemma}\label{lemma: left relatively convex implies RTF}
		If $H$ is left relatively convex in $G$, then $(G,H)$ is $\infty$-RTF.
	\end{lemma}
	\begin{proof}
		Suppose for some $g\in G$ we have $gh_1\ldots g h_n=id$ for some $n\ge2$ and $h_i\in H$ for all $1\le i\le n$. Suppose $gH\succ H$. Then $gh_{n-1}gH\succ gh_{n-1}H=gH\succ H$ by left-invariance. By induction, we have $gh_1\ldots gh_{n-1}g H\succ H$, but $gh_1\ldots gh_{n-1} g H=gh_1\ldots gh_n H=H$, contradicting our assumption. A similar argument shows that we cannot have $gH\prec H$. Thus we must have $g\in H$.
	\end{proof}
	
	The $n$-RTF condition has nice inheritance in graphs of groups (Lemma \ref{lemma: inherits RTF in graphs of groups}). To prove it together with a more precise statement (Lemma \ref{lemma: RTF in graphs of groups, tech}), we first briefly introduce reduced words of elements in graphs of groups. See \cite{Serre} for more details. For a graph of groups $G(\Gamma)=\mathcal{G}(\Gamma,\{G_v\},\{G_e\})$, let $F(\Gamma)$ be the quotient group of $(\star G_v) \star F_E$ by relations $\bar{e}=e^{-1}$ and $e t_e(g) e^{-1}=o_e(g)$ for any edge $e\in E$ and $g\in G_e$, where $F_E$ is the free group generated by the edge set $E$. Let $P=(v_0,e_1,v_1,\ldots,e_k,v_k)$ be any oriented path (so $o(e_i)=v_{i-1}, t(e_i)=v_i$), and let $\mu=(g_0,\ldots,g_k)$ be a sequence of elements with $g_i\in G_{v_i}$. We say any word of the form $g_0 e_1 g_1\cdots e_k g_k$ is \emph{of type $(P,\mu)$}, and it is \emph{reduced} if 
	\begin{enumerate}
		\item $k\ge 1$ and $g_i\notin \Image t_{e_i}$ whenever $\bar{e}_{i}=e_{i+1}$; or
		\item $k=0$ and $g_0\neq id$.
	\end{enumerate}
	It is known that every reduced word represents a nontrivial element in $F(\Gamma)$. Fix any base vertex $v_0$, then $G(\Gamma)$ is isomorphic to the subgroup of $F(\Gamma)$ consisting of words of type $(L,\mu)$ for any oriented loop $L$ based at $v_0$ and any $\mu$. Moreover, any nontrivial element is represented by some reduced word of type $(L,\mu)$ as above.
	
	\begin{lemma}\label{lemma: RTF in graphs of groups, tech}
		With notation as above, let $g\in G=G(\Gamma)$ be an element represented by a reduced word of type $(L,\mu)$ with an oriented loop $L=(v_0,e_1,v_1,\ldots,e_j,v_j=v_0)$ and $\mu=(g_0,\ldots,g_j)$, $j\ge 1$. If $j$ is odd, then $g$ is $\infty$-RTF in $(G,G_{v_0})$; if $j$ is even and $g_{j/2}\in G_{v_{j/2}}$ has order $\ge n$
		rel $\Image t_{e_{j/2}}$ for some $n\ge3$, then $g\in G$ has order $\ge n$ rel $G_{v_0}$.
	\end{lemma}
	\begin{proof}
		If $j$ is odd, then the projection $\bar{g}$ of $g$ in the free group $F_E$ is represented by a word of odd length, and thus must be of infinite order. It follows that the projection of $gh_1\cdots g h_k$ is $\bar{g}^k$ for any $h_i\in G_{v_0}$ and $k>0$, which must be nontrivial.
		
		Now suppose $j=2m$ is even and consider $w\defeq gh_1\cdots g h_k$ for some $1\le k<n$ and $h_i\in G_{v_0}$. We claim that there cannot be too much cancellation between the suffix and prefix of two nearby copies of $g$, more precisely, $g_{m}e_{m+1}\cdots e_j g_j h g_0 e_1\cdots e_{m} g_{m}$ for any $h\in G_{v_0}$ can be represented by 
		\begin{enumerate}
			\item $g_{m}g_{m}'g_{m}$ for some $g_{m}'\in \Image t_{e_{m}}$; or \label{item: reduced word case 1}
			\item a reduced word $g_{m}e_{m+1}\cdots g_{j-s-1}e_{j-s}g_{j-s}'e_{s+1}g_{s+1}\cdots e_{m}g_{m}$ with $0\le s< m$.\label{item: reduced word case 2}
		\end{enumerate}
		In fact, if either $e_j\neq \bar{e}_1$, or $e_j=\bar{e}_1$ and $g_j h g_0\notin \Image t_{e_j}$, then we have case (\ref{item: reduced word case 2}) with $s=0$ and $g_{j}'=g_j h g_0$. If $e_j=\bar{e}_1$ and $g_j h g_0\in \Image t_{e_j}$, then $v_{j-1}=v_1=o(e_j)$ and we can replace $g_{j-1}e_j g_j h g_0 e_1 g_1$ by $g_{j-1}o_{e_j}t_{e_j}^{-1}(g_j h g_0) g_1\in G_{v_{j-1}}=G_{v_1}$ to simplify $w$ to a word of shorter length. This simplification procedure either stops in $s$ steps with $s<m$ and we end up with case (\ref{item: reduced word case 2}) or it continues until we arrive at $g_{m}o_{e_{m+1}}(g_{m}^*)g_{m}$ for some $g_{m}^*\in G_{e_{m+1}}$. Note that in the latter case, we must have $\bar{e}_{m}=e_{m+1}$ since the simplification continues all the way. Thus $g_{m}'\defeq o_{e_{m+1}}(g_{m}^*)=t_{e_{m}}(g_{m}^*)\in \Image t_{e_{m}}$.
		
		For each $1\le i\le k$, write $w_i\defeq e_{m+1}\cdots e_j g_j h_i g_0 e_1\cdots e_{m}$ in a reduced form so that $g_{m}w_i g_{m}$ is of the form as in the claim above. Then a conjugate of $w$ in $F(\Gamma)$ is represented by $g_{m}w_1\cdots g_{m}w_k$. 
		If $g_m w_i g_m$ is of the form (\ref{item: reduced word case 1}) above for all $i$, then $w_i\in \Image t_{e_m}$ and $g_m w_1 \cdots g_m w_k\neq id$ since $g_m$ has order $\ge n$ rel $\Image t_{e_{m}}$ by assumption.
		Now suppose $i_1<i_2<\cdots<i_{k'}$ are the indices $i$ such that $g_{m}w_{i}g_{m}$ is of the form (\ref{item: reduced word case 2}) above, where $k'\ge 1$. Up to a cyclic conjugation, assume $i_{k'}=k$ and let $i_0=0$. We write
		$$g_{m}w_1\cdots g_{m}w_k=\tilde{g}_1 w_{i_1}\cdots \tilde{g}_k w_{i_{k'}},$$
		where $\tilde{g}_s \defeq g_{m}w_{i_{s-1}+1}\cdots g_{m}w_{i_s -1}g_{m}$. Note by the definition of the $i_j$, each $w_{i}$ that appears in $\tilde{g}_s$ (i.e. $i_{s-1}+1\le i\le i_s -1$) lies in $\Image t_{e_m}$. It follows that each $\tilde{g}_s \in G_{v_{m}}\setminus\Image t_{e_{m}}$ since $g_m$ has order $\ge n$ rel $\Image t_{e_{m}}$ by assumption and $i_s-i_{s-1}-1<n$. Thus the expression above puts a conjugate of $w$ in reduced form, and hence $w\neq id$.
	\end{proof}
	\begin{cor}\label{cor: RTF in graphs of groups, condition on words}
		With notation as above, let $g\in G=G(\Gamma)$ be an element represented by a reduced word of type $(L,\mu)$ with an oriented loop $L=(v_0,e_1,v_1,\ldots,e_j,v_j=v_0)$ and $\mu=(g_0,\ldots,g_j)$, $j\ge 1$. Suppose for some $n\ge3$ each $g_i\in G_{v_i}$ has order $\ge n$ rel $G_e$ for any edge $e$ adjacent to $v_i$. Then $g\in G$ has order $\ge n$ rel $G_{v_0}$.
	\end{cor}
	\begin{proof}
		This immediately follows from Lemma \ref{lemma: RTF in graphs of groups, tech}.
	\end{proof}
	
	\begin{lemma}\label{lemma: inherits RTF in graphs of groups}
		Let $G(\Gamma)=\mathcal{G}(\Gamma,\{G_v\},\{G_e\})$ be a graph of groups. If the inclusion of each edge group into an adjacent vertex group is $n$-RTF, then for any connected subgraph $\Lambda\subset\Gamma$, the inclusion of $G(\Lambda)\defeq \mathcal{G}(\Lambda,\{G_v\},\{G_e\})\inj G(\Gamma)$ is also $n$-RTF.
	\end{lemma}
	\begin{proof}
		The case where $\Lambda$ is a single vertex $v$ immediately follows from Corollary \ref{cor: RTF in graphs of groups, condition on words} by choosing $v$ to be the base point in the definition of $G(\Gamma)$ as a subgroup of $F(\Gamma)$.
		
		Now we prove the general case with the additional assumption that $\Gamma\setminus\Lambda$ contains only finitely many edges. We proceed by induction on the number of such edges. The assertion is trivially true for the base case $\Lambda=\Gamma$. For the inductive step, let $e$ be some edge outside of $\Lambda$. If $e$ is non-separating, then $G(\Gamma)$ splits as an HNN extension with vertex group $G(\Gamma-\{e\})$. In this case, the inclusion of the edge group $G_e$ is $n$-RTF in $G_{o(e)}$, which is in turn $n$-RTF in $G(\Gamma-\{e\})$ by the single vertex case above. Thus by Lemma \ref{lemma: subgroup inherits RTF}, the inclusion $G_e\inj G(\Gamma-\{e\})$ is also $n$-RTF. The same holds for the inclusion of $G_e$ into $G(\Gamma-\{e\})$ through $G_{t(e)}$. Therefore, using the single vertex case again for the HNN extension, we see that $(G(\Gamma),G(\Gamma-\{e\}))$ is $n$-RTF. Together with the induction hypothesis that $(G(\Gamma-\{e\}),G(\Lambda))$ is $n$-RTF, this implies that $(G(\Gamma),G(\Lambda))$ is $n$-RTF by Lemma \ref{lemma: subgroup inherits RTF}. If $e$ is separating, then $G(\Gamma)$ splits as an amalgam with vertex groups $G(\Gamma_1)$ and $G(\Gamma_2)$ such that $\Gamma=\Gamma_1\sqcup\{e\}\sqcup\Gamma_2$ and $\Lambda\subset \Gamma_1$. The rest of the argument is similar to the previous case.
		
		Finally the general case easily follows from what we have shown, as any $g\in G(\Gamma)\setminus G(\Lambda)$ can be viewed as an element in $G(\Gamma')\setminus G(\Lambda)$ for some connected subgraph $\Gamma'$ of $\Gamma$ with only finitely many edges in $\Gamma'\setminus\Lambda$.
	\end{proof}
	
	See Lemma \ref{lemma: graph products RTF} for a discussion on the $n$-RTF conditions in graph products. One can also use geometry to show that the peripheral subgroups of the fundamental group of certain compact $3$-manifolds are $3$-RTF; See Lemma \ref{lemma:boundary inclusion is 3-RTF}.

\subsection{Quasimorphisms detecting the spectral gap for left relatively convex edge groups}\label{subsec:extremal qm for lrc subgroups}
	
Recall from Definition \ref{defn: left relatively convex} that a subgroup $H\le G$ of a group $H$ is called \emph{left relatively convex}, if there is a left $G$-invariant order $\prec$ on the cosets $G/H = \{ g H \mid g \in G \}$. 
This property has been studied in \cite{LRC}. 
Since left relatively convex subgroups are $\infty$-RTF by Lemma \ref{lemma: left relatively convex implies RTF}, Theorem \ref{thm: $n$-RTF gap, weak version} implies a sharp gap of $1/2$ for hyperbolic elements in graphs of groups where the edge groups are left relatively convex in the vertex groups. The aim of this subsection is to show the following result, which constructs explicit quasimorphisms detecting the gap $1/2$ and gives a completely different proof of Theorem \ref{thm: $n$-RTF gap, weak version} under this stronger assumption.

	\begin{thm} \label{thm:left relatively convex graph of groups}
	Let $G$ be a graph of groups and let $g \in G$ be a hyperbolic element. If every edge group lies left relatively convex in its corresponding vertex groups then there is an explicit homogeneous quasimorphism $\phi \col G \to \R$ such that $\phi(g) \geq 1$ and $D(\phi) \leq 1$.
	\end{thm}

We will reduce it to proving the result for amalgamated free products and HNN extensions. 
The first case has been done in \cite[Theorem 6.3]{Heuer}.
In both cases we construct a so-called \emph{letter-quasimorphism} (see Definition \ref{defn:letter quasimorphism} below) 
and apply Theorem \ref{thm:letter quasimorphisms} below to produce the desired quasimorphism. 

Roughly speaking, a \emph{letter-quasimorphisms} is a map $\Phi \col G \to \Acl$ which has at most one letter ``as a defect'',
where $\Acl$ is the set of alternating (possibly empty) words in $F(\{ \att, \btt \})$.

\begin{defn}[Letter Quasimorphism] \label{defn:letter quasimorphism}
	A \emph{letter-quasimorphism} is a map $\Phi \col G \to \Acl$
	which is alternating, i.e.\ $\Phi(g^{-1}) = \Phi(g)^{-1}$ for all $g \in G$, and such that for every $g,h \in G$, one of the following cases hold:
	\begin{enumerate}
		\item $\Phi(g) \cdot \Phi(h) = \Phi(gh)$, or
		\item there are elements $c_1, c_2, c_3 \in \Acl$ and a letter $\xtt \in \{ \att, \att^{-1}, \btt, \btt^{-1} \}$ such that \emph{up to a cyclic permutation} of $\Phi(g), \Phi(h), \Phi(gh)^{-1}$  we have that 
		\begin{eqnarray*}
			\Phi(g) &=& c_1^{-1} \xtt c_2 \\
			\Phi(h) &=& c_2^{-1} \xtt c_3 \\
			\Phi(gh)^{-1} &=& c_3^{-1} \xtt^{-1} c_1
		\end{eqnarray*}
		where all expressions are supposed to be reduced.
	\end{enumerate}
\end{defn}

The following Theorem \ref{thm:letter quasimorphisms}, producing a quasimorphism out of a letter-quasimorphism, is a key result of \cite{Heuer} stated in slightly greater generality. 
In the original statement $w_l$ and $w_r$ are required to be the identity element, but the same proof works for the more general statement.
A more direct and completely different argument can be found in the old version of the current paper \cite{CH:sclgapgog_old}, which will now appear in a separate companion paper. 
\begin{thm}[\protect{\cite[Theorem 4.7]{Heuer}}] \label{thm:letter quasimorphisms}
	Let $\Phi \col G \to \Acl$ be a letter-quasimorphism and let $g_0 \in G$ be an element such that 
	there are $K>0$ and $w_l,w,w_r \in \Acl$ with
	$\Phi(g_0^n) = w_l w^{n-K} w_r$ for all $n \geq K$, where $w\notin\{ id,\att,\btt\}$. 
	Then there is a homogeneous quasimorphism $\phi \col G \to \R$ such that $\phi(g) \geq 1$ and $D(\phi) \leq 1$. In particular, $\scl_G(g_0) \geq 1/2$.
\end{thm}

We first sketch the construction of the desired letter-quasimorphism in the case of amalgamated free products from \cite{Heuer}.
Let $G = A \star_C B$ be an amalgamated free product of $A$ and $B$ over the subgroup $C$ and suppose that $C\le A$ (resp. $C \le B$) is left relatively convex with order $\prec_A$ (resp. $\prec_B$).
Define a function $\sign_A$ on $A$ by setting
$$
\sign_A(a) =
\begin{cases}
-1 & \mbox{if } aC \prec_A C \\
0 & \mbox{if } aC = C \\
1 & \mbox{if } aC \succ_A C,
\end{cases}
$$
and define $\sign_B$ analogously.
Every element $g \in G$ either lies in $C$ or may be written as a product
$$
g = a_0 b_0 \cdots a_n b_n
$$
where possibly $a_0$ and $b_n$ are the identity and all other $a_i \in A \setminus C$ and $b_i \in B \setminus C$.

Then we define a map $\Phi \col G \to \Acl$ as follows:
If $g \in C$ set $\Phi(g)=e$.
If $g = a_0 b_0 \cdots a_n b_n$ set
$$
\Phi(g) = \att^{\sign_A(a_0)} \btt^{\sign_B(b_0)} \cdots \att^{\sign_A(a_n)} \btt^{\sign_B(b_n)}.
$$
This is well defined although the expression of $g$ is not unique: Each double coset $Ca_iC$ (resp. $Cb_i C$) is uniquely determined and $\sign_A$ (resp. $\sign_B$) is bi-invariant under multiplication by $C$. The lemma below shows that $\Phi$ is a desired letter-quasimorphism for all hyperbolic elements in $G$.

\begin{lemma}[\protect{\cite[Theorem 6.3 and Lemma 6.1]{Heuer}}] \label{lemma:amalgamated free products lq}
Let $G = A \star_C B$ be an amalgamated free product where $C$ lies left relatively convex in both $A$ and $B$. Then the map $\Phi \col G \to \Acl$ defined as above is a 
letter-quasimorphism. 
If $g \in G$ is a hyperbolic element then there is a conjugate $g_0\in G$ of $g$ and $w\in \Acl$ such that $\Phi(g_0^n) = w^n$ for all $n\ge1$. 
\end{lemma}

We now describe a similar construction for HNN extensions. Let $V$ be a group and let $\phi \col A \to B$ be an isomorphism between two subgroups $A, B$ of $V$.
Let 
$$
G = V \star_{\phi} \defeq \langle V, t \mid \phi(a) = t a t^{-1}, a \in A \rangle
$$
be the associated HNN extension with stable letter $t$.

Our construction makes use of a function $\sign_A \col G \to \{ -1, 0, 1 \}$ as follows. If $A,B$ are left relatively convex in $V$, then a result of 
{Antol{\'\i}n}--Dicks--{\v{S}uni{\'c}} \cite[Theorem 14]{LRC} shows that $A$ is also left relatively convex in $G$. Define
$\sign_A \col G \to \{-1,0,1 \}$ as
$$
\sign_A(g) =
\begin{cases}
-1 & \mbox{if } g A \prec A \\
0 & \mbox{if } g A = A \\
1 & \mbox{if } g A \succ A.
\end{cases}
$$
Observe that for every $g \in G$ and $a \in A$, $\sign_A(g) = \sign_A(ga) = \sign_A(ag)$, i.e.\ $\sign_A$ is invariant under both left and right multiplications by $A$.

Britton's lemma \cite{LyndonSchupp} implies that each $g \in G$ in the HNN extension may be written as 
$$
g = v_0 t^{\epsilon_0} \cdots t^{\epsilon_{n-1}} v_n
$$
where $\epsilon_i \in \{ +1,-1 \}$, $v_i \in V$ and there is no subword $t a t^{-1}$ with $a \in A$ or $t^{-1} b t$ with $b \in B$.

Moreover, such an expression is unique in the following sense.
If $v'_0 t^{\epsilon'_0} \cdots t^{\epsilon'_{n'-1}} v'_{n'}$ is another such expression then $n = n'$, $\epsilon_i = \epsilon'_i$ for all $i \in \{0, \cdots, n-1 \}$, and
$$
v_i = h^l_i v'_i h^r_i,
$$
 where 
$$
 \begin{cases} 
 h^l_i = e & \mbox{if } i=0, \\
h^l_i \in A & \mbox{if } \epsilon_{i-1} = 1, i>0 \mbox{ and} \\
 h^l_i \in B & \mbox{if } \epsilon_{i-1} = -1, i >0
 \end{cases}
$$ 
and
$$
 \begin{cases} 
 h^r_i = e & \mbox{if } i=n, \\
h^r_i \in A & \mbox{if } \epsilon_i = -1, i < n \mbox{ and} \\
 h^r_i \in B & \mbox{if } \epsilon_i = 1, i < n. \\
 \end{cases}
$$ 

We define a letter-quasimorphism $\Phi \col G \to F(\{ \att, \btt \})$ as follows.
If $g \in V$, set $\Phi(g)=e$. Otherwise, express $g = v_0 t^{\epsilon_0} \cdots t^{\epsilon_{n-1}} v_n$ in the above form.
Then set
$$
\Phi(g) = \att \btt^{\epsilon_0} \att^{\sign_A(\tilde{v}_1)} \btt^{\epsilon_1} \cdots \att^{\sign_A(\tilde{v}_{n-1})}
 \btt^{\epsilon_{n-1}} \att^{-1}
$$
where $\tilde{v}_i = t_i^l v_i t_i^r$ with 
$$
t^l_i = 
\begin{cases}
t^{-1} & \mbox{if } \epsilon_{i-1} = {-1} \\
e & \mbox{else,}
\end{cases}
$$
and
$$
t_i^r = 
\begin{cases}
t & \mbox{if } \epsilon_{i} = +1 \\
e & \mbox{else}.
\end{cases}
$$

We verify below that the map $\Phi$ is well defined. 
If $v'_0 t^{\epsilon'_0} \cdots t^{\epsilon'_{n-1}} v'_n$ is another such expression, then we know that $\epsilon'_i = \epsilon_i$, and thus the $\btt$-terms agree.

We are left to verify that the signs of $t^l_i v_i t^r_i$ and $t^l_i v'_i t^r_i$ agree, where we know that $v_i = h^l_i v'_i h^r_i$ for some $h^l_i, h^r_i$ as above.

If $\epsilon_{i-1}=1$, then
\begin{eqnarray*}
\sign_A(t^l_i v_i t^r_i) = \sign_A(h^l_i v'_i h^r_i t^r_i) = \sign_A(t^l_i v'_i h^r_i t^r_i)
\end{eqnarray*}
using that $t^l_i$ is trivial, $h^l_i \in A$ and that $\sign_A$ is invariant under left multiplication.
If  $\epsilon_{i-1}=-1$, then
\begin{eqnarray*}
\sign_A(t^l_i v_i t^r_i) = \sign_A(t^{-1} h^l_i v'_i h^r_i t^r_i) = \sign_A(  (t^{-1} h^l_i  t )\cdot t^{-1} v'_i h^r_i t^r_i) = \sign_A( t^l_i v'_i h^r_i t^r_i)
\end{eqnarray*}
using that $t^l_i= t^{-1}$, $h^l_i \in B$, $t^{-1} h^l_i  t \in A$ and that $\sign_A$ is invariant under left multiplication.
In each case we see that
$$
\sign_A(t^l_i v_i t^r_i) = \sign_A(t^l_i v'_i h^r_i t^r_i).
$$
By an analogous argument for the right hand side of $v_i$ and $v'_i$ we see that
$$
\sign_A(t^l_i v_i t^r_i) = \sign_A(t^l_i v'_i t^r_i).
$$
Thus $\Phi$ is indeed well defined.

\begin{lemma} \label{lemma: HNN phi is letter quasimorphism}
Let $V$ be a group with left relatively convex subgroups $A,B \le V$. Let $\phi \col A \to B$ be an isomorphism between $A$ and $B$ and let $G = V \star_\phi$ be the associated HNN extension. Then the map $\Phi \col G \to \Acl$ defined as above is a letter-quasimorphism.
If $g$ is a hyperbolic element, then there is a conjugate $g_0\in G$ and $w_l, w,  w_r,\in \Acl$ such that $\Phi(g_0^m) = w_l w^{m-1} w_r$ for all $m\ge1$ where $w \in \Acl$ is nontrivial and has even length.
\end{lemma}

We will show that for every $g, h \in G$ either $(1)$ or $(2)$ of Definition \ref{defn:letter quasimorphism} holds by realizing $\Phi$ as a map defined on the Bass--Serre tree associated to the HNN extension, which is similar to the proof for the case of amalgamated free products \cite[Lemma 6.1]{Heuer}.

Let $\Trm$ be a tree with vertex set $V(\Trm) = \{ g V \mid g \in G \}$ and edge set 
$$
E(\Trm) = \{ (g V, g t V) \mid g \in G \} \cup \{ (g V, g t^{-1} V) \mid g \in G \}.
$$
Define $o, \tau \col E(\Trm) \to V(\Trm)$ by setting
$o(g V, g t V) = g V$, $\tau(g V, g t V) = g t V$, 
$o(g V, g t^{-1} V)= g V$, $\tau(g V, g t^{-1} V) = g t^{-1} V$.
Moreover, set $\overline{(g V, gtV)} = (gtV, gV)$ and $\overline{(g V, gt^{-1}V)} = (g t^{-1} V, gV)$.
It is well known \cite{Serre} that $\Trm$ is a tree and that $G$ acts on $\Trm$ with vertex stabilizers conjugate to $V$ and edge stabilizers conjugate to $A$. 
For what follows, we will define two maps $\sign_t \col E(\Trm) \to \{1, -1 \}$ and $\mu \col E(\Trm) \to G/A$ on the set of edges.
We set
\begin{itemize}
\item $\sign_t(e) = 1$ if $e=(g V, g t V)$ and
$\sign_t(e) = -1$ if $e=(g V, g t^{-1} V)$. In the first case we call $e$ \emph{positive} and in the second case \emph{negative}.
\item $\mu(e) = gtA$ if $e = (gV, gt V)$, and $\mu(e)=gA$ if $e = (g V, gt^{-1} V)$.
\end{itemize}

\begin{claim}
Both $\sign_t$ and $\mu$ are well defined.
\end{claim}

\begin{proof}
For $\sign_t$ there is nothing to prove. For $\mu$, observe that if $(g V, g t V) = (g' V, g' t V )$, then there is an element $b \in B$ such that $g = g' b$. Thus $g t A = g' b t A = g' t A$, using that $t A t^{-1} = B$ in $G$. Similarly, we see that $\mu$ is also well defined for negative edges.
\end{proof}

\begin{claim}
	Let $e \in E(\Trm)$ be an edge. Then
	$\sign_t(\bar{e}) = - \sign_t(e)$ and $\mu(e) = \mu(\bar{e})$.
\end{claim}

\begin{proof}
	Suppose that $e = (g V, g t V)$. Hence $\sign_t(e)=1$ and $\mu(e)=gtA$.
	We see that $\bar{e} = (g t V, g V) = ((g t) V, (g t) t^{-1} V)$. Thus $\sign_t(\bar{e}) = -1$ and $\mu(\bar{e})=gtA$. The case where $e$ is an edge of the form $(g V, g t^{-1} V)$ is analogous.
\end{proof}

A \emph{reduced path in $\Trm$} is a sequence $\wp = (e_1, \ldots e_n)$, $e_i \in E(\Trm)$ of edges such that $\tau(e_i)=o(e_{i+1})$ for every $i \in \{ 1, \ldots, n-1 \}$, without backtracking.
For what follows, $\mathcal{P}$ will be the set of all reduced paths. We also allow the empty path.

We define the following map $\Xi \col \mathcal{P} \to \Acl$ assigning an alternating word to each path of edges.

If $\wp$ is empty, set $\Xi(\wp) = e$.
Otherwise, suppose that $\wp = (e_1, \ldots, e_n)$. Then define
$\Xi(\wp) \in F(\{ \att, \btt \})$ as
$$
\att
\btt^{\sign_t(e_1)} \att^{\sign_A(\mu(e_1)^{-1} \mu(e_2))} \btt^{\sign_t(e_2)} \cdots 
\btt^{\sign_t(e_{n-1})} \att^{\sign_A(\mu(e_{n-1})^{-1} \mu(e_n))} \btt^{\sign_t(e_n)} \att^{-1}.
$$

For any $g \in G$, let $\wp(g)$ be the unique geodesic in $\Trm$ from $V$ to $g V$. 

\begin{claim} \label{claim:properties of xi}
$\Xi \col \mathcal{P} \to \Acl$ has the following properties:
\begin{enumerate}[(i)]
\item For any $\wp \in \mathcal{P}$ and $g \in G$ we have $\Xi(^g \wp) = \Xi(\wp)$, where $^g \wp$ is the translate of $\wp$ by $g\in G$.\label{item: i in letter qm claim}
\item Let $\wp_1, \wp_2$ be two reduced paths such that the last edge in $\wp_1$ is $e_1$, the first edge in $\wp_2$ is $e_2$, so that $\tau(e_1)=o(e_2)$ and $e_1 \not = \bar{e}_2$.
Then 
$$
\Xi(\wp_1 \cdot \wp_2) = \Xi(\wp_1) \att^{\sign_A(\mu(e_1)^{-1} \mu(e_2))} \Xi(\wp_2)
$$
as reduced words, where $\wp_1 \cdot \wp_2$ denotes the concatenation of paths.\label{item: ii in letter qm claim}
\item For any $g\in G$ we have 
$\Phi(g) = \Xi(\wp(g))$, where $\Phi$ is the map of interest in Lemma \ref{lemma: HNN phi is letter quasimorphism}.\label{item: iii in letter qm claim}
\end{enumerate}
\end{claim}

\begin{proof}
To see (\ref{item: i in letter qm claim}), denote by $^g e$ the translate of an edge $e$ by the element $g \in G$. Then observe that $\mu(^g e) = g \mu(e)$ and that $\sign_t(^g e) = \sign_t(e)$. Thus for every sequence $(^g e_1, ^g e_2)$ of edges we see that $\mu(^g e_1)^{-1} \mu(^g e_2) = \mu(e_1)^{-1} \mu(e_2)$. Property (\ref{item: ii in letter qm claim}) follows immediately from the definition.
For (\ref{item: iii in letter qm claim}), let $g = v_0 t^{\epsilon_0} \cdots t^{\epsilon_{n-1}} v_n$ as above.
Then the unique geodesic between $V$ and $g V$ may be described as
$$
((v_0 V, v_0 t^{\epsilon_0} V) , (v_0 t^{\epsilon_0} v_1 V, v_0 t^{\epsilon_0} v_1 t^{\epsilon_1} V) ,\cdots, 
(v_0 t^{\epsilon_0} v_1 \cdots t^{\epsilon_{n-2}} v_{n-1} V, g V)).
$$
We conclude by multiple applications of (\ref{item: ii in letter qm claim}).
\end{proof}

We can now prove Lemma \ref{lemma: HNN phi is letter quasimorphism}.
\begin{proof}[Proof of Lemma \ref{lemma: HNN phi is letter quasimorphism}]
Since $\sign_A(g^{-1})=-\sign_A(g)$, it is easy to see that $\Phi$ is alternating, i.e.\ $\Phi(g^{-1}) = \Phi(g)^{-1}$.
Let $g \in G$ be a hyperbolic element.
A conjugate $g_0$ of $g$ may be written as a cyclically reduced word $g_0 = v_0 t^{\epsilon_0} \cdots v_n t^{\epsilon_n}$. So the subword $t^{\epsilon_n} v_0 t^{\epsilon_0}$ is reduced.
We observe that 
$$
\Phi(g_0^m) = \att \left( \btt^{\epsilon_0} \att^{\sign_A(\tilde{v}_1)} \cdots \att^{\sign_A(\tilde{v}_n)} \btt^{\epsilon_n} \att^{\sign_A(\tilde{v}_{0})} \right)^{m-1} \btt^{\epsilon_0} \att^{\sign_A(\tilde{v}_1)} \cdots  \att^{\sign_A(\tilde{v}_n)} \btt^{\epsilon_n} \att^{-1}
$$
where $\tilde{v}_i$ is defined as before, interpreting indices mod $n$ in the case of $\tilde{v}_0$.
This produces desired $w_l, w, w_r \in \Acl$ such that $\Phi(g_0^m) = w_l w^{m-1} w_r$ for all $m\ge1$.

It remains to show that $\Phi$ is a letter-quasimorphism. Let $g,h \in G$.
First, suppose that $V$, $gV$ and $gh V$ lie on a geodesic segment of $\Trm$.
If $g V$ lies in the middle of this segment, then there are paths $\wp_1$ and $\wp_2$ such that
$\wp(g) = \wp_1$, $^{g}\wp(h) = \wp_2$ and $\wp(gh) = \wp_1 \cdot \wp_2$.
Using Claim \ref{claim:properties of xi} points (\ref{item: i in letter qm claim}) and (\ref{item: ii in letter qm claim}) 
we see that $\Phi$ behaves as a letter-quasimorphism on such $g,h$. 
The cases where $V$ or $gh V$ lie in the middle of the segment are similar.

Now suppose that $V$, $gV$ and $gh V$ do not lie on a common geodesic segment, so they bound a tripod.
Then there are nontrivial paths $\wp_1, \wp_2, \wp_3 \in \mathcal{P}$ with initial edges $e_1, e_2, e_3$ satisfying
$o(e_1)=o(e_2)=o(e_3)$ and $e_i \not = e_j$ for $i \not = j$ such that
\[
\wp(g) = \wp_1^{-1} \cdot \wp_2 \text{ ,  }  ^g \wp(h) = \wp_2^{-1} \cdot \wp_3 \text{ , and } ^{gh} \wp((gh)^{-1}) = \wp_3^{-1} \cdot \wp_1.
\]

By Claim \ref{claim:properties of xi}, for $c_i=\Xi(\wp_i)$, we have
\begin{align*}
\Phi(g) &= c_1^{-1} \att^{\sign_A(\mu(e_1)^{-1} \mu(e_2))} c_2 \\
\Phi(h) &= c_2^{-1} \att^{\sign_A(\mu(e_2)^{-1} \mu(e_3))}  c_3 \\
\Phi(gh)^{-1} &= c_3^{-1} \att^{\sign_A(\mu(e_3)^{-1} \mu(e_1))}  c_1
\end{align*}
The same holds if we remove the starting letter $\att$ of each $c_i$, and then expressions above are reduced.

It remains to verify that not all of the above signs can be the same.
Indeed suppose that
$$
\sign_A(\mu(e_1)^{-1} \mu(e_2)) = \sign_A(\mu(e_2)^{-1} \mu(e_3)) = \sign_A(\mu(e_3)^{-1} \mu(e_1))=1.
$$
Then $\mu(e_1)^{-1} \mu(e_2) A \succ A$ and $\mu(e_2)^{-1} \mu(e_3) A \succ A$, thus by transitivity and invariance of the order
$$
\mu(e_1)^{-1} \mu(e_3)  A=\mu(e_1)^{-1} \mu(e_2) \mu(e_2)^{-1} \mu(e_3)  A  \succ \mu(e_1)^{-1} \mu(e_2)  A \succ A.
$$
But then $\mu(e_3)^{-1} \mu(e_1)  A \prec A$, contradicting $\sign_A(\mu(e_3)^{-1} \mu(e_1))=1$.
Similarly the signs cannot all be $-1$.
Thus $\Phi$ is a letter-quasimorphism.
\end{proof}

We may now prove Theorem \ref{thm:left relatively convex graph of groups}.

\begin{proof}
Let $G$ be  the fundamental group of a graph of groups such that the edge groups are left relatively convex.
Let $g \in G$ be a hyperbolic element.
$G$ arises as a succession of amalgamated free products and HNN extensions. In particular, similarly to the proof of Lemma \ref{lemma: inherits RTF in graphs of groups}, we may write $G$ either as an amalgamated free product or an HNN extension such that $g$ is hyperbolic.
By a result of 
{Antol{\'\i}n}--{Dicks}--{\v{S}uni{\'c}} \cite[Theorem 14]{LRC}, 
the edge groups of this HNN extension or amalgamated free product are left relatively convex.
Lemmas \ref{lemma:amalgamated free products lq} and \ref{lemma: HNN phi is letter quasimorphism} assert that there is a letter-quasimorphism $\Phi \col G \to \Acl$ such that for a conjugate $g_0\in G$ of $g$ we have $\Phi(g_0^n) = w_l w^{n-1} w_r$ for all $n \geq 1$ and where $w \in \Acl \setminus\{id, \att,\btt\}$. 
Now we conclude using Theorem \ref{thm:letter quasimorphisms} as it provides the desired homogeneous quasimorphism with $\phi(g)=\phi(g_0)\ge1$.
\end{proof}

\section{Scl of vertex and edge groups elements} \label{sec: scl vertex and edge group}
To promote relative spectral gap of graphs of groups to spectral gap in the absolute sense, one needs to further control scl in vertex groups. For a graph of groups $G=\mathcal{G}(\Gamma,\{G_v\},\{G_e\})$, the goal of this section is to characterize $\scl_G$ of chains in vertex groups in terms of $\{\scl_{G_v}\}_{v\in V}$.

Let us start with simple examples. For a free product $G=A\star B$, we know that $G$ has a retract to each free factor, and thus $\scl_G(a)=\scl_A(a)$ for any $a\in A$. This is no longer true in general for amalgams.

\begin{exmp}\label{exmp: vertex scl in amalgam, surface examples}
	Let $S$ be a closed surface of genus $g> 4$. Let $\gamma$ be a separating simple closed curve that cuts $S$ into $S_A$ and $S_B$, where $S_A$ has genus $g-1$ and $S_B$ has genus $1$. Let $a$ be an element represented by a simple closed loop $\alpha$ in $S_A$ that co-bounds a twice-punctured torus $S_m$ with $\gamma$; See Figure \ref{fig: surfscl}. Then $G=\pi_1(S)$ splits as $G=A\star_{C} B$, where $A=\pi_1(S_A)$, $B=\pi_1(S_B)$ and $C$ is the cyclic group generated by an element represented by $\gamma$. In this case, the element $a$ is supported in $A$ and the corresponding loop $\alpha$ does bound a subsurface $S_\ell$ in $S_A$ of genus $g-2$, which is actually a retract of $S_A$ and we have $\scl_A(a)=g-5/2$. However, $\alpha$ also bounds a genus two surface from the $B$ side, which is the union of $S_B$ and $S_m$, showing that $\scl_G(a)\le 3/2$, which is smaller than $\scl_A(a)$ since $g>4$.
\end{exmp}

\begin{figure}
	\labellist
	\small \hair 2pt
	
	\pinlabel $S_A$ at 105 88
	\pinlabel $S_B$ at 230 88
	\pinlabel $S_\ell$ at 80 -8
	\pinlabel $S_m$ at 185 -8
	\pinlabel $\gamma$ at 190 50
	\pinlabel $\alpha$ at 145 50
	
	\endlabellist
	\centering
	\includegraphics[scale=0.8]{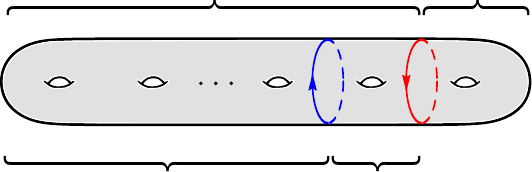}
	\vspace{5pt}
	\caption{Illustration of Example \ref{exmp: vertex scl in amalgam, surface examples}}\label{fig: surfscl}
\end{figure}

The example above shows that one can use chains in edge groups to adjust the given chain in vertex groups to a better one before evaluating it in individual vertex groups. Actually, $\scl_G$ of such a chain is obtained by making the best adjustment of this kind.

\begin{thm}\label{thm: vert scl compute}
	Let $G=\mathcal{G}(\Gamma,\{G_v\},\{G_e\})$ be a graph of groups with $\Gamma=(V,E)$. For any finite collection of chains $c_v\in C_1^H(G_v)$, $v\in V$, we have
	\begin{equation}\label{eqn: vert scl compute}
	\scl_G(\sum_v c_v)=\inf \sum_v \scl_{G_v}(c_v+\sum_{t(e)=v} c_e),
	\end{equation}
	where the infimum is taken over all finite collections of chains $c_e \in C_1^H(G_e)$ satisfying $c_e+c_{\bar{e}}=0$ for each $e\in E$.
\end{thm}
\begin{proof}
	By Mayer--Vietoris (or rational abelianization), $H_1(G)$ is the direct sum $ H_1(\Gamma)\oplus(\oplus_v H_1(G_v))$ mod relations of the form $t_{e*}([c])-t_{\bar{e}*}([c])$ for chains $c\in C_1(G_e)$ and $e\in E$. It follows that $\scl_G(\sum_v c_v)$ is finite if and only if the other side of the equation is. Thus we assume both to be finite in the sequel. 
	We will prove the equality for an arbitrary collection of rational chains $c_v$. Then the general case follows by continuity: Scl of a null-homologous real chain is defined as the limit of scl of null-homologous rational chains with perturbed coefficients; See Definition~\ref{def: scl of real chain}.
	For a similar reason, when the chains $c_v$'s are rational, we can in addition require each $c_e$ to be rational, which does not affect the infimum in equation (\ref{eqn: vert scl compute}). 
	
	Let $X$ be the standard realization of $G$ as in Section \ref{subsec:graph of groups}. Represent each chain $c_v$ by a rational formal sum of elliptic tight loops in the corresponding vertex space $X_v$. Let $f:S\to X$ be any admissible surface in normal form (see Definition \ref{def: normal form}) of degree $n$ for the rational chain $\sum c_v$ in $C_1^H(G)$. For each $v\in V$, let $S_v$ be the union of components in the decomposition of $S$ that are supported in the thickened vertex space $N(X_v)$. Note that $S_v$ only has loop boundary since our chain is represented by elliptic loops. Moreover, any loop boundary supported in some edge space is obtained from cutting $S$ along edge spaces. For each edge $e$ with $t(e)=v$, let $c_e\in C_1^H(G_e)$ be the integral chain that represents the union of loop boundary components of $S_v$ supported in $X_e$. Then $S_v$ is admissible of degree $n$ for the chain $c_v+\frac{1}{n}\sum_{t(e)=v} c_e$ in $C_1^H(G_v)$ for all $v\in V$. See Figure \ref{fig: cutsurf}. Note that we must have $c_e+c_{\bar{e}}=0$ since loops in $c_e$ and $c_{\bar{e}}$ are paired and have opposite orientations. Since $S$ is in normal form and $\sqcup S_v$ has no polygonal boundary, there are no disk components and hence we have 
	$$\frac{-\chi(S_v)}{2n}=\frac{-\chi^-(S_v)}{2n}\ge \scl_{G_v}(c_v+\sum_{t(e)=v} c_e), \quad \text{and}$$
	$$\frac{-\chi(S)}{2n}=\sum_v \frac{-\chi(S_v)}{2n}\ge\sum_v \scl_{G_v}(c_v+\sum_{t(e)=v} c_e).$$
	Since $S$ is arbitrary, this proves the ``$\ge$'' direction in (\ref{eqn: vert scl compute}).
	
	\begin{figure}
		\labellist
		\small \hair 2pt
		
		\pinlabel $S$ at 90 118
		\pinlabel $X_{e_1}$ at 43 165
		\pinlabel $X_{e_2}$ at 123 165
		\pinlabel $S_{v_1}$ at 300 125
		\pinlabel $c_{e_1}$ at 325 73
		\pinlabel $S_{v_2}$ at 375 135
		\pinlabel $c_{\bar{e}_1}$ at 350 73
		\pinlabel $c_{e_2}$ at 445 97
		\pinlabel $c_{v_2}$ at 425 60
		\pinlabel $S_{v_3}$ at 500 155
		\pinlabel $c_{\bar{e}_2}$ at 455 153
		
		\endlabellist
		\centering
		\includegraphics[scale=0.8]{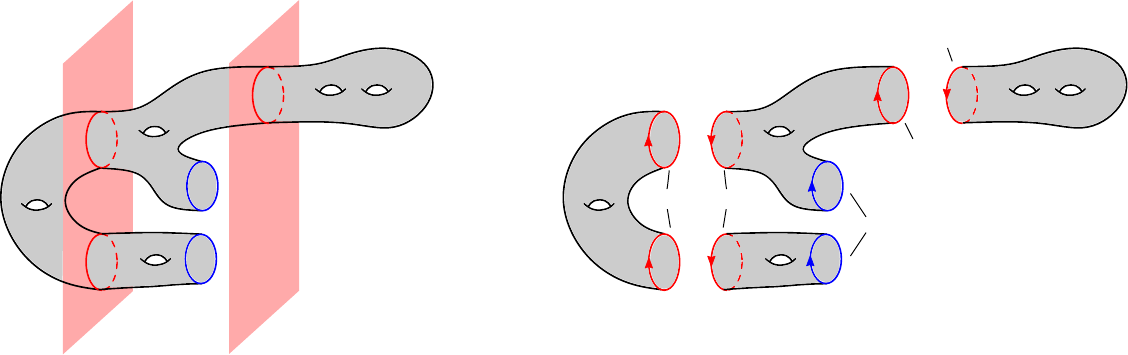}
		\caption{This is an example where $\sum c_v=c_{v_2}$ is supported in a single vertex group $G_{v_2}$, represented as the formal sum of the blue loops in $X_{v_2}$. On the left we have an admissible surface $S$ for $c_{v_2}$ of degree $1$. The edge spaces $X_{e_1},X_{e_2}$ cut $S$ in to $S_{v_1},S_{v_2},S_{v_3}$ shown on the right. The edges are oriented so that $t(e_1)=v_1$ and  $t(e_2)=v_2$. Then $S_{v_1}$ is admissible for its boundary, which is the chain $c_{e_1}$, and similarly $S_{v_3}$ is admissible for $c_{\bar{e}_2,}$. The surface $S_{v_2}$ is admissible for $c_{v_2}+c_{\bar{e}_1}+c_{e_2}$, where $c_{\bar{e}_1}=-c_{e_1}$ due to opposite orientations induced from $S_{v_2}$ and $S_{v_1}$, and similarly $c_{e_2}=-c_{\bar{e}_2}$. Thus the sum of all boundary components of $\sqcup_i S_{v_i}$ is equal to $c_{v_2}$ in $C_1^H(G)$.}\label{fig: cutsurf}
	\end{figure}
	
	Conversely, consider any collection of chains $c_e\in C_1^H(G_e)$ satisfying $c_e+c_{\bar{e}}=0$. Recall that each $c_v$ is rational and we in addition require each $c_e$ to be a rational chain. 
	For each $v\in V$, let $S_v$ be any admissible surface for the rational chain $c_v+\sum_{e:\ t(e)=v} c_e$ of a certain degree $n_v$. Let $n$ be a common multiple of all $n_v$'s. Replacing $S_v$ by $n/n_v$ disjoint copies of itself, we may assume all $S_v$'s have the same degree $n$. Since $c_e+c_{\bar{e}}=0$ for all $e\in E$, the union $\sqcup_v S_v$ is an admissible surface of degree $n$ for a chain equivalent to $\sum_v c_v$ in $C_1^H(G)$. Since the surface $S_v$ and the collection $c_e$ are arbitrary, this proves the other direction of (\ref{eqn: vert scl compute}).
\end{proof}
\begin{rmk}\label{rmk: adjust finite dimensional}
	If some $e\in E$ has $\scl_{G_e}\equiv0$ (e.g. when $G_e$ is amenable), then $\scl_{G_{t(e)}}$ and $\scl_G$ both vanish on $B_1^H(G_e)$ by monotonicity. Given this, the typically infinite-dimensional space $C_1^H(G_e)$ in Theorem \ref{thm: vert scl compute} can be replaced by the quotient $C_1^H(G_e)/B_1^H(G_e)\cong H_1(G_e,\R)$, which is often (e.g. when $G_e$ is finitely generated) finite-dimensional, for which Theorem \ref{thm: vert scl compute} is still valid. Thus if all edge groups have vanishing scl and $\scl_{G_v}$ is understood in each finite-dimensional subspace of $C_1^H(G_v)$ for all $v$, then $\scl_G$ in vertex groups can be practically understood by equation (\ref{eqn: vert scl compute}), which is a convex programming problem.
\end{rmk}

\begin{cor}\label{cor: vertex scl in amalgam}
	Let $G=A\star_{C} B$ be an amalgam. If $C$ has vanishing scl and $H_1(C;\R)=0$, then $\scl_G(a)=\scl_A(a)$ for any $a\in C_1^H(A)$.
\end{cor}
\begin{proof}
	For any chain $c\in C_1^H(C)$, we have $\scl_C(c)=0$ by our assumption, and thus $\scl_A(c)=\scl_B(c)=0$ by monotonicity of scl. The conclusion follows readily from Theorem \ref{thm: vert scl compute}.
\end{proof}

It is clear from Example \ref{exmp: vertex scl in amalgam, surface examples} that the assumption $H_1(C;\R)=0$ is essential in Corollary \ref{cor: vertex scl in amalgam}.

By inclusion of edge groups into vertex groups, apparently Theorem \ref{thm: vert scl compute} also applies to chains supported in edge groups. For later applications, we would like to carry out the characterization of scl in edge groups carefully, where we view scl as a \emph{degenerate norm}.

\begin{defn}\label{defn:degenerate norm}
	A \emph{degenerate norm} $\|\cdot\|$ on a vector space $V$ is a semi-norm on a linear subspace $V^f$, called the \emph{domain} of $\|\cdot\|$, and is $+\infty$ outside $V^f$. The unit norm ball $B$ of $\|\cdot\|$ is the (convex) set of vectors $v$ with $\|v\|\le 1$. The \emph{vanishing locus} $V^z$ is the subspace consisting of vectors $v$ with $\|v\|=0$. Note that $V^z\subset B\subset V^f$.
\end{defn}

In the sequel, norms refer to degenerate norms unless emphasized as genuine norms. A norm in a finite-dimensional space with rational vanishing locus automatically has a ``spectral gap''.

\begin{lemma}\label{lemma: auto gap of norm}
	Let $\|\cdot\|$ be a norm on $\R^n$. If the vanishing locus $V^z$ is a rational subspace, then $\|\cdot\|$ satisfies a gap property on $\Z^n$: there exists $C>0$ such that either $\|P\|=0$ or $\|P\|\ge C$ for all $P\in \Z^n$.
\end{lemma}
\begin{proof}
	Extend a rational basis $e_1,\ldots,e_m$ of $V^z$ to a rational basis $e_1,\ldots,e_n$ of $\R^n$, where $m\le n$. Then there is some integer $N>0$ such that any $P=\sum_i P_i e_i\in \Z^n$ has $NP_i\in\Z$ for all $i$. The restriction of $\|\cdot\|$ to the subspace spanned by $e_{m+1},\ldots,e_n$ has trivial vanishing locus and thus there is a constant $C>0$ such that $\|\sum_{j=m+1}^n Q_j e_j\|\ge NC$ for any integers $Q_{m+1},\ldots,Q_n$ unless they all vanish. Therefore, if $P=\sum_{i=1}^n P_ie_i\in \Z^n\setminus V^z$, then $\|P\|=\|\sum_{j=m+1}^n NP_j e_j\|/N\ge C$.
\end{proof}

Recall from Subsection \ref{subsec:scl basic} that $\scl_G$ is a (degenerate) norm on $C_1^H(G)$. If $G$ is a subgroup of $\tilde{G}$, then $\scl_{\tilde{G}}$ also restricts to a norm on $C_1^H(G)$.

Let us set up some notation. For a graph of groups $G=\mathcal{G}(\Gamma,\{G_v\},\{G_e\})$ with $\Gamma=(V,E)$. For each vertex $v\in V$, let $C_v\defeq\bigoplus_{t(e)=v}C_1^H(G_e)$ be the space parameterizing chains in edge groups adjacent to $v$. Let $C_E\defeq\bigoplus_{\{e,\bar{e}\}}C_1^H(G_e)$ parameterize chains in all edge groups, one for each unoriented edge $\{e,\bar{e}\}$. 

Define $\|(c_{\{e,\bar{e}\}})\|_E\defeq\scl_G(\sum c_{\{e,\bar{e}\}})$ for any $(c_{\{e,\bar{e}\}})\in C_E$. Equivalently, $\|\cdot\|_E$ is the pullback of $\scl_G$ via
$$\bigoplus_{\{e,\bar{e}\}}C_1^H(G_e)
\longrightarrow\bigoplus_{\{e,\bar{e}\}}C_1^H(G)
\stackrel{+}{\longrightarrow}C_1^H(G),$$
where the former map is inclusion on each summand and the latter map takes the summation.

Similarly, for each vertex $v\in V$, we pull back $\scl_{G_v}$ to get a norm $\|\cdot\|_v$ on $C_v=\bigoplus_{t(e)=v}C_1^H(G_e)$ via the composition
$$\bigoplus_{e:\ t(e)=v}C_1^H(G_e)
\stackrel{\oplus t_{e*}}{\longrightarrow}\bigoplus_{e:\ t(e)=v}C_1^H(G_v)
\stackrel{+}{\longrightarrow}C_1^H(G_v),$$

Note that $\bigoplus_{v\in V}C_v$ is naturally isomorphic to $\bigoplus_{\{e,\bar{e}\}}(C_1^H(G_e)\oplus C_1^H(G_{\bar{e}}))$, which has a surjective map $\pi$ to $C_E$ whose restriction on each summand is
$$C_1^H(G_e)\oplus C_1^H(G_{\bar{e}})=C_1^H(G_e)\oplus C_1^H(G_{e})\stackrel{+}{\to}C_1^H(G_e).$$
The kernel of $\pi$ is exactly the collections of chains $c_e$ over which we take infimum in equation (\ref{eqn: vert scl compute}).

Equip $\bigoplus_{v\in V}C_v$ with the $\ell^1$ product norm given by $\|(x_v)\|_1\defeq\sum_{v\in V} \|x_v\|_v$. That is, the $\ell^1$ norm of a vector $(x_v)$ is the sum of norms $\|x_v\|_v$ of individual components $x_v\in C_v$. This induces a norm $\|\cdot \|$ on the quotient $C_E$ via $\|x\|\defeq \inf_{x=\pi(y)} \|y\|_1$.
\begin{cor}\label{cor: scl edge compute}
	With the notation above, the two norms $\|\cdot\|$ and $\|\cdot\|_E$ on $C_E$ agree.
\end{cor}
\begin{proof}
	Choose an orientation $e$ for each unoriented edge $\{e,\bar{e}\}$.
	For any $(c_{\{e,\bar{e}\}})\in C_E$, let $c_v=\sum_{t(e)=v} t_{e*} c_{\{e,\bar{e}\}}\in C_1^H(G_v)$ using the chosen orientation for each $v\in V$. 
	Then by definition we have $\|(c_{\{e,\bar{e}\}})\|_E=\scl_G(\sum_v c_v)$ and $\|(c_{\{e,\bar{e}\}})\|=\inf\sum_v\scl_{G_v}(c_v+\sum c_e)$. 
	Thus the assertion follows immediately from Theorem \ref{thm: vert scl compute}.
\end{proof}

\subsection{Scl in the edge group of amalgamated free products}  \label{subsec:scl in edge gp of Amalgamated Free Products}
Corollary \ref{cor: scl edge compute} is particularly simple in the special case of amalgams $G=A\star_{C}B$. We also give a precise description of the unit norm ball in Theorem \ref{thm:scl edge for amalgam, norm ball}.

\begin{thm}\label{thm:scl edge for amalgam} 
	Let $G=A\star_{C}B$ be the amalgamated free product associated to inclusions $\iota_A:C\to A$ and $\iota_B:C\to B$.
	Then for any chain $c\in C_1^H(C)$, we have 
	\begin{equation}\label{eqn:scledge}
		\scl_G(c)=\inf_{\substack{c_1,c_2\in C_1^H(C)\\ c_1+c_2=c}}\{\scl_A(c_1)+\scl_B(c_2)\}.
	\end{equation}
	In particular, if $C\cong\Z$ with generator $t$, then
	\[
	\scl_G(t) = \min \{ \scl_A(t), \scl_B(t) \}.
	\]
\end{thm}

\begin{proof}
	Equation (\ref{eqn:scledge}) is an explicit equivalent statement of Corollary \ref{cor: scl edge compute} in our case. 
	When $C\cong \Z=\langle t\rangle$ and $c=t$, we have $c_1=\lambda t$ and $c_2=(1-\lambda) t$ for some $\lambda\in \R$ and $\scl_G(t)=\inf_\lambda \{|\lambda|\scl_A(t)+|1-\lambda|\scl_B(t)\}$, where the optimization is achieved at either $\lambda=0$ or $\lambda=1$.
\end{proof}

The simple formula (\ref{eqn:scledge}) in the case of amalgams allows us to describe the unit norm ball of $\left.(\scl_G)\right|_{C_1^H(C)}$. To accomplish this, we need the following notion from convex analysis. The \emph{algebraic closure} of a set $A$  in a vector space $V$ consists of points $x$ such that there is some $v\in V$ so that for any $\epsilon>0$, there is some $t\in[0,\epsilon]$ with $x+tv\in A$. If $A$ is convex, the algebraic closure of $A$ coincides with $\mathrm{lin}(A)$ defined in \cite[p.\ 9]{Convanal}, which is also convex. If $A$ is in addition finite-dimensional, then its algebraic closure agrees with the topological closure of $A$ \cite[p.\ 59]{Convanal}.

We also introduce the following definition for the discussion below.

\begin{defn}
	For two degenerate norms $\|\cdot\|_1$ and $\|\cdot\|_2$ on a vector space $V$. The \emph{$\ell^1$-mixture norm} $\|\cdot\|_m$ is defined as $$\|v\|_m=\inf_{v_1+v_2=v}(\|v_1\|_1+\|v_2\|_2).$$
\end{defn}

Then the norm $\left.(\scl_G)\right|_{C_1^H(C)}$ is the $\ell^1$-mixture of $\scl_A$ and $\scl_B$ by (\ref{eqn:scledge}) in Theorem \ref{thm:scl edge for amalgam}.

Let $V_i^f$, $V_i^z$ and $B_i$ be the domain, vanishing locus and unit ball of the norm $\|\cdot \|_i$ for $i=1,2$ respectively. Note that the domain $V_m^f$ of $\|\cdot\|_m$ is $V_1^f+V_2^f$, and the vanishing locus $V_m^z$ of $\|\cdot\|_m$ contains $V_1^z+V_2^z$ as a subspace.
\begin{lemma}\label{lemma:vanishing locus}
	The vanishing locus $V_m^z$ of $\|\cdot\|_m$ is $V_1^z+V_2^z$ if $V$ is finite-dimensional.
\end{lemma}
\begin{proof}
	Fix an arbitrary genuine norm $\|\cdot\|$ on $V$. Let $E_1$ be a subspace of $V_1^f$ such that $V_1^f$ is the direct sum of $E_1$ and $V_1^z$. Then $\|\cdot\|_1$ is a genuine norm on $E_1$, a finite-dimensional space, and thus there exists $r_1>0$ such that any $u\in E_1$ with $\|u\|_1\le 1$ has $\|u\|\le r_1$. It follows that every vector $v$ in $B_1$ can be written as $v_0+u$ where $v_0\in V_1^z$ and $\|u\|\le r_1$. A similar result holds for $\|\cdot\|_2$ with some constant $r_2>0$. Then for any $v\in V_m^z$, for any $\epsilon>0$, we have $v=v_1+v_2+u_1+u_2$ for some $v_i\in V_i^z$ and $u_i$ satisfying $\|u_i\|\le \epsilon r_i$, $i=1,2$. Hence $V_m^z$ is contained in the closure of $V_1^z+V_2^z$. But $V_1^z+V_2^z$ is already closed since $V$ is finite-dimensional.	
\end{proof}

The unit norm ball $B_m$ of an $\ell^1$-mixture norm $\|\cdot\|_m$ has a simple description.
\begin{lemma}\label{lemma: norm ball}
	The unit norm ball $B_m$ is the algebraic closure of $\conv(B_1\cup B_2)$, where $\conv(\cdot)$ takes the convex hull. If the underlying space is finite-dimensional, then we can take the topological closure of $\conv(B_1\cup B_2)$ instead.
\end{lemma}
\begin{proof}
	Fix any $v\in B_m$. For any $\epsilon>0$, there exist $v_1,v_2$ with $v=v_1+v_2$ and $\|v_1\|_1+\|v_2\|_2<1+\epsilon$. Let $u_i\in B_i$ be $v_i/\|v_i\|_i$ if $\|v_i\|_i\neq 0$, and $0$ otherwise. With $d=\max(1,\|v_1\|_1+\|v_2\|_2)$, we have
	$$\frac{v}{d}=\frac{\|v_1\|_1}{d}\cdot u_1+\frac{\|v_2\|_2}{d}\cdot u_2+(1-\frac{\|v_1\|_1+\|v_2\|_2}{d})\cdot 0\in \conv(B_1\cup B_2).$$
	It follows that for any $\epsilon>0$, there is some $0\le t<\epsilon$ such that $(1-t)v\in \conv(B_1\cup B_2)$. Thus $v$ is in the algebraic closure of $\conv(B_1\cup B_2)$.
	
	Conversely, fix any $v$ in the algebraic closure of $\conv(B_1\cup B_2)$. Note that $\conv(B_1\cup B_2)$ is a subset of $V_1^f+V_2^f$ and that any linear subspace is algebraically closed, so the algebraic closure of $\conv(B_1\cup B_2)$ is a subset of $V_1^f+V_2^f$. Then by definition, there is some $u=u_1+u_2$ with $u_i\in V_i^f$ such that for any $\epsilon>0$, we have $v+tu\in \conv(B_1\cup B_2)$ for some $0\le t\le \epsilon$. Thus $v=\lambda v_1+(1-\lambda)v_2-tu=[\lambda v_1-tu_1]+[(1-\lambda) v_2-tu_2]$ for some $\lambda\in[0,1]$ and $v_i\in B_i$. We see 
	\begin{eqnarray*}
		\|v\|_m&\le& \|\lambda v_1-tu_1\|_1+\|(1-\lambda) v_2-tu_2\|_2\\
		&\le&\lambda\|v_1\|_1+(1-\lambda)\|v_2\|_2+t(\|u_1\|_1+\|u_2\|_2)\\
		&\le&1+\epsilon(\|u_1\|_1+\|u_2\|_2).
	\end{eqnarray*}
	Since $\epsilon$ is arbitrary and $\|u_1\|_1+\|u_2\|_2$ is finite, we get $v\in B_m$.
\end{proof}
\begin{rmk}
	It is necessary to take the algebraic closure. On $\R^2=\{(x,y)\}$, let $\|(x,y)\|_1=\infty$ if $y\neq0$ and $\|(x,0)\|_1=|x|$, and let $\|(x,y)\|_2=\infty$ if $x\neq 0$ and $\|(0,y)\|_2=0$. Then their $\ell^1$-mixture has formula $\|(x,y)\|_m=|x|$. Thus the unit balls $B_1=[-1,1]\times\{0\}$, $B_2=\{0\}\times \R$ and $B_m=[-1,1]\times \R$. Thus $\conv(B_1\cup B_2)=(-1,1)\times\R\cup \{(\pm 1,0)\}$ does not agree with $B_m$ but its algebraic closure does.
\end{rmk}

Lemma \ref{lemma: norm ball} in combination with Theorem \ref{thm:scl edge for amalgam} allows us to describe the unit norm ball of $\scl_G$ on the edge group $C$ for $G=A\star_C B$.
\begin{thm}\label{thm:scl edge for amalgam, norm ball}
	Let $G=A\star_C B$ be the amalgamated free product associated to inclusions $\iota_A: C\to A$ and $\iota_B:C\to B$. Then unit ball of $\scl_G$ on $C$ equals the algebraic closure of $\conv(B_A \cup B_B)$, where $\conv(\cdot)$ denotes the convex hull, $B_A$ and $B_B$ are the unit norm balls of the pullbacks of $\scl_A$ and $\scl_B$ via $\iota_A$ and $\iota_B$ on $C_1^H(C)$ respectively. 
\end{thm}
\begin{proof}
	By Theorem \ref{thm:scl edge for amalgam}, the norm $\scl_G$ restricted to $C_1^H(C)$ is the $\ell^1$-mixture of (the pullback of) $\scl_A$ and $\scl_B$. Thus the result follows from Lemma \ref{lemma: norm ball}.
\end{proof}

This enables us to look at explicit examples showing how scl behaves under surgeries.

\begin{exmp}\label{exmp: surgery example}
	Let $\Sigma$ be a torus with one boundary $\gamma=\partial \Sigma$. Then $X_A\defeq S^1\times \Sigma$ is a compact $3$-manifolds with torus boundary $T_A$. Let $\gamma_A$ be a chosen section of $\gamma$ in $T_A$ and let $\tau_A$ be a simple closed curve on $T_A$ representing the $S^1$ factor. Let $C\defeq\pi_1(T_A)=\langle\gamma_A, \tau_A \rangle\cong \Z^2$ be the peripheral subgroup of $A\defeq\pi_1(X_A)$, where we abuse the notation and use $\gamma_A,\tau_A$ to denote their corresponding elements in $\pi_1(T_A)$.
	
	Let $X_B$ be another copy of $X_A$, where $T_B$, $\gamma_B$ and $\tau_B$ correspond to $T_A$, $\gamma_A$ and $\tau_A$ respectively. For any coprime integers $p,q$, there is an orientable closed $3$-manifold $M_{p,q}$ obtained by gluing $T_A$ and $T_B$ via a map $\phi:\pi_1(T_B)\to \pi_1(T_A)=C$ taking $\gamma_B$ to $p\gamma_A+q\tau_A$. Then $\pi_1(M_{p,q})$ is an amalgam $A\star_{C} B$ where $B\defeq\pi_1(X_B)$ and the inclusion $C\to B$ is given by $C\overset{\phi^{-1}}{\longrightarrow}\pi_1(T_B)\inj B$.
	
	Identify $H_1(C;\R)$ with $\R^2$ with $(1,0)$ representing $[\gamma_A]$ and $(0,1)$ representing $[\tau_A]$. According to Remark \ref{rmk: adjust finite dimensional}, $\scl_A$ and $\scl_B$ induce norms on $H_1(C;\R)\cong C_1^H(C)/B_1^H(C)$. Then the norm $\scl_A$ on $H_1(C;\R)$ has a one-dimensional unit norm ball, which is the segment connecting $(-2,0)$ and $(2,0)$ since $\scl_A(\gamma_A)=\scl_\Sigma(\gamma)=1/2$. Similarly the unit norm ball of $\scl_B$ on $H_1(C;\R)$ is the segment connecting $(2p,2q)$ and $(-2p,-2q)$. By Theorem \ref{thm:scl edge for amalgam, norm ball}, the unit norm ball of $\scl_{M_{p,q}}$ on $H_1(C;\R)$ is the convex hull of $\{(\pm 2,0), \pm (2p,2q)\}$ (which is already closed), which intersects the positive $y$-axis at $(0,\frac{2q}{p+1})$ when $p,q\ge0$. In this case, we have $\scl_{M_{p,q}}(\tau_A)=\frac{p+1}{2q}$.
\end{exmp}

\section{Spectral gaps of graph products} \label{sec:spectral gap graph products}
In this section we apply Theorem \ref{thm: $n$-RTF gap, strong version} to obtain sharp gaps of scl in graph products, which are groups obtained from given collections of groups generalizing both free products and direct products. 
\begin{defn}
	Let $\Gamma$ be a \emph{simple} graph (not necessarily connected or finite) and let $\{G_v\}$ be a collection of groups each associated to a vertex of $\Gamma$. The \emph{graph products} $G_\Gamma$ is the quotient of the free product $\star G_v$ by the set of relations $\{[g_u,g_v]=1\ |\ g_u\in G_u, g_v\in G_v, u,v\text{ are adjacent}\}$.
\end{defn}

\begin{exmp} \label{exmp:graph products}
	Here are some well known examples.
	\begin{enumerate}
		\item If $\Gamma$ has no edges at all, then $G_\Gamma$ is the free product $\star_v G_v$;
		\item If $\Gamma$ is a complete graph, then $G_\Gamma$ is the direct product $\oplus_v G_v$;
		\item If each $G_v\cong \Z$, then $G_\Gamma$ is called the \emph{right-angled Artin group} (RAAG for short) associated to $\Gamma$;
		\item If each $G_v\cong \Z/2\Z$, then $G_\Gamma$ is called the \emph{right-angled Coxeter group} associated to $\Gamma$.
	\end{enumerate}
\end{exmp}

We first introduce some terminology necessary for the statements and proofs. Denote the vertex set of $\Gamma$ by $V(\Gamma)$. For any $V'\subset V(\Gamma)$,  the \emph{full subgraph} on $V'$ is the subgraph of $\Gamma$ whose vertex set is $V'$ and edge set consists of all edges of $\Gamma$ connecting vertices in $V'$. Any full subgraph $\Lambda$ gives a graph product denoted $G_\Lambda$ which is naturally a subgroup of $G_\Gamma$. It is actually a retract of $G_\Gamma$, by trivializing $G_v$ for all $v\notin \Lambda$. Denote by $lk(v)$ the \emph{link} of a vertex $v$, which is the full subgraph of $\{w\ |\ w\text{ is adjacent to }v\}$. The \emph{star} $st(v)$ is the full subgraph of $\{v\}\cup\{w\ |\ w\text{ is adjacent to }v\}$.

Finally, each element $g\in G_\Gamma$ can be written as a product $g_1\cdots g_m$ with $g_i\in G_{v_i}$. Such a product is \emph{reduced} if 
\begin{enumerate}
	\item $g_i\neq id$ for all $i$, and
	\item $v_i\neq v_j$ whenever we have $i\le k<j$ such that $[g_i,g_t]=id$ for all $i\le t\le k$ and $[g_t,g_j]=id$ for all $k+1\le t\le j$.
\end{enumerate}

It is known that every nontrivial element of $G_\Gamma$ can be written in a reduced form, which is unique up to certain operations (syllable shuffling) \cite[Theorem 3.9]{graphproductthesis}. In particular, any $g$ expressed in the reduced form above is nontrivial in $G_\Gamma$.

The following standard splitting of graph products is the key to applying our techniques.
\begin{lemma}\label{lemma: graph products splitting}
	For a graph product $G_\Gamma$ and any vertex $v$ of $\Gamma$, the group $G_\Gamma$ is an amalgam $A\star_{C} B$ with $A=G_{st(v)}$, $C=G_{lk(v)}$ and $B=G_{\Lambda}$, where $\Lambda$ is the full subgraph of the complement of $v$ in $\Gamma$.
\end{lemma}

\begin{lemma}\label{lemma: graph products RTF}
	Let $G_\Gamma$ be a graph product. Suppose $g=g_1\cdots g_m\in G_\Gamma$ ($m\ge 1$) is in reduced form such that for some $n\ge 3$ each $g_i\in G_{v_i}$ has order at least $n$, $1\le i\le m$. Then $g$ has order $\ge n$ rel $G_\Lambda$ for a full subgraph $\Lambda\subset \Gamma$ unless $v_i\in \Lambda$ for all $1\le i\le m$.
\end{lemma}
\begin{proof}
	We proceed by induction on $m$. The base case $m=1$ is obvious using the retract from $G_\Gamma$ to $G_{v_1}$. For the inductive step, we show $g$ has order $\ge n$ rel $G_\Lambda$ if $v_1\notin \Lambda$, and the proof is similar if $v_i\notin \Lambda$ for some $i>1$. It suffices to prove that $g\in G_\Gamma$ has order $\ge n$ rel $G_{\Lambda_1}$ where $\Lambda_1$ is the full subgraph of the complement of $v_1$ in $\Gamma$ since $G_\Lambda\le G_{\Lambda_1}$ and using Lemma \ref{lemma: subgroup inherits RTF} (\ref{item: subgroup inherits RTF (1.5)}).
	Consider $G_\Gamma$ as an amalgam $A\star_{C} B$ with $A=G_{st(v_1)}$, $C=G_{lk(v_1)}$ and $B=G_{\Lambda_1}$ as in Lemma \ref{lemma: graph products splitting}. Then there is a unique decomposition of $g$ into $g=a_1b_1\cdots a_\ell b_\ell$ with $a_i\in A$ and $b_i\in B$, where each $a_i$ is a maximal subword of $g_1\cdots g_m$ that stays in $A-C$. To be precise, there is some $\ell\ge 1$ and indices $0=\beta_0< \alpha_1<\beta_1<\cdots<\alpha_\ell\le \beta_\ell\le m$, such that $g=a_1b_1\cdots a_\ell b_\ell$, where $a_i=g_{\beta_{i-1}+1}\cdots g_{\alpha_i}\in A$ and $b_i=g_{\alpha_{i}+1}\cdots g_{\beta_i}\in B$ for $1\le i\le \ell$, and such that
	\begin{enumerate}
		\item $b_\ell=id$ if $\alpha_\ell=m$;
		\item For each $1\le i\le \ell$, we have $v_j\in st(v_1)$ for all $\beta_{i-1}+1\le j\le \alpha_i$, and $v_j=v_1$ for some $\beta_{i-1}+1\le j\le \alpha_i$; and
		\item For each $1\le i\le \ell$ (or $i<\ell$ if $\alpha_\ell=m$), we have $v_j\neq v_1$ for all $\alpha_{i}+1\le j\le \beta_i$, and $v_{\alpha_{i}+1}, v_{\beta_i}\notin st(v_1)$.
	\end{enumerate}
	Since $g$ is reduced, so are each $a_i$ and $b_i$. Thus each $a_i\in A$ (resp. $b_i\in B$, except the case $b_\ell=id$) has order $\ge n$ rel $C$ by the induction hypothesis, and thus $g\in G_\Gamma$ has order $\ge n$ rel $B$ by Lemma \ref{lemma: RTF in graphs of groups, tech}, unless $\ell=1$ and $b_\ell=id$. In the exceptional case we have $v_i\in st(v_1)$ for all $1\le i\le m$, and the assertion is obvious using the direct product structure of $G_{st(v_1)}$ and the fact that $g=g_1\cdots g_m$ is reduced.
\end{proof}

Now we can deduce lower bounds of scl in graph products using Theorem \ref{thm: $n$-RTF gap, strong version}.
\begin{thm}\label{thm: gap for graph products, element-wise statement}
	Let $G_\Gamma$ be a graph product. Suppose $g=g_1\cdots g_m\in G_\Gamma$ ($m\ge 1$) is in cyclically reduced form such that for some $n\ge 3$ each $g_i\in G_{v_i}$ has order at least $n$, $1\le i\le m$. Then either
	$$\scl_{G_\Gamma}(g)\ge \frac{1}{2}-\frac{1}{n},$$
	or the full subgraph $\Lambda$ on $\{v_1,\ldots, v_m\}$ in $\Gamma$ is a complete graph. In the latter case, we have
	$$\scl_{G_\Gamma}(g)=\scl_{G_\Lambda}(g)=\max \scl_{G_i}(g_i).$$
\end{thm}
\begin{proof}
	Fix any $v_k$, similar to the proof of Lemma \ref{lemma: graph products RTF}, we express $G_\Gamma$ as an amalgam $A\star_{C} B$, where $A$, $C$ and $B$ are the graph products associated to $st(v_k)$, $lk(v_k)$ and the full subgraph on $V(\Gamma)-\{v_k\}$ respectively. If there is some $v_i\notin st(v_k)$, then up to a cyclic conjugation, $g=a_1b_1\cdots a_\ell b_\ell$ where $\ell\ge 1$, each $a_i$ and $b_i$ is a product of consecutive $g_j$'s such that $b_i\in B-C$ and each $a_i\in A-C$ is of maximal length. Since $g$ is cyclically reduced, each $a_i\in A$ (resp. $b_i\in B$) has order $\ge n$ rel $C$ by Lemma \ref{lemma: graph products RTF}. It follows from Theorem \ref{thm: $n$-RTF gap, strong version} that $\scl_{G_\Gamma}(g)\ge 1/2-1/n$.
	
	Therefore, the argument above implies $\scl_{G_\Gamma}(g)\ge 1/2-1/n$ unless $v_i\in st(v_k)$ for all $i$, which holds for all $k$ only when the full subgraph $\Lambda$ on $\{v_1,\ldots,v_m\}$ in $\Gamma$ is complete. In this case, $G_\Gamma$ retracts to $G_\Lambda=\oplus G_{v_i}$. Then $v_i\neq v_j$ whenever $i\neq j$ since $g$ is reduced, thus the conclusion follows from Lemma \ref{lemma: basic prop of scl} (\ref{item: basic prop direct prod}).
\end{proof}
\begin{rmk}\label{rmk: graph product gap sharp}
	The estimate is sharp in the following sense. For any $g_v\in G_v$ of order $n\ge2$ and any $g_u\in G_u$ of order $m\ge 2$ with $u$ not equal or adjacent to $v$, then the retract from $G_\Gamma$ to $G_u\star G_v$ gives
	$$\scl_{G_\Gamma}([g_u,g_v])=\scl_{G_u\star G_v}([g_u,g_v])=\frac{1}{2}-\frac{1}{\min(m,n)}$$
	by \cite[Proposition 5.6]{Chen:sclfp}.
\end{rmk}

\begin{thm}\label{thm: gap for graph products, weak version}
	Let $G_\Gamma$ be the graph product of $\{G_v\}$. Suppose for some $n\ge 3$ and $C>0$, each $G_v$ has no $k$-torsion for all $2\le k\le n$ and has a strong gap $C$. Then $G_\Gamma$ has a strong gap $\min\{C,1/2-1/n\}$.
\end{thm}
\begin{proof}
	For any nontrivial $g\in G_\Gamma$ written in reduced form, by Theorem \ref{thm: gap for graph products, element-wise statement}, we either have $\scl_{G_\Gamma}(g)\ge 1/2-1/n$ or $\scl_{G_\Gamma}(g)=\max \scl_{G_i}(g_i)\ge C$.
\end{proof}
\begin{cor}
	For $n\ge 3$, any graph product of abelian groups without $k$-torsion for all $2\le k\le n$ have strong gap $1/2-1/n$. In particular, all right-angled Artin groups have strong gap $1/2$, originally proved by the second author \cite{Heuer}.
\end{cor}

Unfortunately, our result does not say much about the interesting case of right-angled Coxeter groups due to the presence of $2$-torsion.

\begin{quest}
	Is there a spectral gap for every right-angled Coxeter group? If so, is there a uniform gap?
\end{quest}

The first part of the question has now been answered affirmatively by the authors in a subsequent paper \cite[Corollary 6.18]{CH:RAAG_chain}, but the gap obtained is not uniform over all RACGs, relying on the defining graph. However, we do have a uniform gap $1/60$ restricting to hyperbolic RACGs \cite[Corollary 6.19]{CH:RAAG_chain}. It is still not clear if there is a uniform gap for all RACGs.

\section{Spectral gap of $3$-manifold groups} \label{sec: spectral gap 3 mnfd}
In this section, we show that any closed connected orientable $3$-manifold has a scl spectral gap. All $3$-manifolds in this section are assumed to be \emph{orientable}, \emph{connected} and \emph{closed} unless stated otherwise. Throughout this section, we will use $\scl_M$ to denote $\scl_{\pi_1(M)}$ and use $\scl_{(M,\partial M)}$ to denote scl in $\pi_1(M)$ relative to the peripheral subgroups when $M$ potentially has boundary.

\subsection{Decompositions of $3$-manifolds} \label{subsec: decomp of 3 mfd}
We first recall some important decompositions of $3$-manifolds and the geometrization theorem. Every $3$-manifold has a unique \emph{prime decomposition} as a connected sum of finitely many prime $3$-manifolds. 
We call a prime $3$-manifold \emph{non-geometric} if it does not admit one of the eight model geometries.

By the geometrization theorem, conjectured by Thurston \cite{Thurston} and proved by Perelman \cite{Perelman1,Perelman2,Perelman3}, the \emph{JSJ decomposition} \cite{JSJ1,JSJ2} cuts each non-geometric prime $3$-manifold along a \emph{nonempty} 
minimal finite collection of embedded disjoint incompressible $2$-tori into compact manifolds (with incompressible tori boundary) that each either admits one of the eight model geometries in the interior of \emph{finite volume} or is the twisted $I$-bundle $K$ over the Klein bottle defined below. Here $K$ is $(T^2\times[0,1])/\sigma$, where the involution $\sigma$ is defined as $\sigma(x,t)=(\tau(x),1-t)$ and $\tau$ is the unique nontrivial deck transformation for the $2$-fold covering of the Klein bottle by $T^2$. The compact manifold (with torus boundary) $K$ is also recognized as the compact regular neighborhood of any embedded one-sided Klein bottle in any $3$-manifold. 

We will refer to such a decomposition of prime manifolds as the \emph{geometric decomposition}. The only place that it differs from applying the JSJ decomposition to all prime factors is that some geometric prime $3$-manifolds (e.g. of $Sol$ geometry) admit nontrivial JSJ decomposition, but we do not cut them into smaller pieces in the geometric decomposition.

Furthermore, five ($\mathbb{S}^3, \mathbb{E}^3, \mathbb{S}^2\times\mathbb{E}$, $Nil$, $Sol$) out of the eight geometries have no non-cocompact lattice \cite[Theorem 4.7.10]{Thurstonbook}. It follows that pieces obtained by a nontrivial geometric decomposition are either $K$ or have one of the other three geometries in the interior: $\Hbb^3$, $\Hbb^2\times\Ebb$ or $\PSLtwotilde\R$.

If a compact manifold $M$ with (possibly empty) tori boundary has $\Hbb^2\times\Ebb$ or $\PSLtwotilde\R$ geometry in the interior with finite volume, then it is \emph{Seifert fibered} over a $2$-dimensional compact \emph{orbifold} $B$ whose orbifold Euler characteristic $\chi_o(B)<0$; See \cite[Theorems 4.13 and 4.15]{Scott} and \cite[Corollary 4.7.3]{Thurstonbook}. We will introduce basic properties of Seifert fibered spaces and orbifolds in Subsection \ref{subsec: SFS}.

We summarize the discussion above as the following theorem.

\begin{thm}\label{thm: 3-mfd decomp}
	There is a unique geometric decomposition of every prime $3$-manifold $M$, which is trivial if $M$ already admits one of the eight geometries and is the JSJ decomposition if $M$ is non-geometric. Moreover, if $M$ is non-geometric, then each piece $N$ in its geometric decomposition has incompressible tori boundary, and exactly one of the following three cases occurs:
	\begin{enumerate}
		\item $N$ has hyperbolic geometry of finite volume in the interior,
		\item $N$ is Seifert fibered over a compact orbifold $B$ with $\chi_o(B)<0$, or
		\item $N$ is homeomorphic to the twisted $I$-bundle $K$ over the Klein bottle.
	\end{enumerate}
	The minimality of the JSJ decomposition implies that the fiber directions disagree when two boundary components from Seifert fibered pieces are glued. $M$ being non-geometric implies that no two pieces homeomorphic to $K$ are glued together.
\end{thm}

Since the tori in the geometric decomposition are incompressible, the fundamental group $\pi_1(M)$ of any non-geometric prime manifold $M$ has the structure of a graph of groups. We refer to the corresponding tree that $\pi_1(M)$ acts on as the \emph{JSJ-tree}. By basic hyperbolic geometry and the minimality of JSJ decomposition, it is noticed by Wilton--Zalesskii \cite[Lemma 2.4]{Acyl} that the action on the JSJ-tree has a nice \emph{acylindricity} property.

\begin{defn}
	For any $K\ge1$, an action of $G$ on a tree is \emph{$K$-acylindrical} if the fixed point set of any $g\neq id\in G$ has diameter at most $K$.
\end{defn}

\begin{lemma}[Wilton--Zalesskii]\label{lemma: acyl JSJ}
	The action of any non-geometric prime $3$-manifold group on its JSJ-tree is $4$-acylindrical.
\end{lemma}

\subsection{$2$-dimensional orbifolds and Seifert fibered $3$-manifolds}\label{subsec: SFS}
For our purpose, a compact $2$-dimensional (cone-type) \emph{orbifold} $B$ is a compact possibly nonorientable surface with finitely many so-called \emph{cone points} in its interior. Each cone point has an order $n\ge 2$, meaning that locally it is modeled on the quotient of a round disk by a rotation of angle $2\pi/n$. The orbifold Euler characteristic $\chi_o(B)$ is the Euler characteristic of the surface minus $\sum (1-1/n_i)$, where the sum is taken over all cone points and $n_i$ is the corresponding order. 

Only those orbifolds $B$ with negative $\chi_o(B)$ will appear in our discussion, all of which have hyperbolic structures of finite volume realizing each boundary as a cusp \cite{davis1984}. In this case, $B$ can be thought of as $\mathbb{H}^2/\Gamma$ , where $\Gamma$ is a discrete subgroup of $\mathrm{Isom}(\mathbb{H}^2)$ isomorphic to the \emph{orbifold fundamental group} $\pi_1(B)$. It follows that each element in $\pi_1(B)$ of finite order acts by an elliptic isometry such that its order divides the order of some cone point, an element acts by parabolic isometry if and only if its conjugacy class represents a loop on the boundary, and all other elements act by hyperbolic isometries (possibly composed with a reflection about the axis of translation). It also follows that a finite orbifold cover of $B$ is a hyperbolic surface, and thus $\pi_1(B)$ is word-hyperbolic.

For any compact $2$-dimensional orbifold $B$ with $\chi_o(B)<0$, it is known that $H^k(\pi_1(B);\R)=0$ for all $k>2$, 
and that $H^2(\pi_1(B);\R)\neq 0$ if and only if $B$ is closed and orientable, in which case a generator is given by the Euler class $\eurm(B)$ associated to the $\pi_1(B)$ action on the circle at infinity of the hyperbolic plane (for any fixed hyperbolic structure on $B$).


A compact $3$-manifold $M$ with (possibly empty) tori boundary is \emph{Seifert fibered} over an orbifold $B$ if there is a projection $p:M\to B$ such that 
\begin{enumerate}
	\item each fiber $p^{-1}(b)$ is $S^1$,
	\item it is an $S^1$ bundle away from the preimage of cone points of $B$, and
	\item for each cone point $b$ of order $n$, a neighborhood of $p^{-1}(b)$ is obtained by gluing the bottom of a solid cylinder to its top by a rotation of angle $2m\pi/n$ for some $m$ coprime to $n$.
\end{enumerate}

A fiber $p^{-1}(b)$ is \emph{regular} if $b$ is not a cone point. If $b$ is a cone point of order $n$, we say the fiber $p^{-1}(b)$ has \emph{multiplicity} $n$. 
When $\chi_o(B)<0$, $M$ is aspherical and thus $\pi_1(M)$ is torsion-free \cite[Lemma 3.1]{Scott}, and moreover $\pi_1(M)$ fits into a short exact sequence
\begin{equation}\label{eqn:exact sequence}
1 \to \Z \to \pi_1(M) \overset{p}{\to} \pi_1(B) \to 1,
\end{equation}
where the normal $\Z$ subgroup is generated by \emph{any} regular fiber \cite[Lemma 3.2]{Scott}. This is a central extension if the bundle is orientable, or equivalently $B$ is orientable since $M$ is orientable. Note that $\partial M$ is exactly the preimage of $\partial B$, and thus the $\Z$ subgroup lies in every peripheral $\Z^2$ subgroup of $\pi_1(M)$ if $\partial M\neq \emptyset$.

We will frequently use the map $p$ and monotonicity of scl in our estimates. In many cases, $p$ actually preserves scl.
\begin{lemma}[Calegari]\label{lemma: p preserves scl}
	Let $M$ be a $3$-manifold Seifert fibered over an orbifold $B$ with $\chi_o(B)<0$. Let $p:\pi_1(M)\to\pi_1(B)$ be induced by the projection and let $z\in\pi_1(M)$ represent a regular fiber. Then $\scl_M(g)=\scl_B(p(g))$ for all null-homologous $g\in\pi_1(M)$ if $H^2(\pi_1(B);\R)=0$.
\end{lemma}
\begin{proof}
	This is the proof of \cite[Proposition 4.30]{Cal:sclbook} applied to the exact sequence (\ref{eqn:exact sequence}).
\end{proof}

When the base space $B$ is closed and orientable, we need the following two lemmas as additional tools.
\begin{lemma}\label{lemma: distinguish two geometries}
	Let $M$ be a closed Seifert fibered $3$-manifold with orientable base orbifold $B$ so that $\chi_o(B)<0$. Let $z\in\pi_1M$ represent a regular fiber. 
	Then $M$ has $\PSLtwotilde\R$ or $\Hbb^2\times\Ebb$ geometry, and the former case arises iff $[z]=0\in H_1(M;\R)$.
\end{lemma}
\begin{proof}
	We know $M$ admits $\PSLtwotilde\R$ or $\Hbb^2\times\Ebb$ geometry by our discussion in Section \ref{subsec: decomp of 3 mfd}; See also \cite{Scott}. 
	We give a dumb but direct way to check the second assertion: $M$ has $\PSLtwotilde\R$ geometry if and only if the Euler number of the Seifert fibration is nonzero \cite{Scott}, 
	and it is straightforward from the presentation of $\pi_1 M$ (see \cite[Section 5.3]{Orlik:Seifert}) that this is also equivalent to $[z]=0\in H_1(M;\R)$, 
	using the formula in \cite[Page 437]{Scott} relating the Euler number to the Seifert invariants.
\end{proof}

\begin{lemma}\label{lemma: additional rot}
	Let $M$ be a closed $3$-manifold admitting $\PSLtwotilde\R$ geometry.
	Suppose the Seifert fibration of $M$ is over an orientable orbifold $B$. 
	Let $p:\pi_1(M)\to\pi_1(B)$ be induced by the projection and let $z\in\pi_1(M)$ represent a regular fiber, which is rationally null-homologous. 
	Pulling back the standard $\PSLtwotilde\R$ action on $\R$ (covering the $\mathrm{PSL}_2\R$ action on $S^1=\partial \Hbb^2$) by the representation $\pi_1 M\to \PSLtwotilde\R$, 
	we obtain an action $\rho:\pi_1 M\to \Homeo^+_\Z(\R)$ covering the $\pi_1(B)$ action on $S^1=\partial \Hbb^2$.
	Then the homogeneous quasimorphism $\phi\defeq \rho^*\rot \in \mathcal{Q}(\pi_1(M))$ has the following properties:
	\begin{enumerate}
		\item the defect $D(\phi)\le 1$, 
		\item $\phi(z)\in\Z\setminus\{ 0\}$,
		\item $\phi(g)\in \Z$ whenever $p(g)$ is of infinite order, and
		\item $\mathcal{Q}(\pi_1(M))/H^1(\pi_1(M))$ is spanned by $\phi$ and the image of $\mathcal{Q}(\pi_1(B))/H^1(\pi_1(B))$ under $p^*$.
	\end{enumerate}
\end{lemma}
\begin{proof}
	Let $Z=\langle z\rangle$. We have the following commutative diagram, where $i:\pi_1(M)\to\PSLtwotilde\R$ is the representation endowing $M$ with $\PSLtwotilde\R$ geometry,
	and both rows are central extensions with the first one identical to (\ref{eqn:exact sequence}).
	$$
	\begin{CD}
		1 	@>>> 	Z			@>>>	\pi_1 M	@>p>>	\pi_1 B		@>>>		1\\
		@.			@VVV		@ViVV			@V\bar{i}VV @.\\
		1 	@>>> 	\Z			@>>>	\PSLtwotilde\R	@>\pi>>	\PSLtwo\R		@>>>		1,\\
	\end{CD}
	$$
	Let $\rho_0: \PSLtwotilde\R\to\Homeo^+_\Z(\R)$ be the action on $\R$ lifting the standard action $\bar{\rho}_0:\PSLtwo\R\to \Homeo^+(S^1)$,
	where we identify $S^1$ with $\partial\Hbb^2$.
	Then the action $\rho$ is by definition $\rho=\rho_0\circ i$, which is a lift of $\bar{\rho}_0 \bar{i}p=\bar{\rho}_0 \pi i: \pi_1 M\to\Homeo^+(S^1)$.
	The first desired property follows from the fact that $D(\rot)=1$ (Theorem \ref{thm:rot homogeneous qm classical}) and $D(\phi)=D(\rho^*\rot)\le D(\rot)$.
	
	Let $z_0$ be the generator of the central $\Z$ subgroup of $\PSLtwotilde\R$ that acts on $\R$ by translation $+1$.
	We have $i(z)=nz_0$ for some $n\neq0\in\Z$ since $i$ is faithful, so $\phi(z)=\rot(\rho(z))=n\rot(\rho_0(z_0))=n\in\Z\setminus\{0\}$, verifying the second property.
	
	For any $g\in\pi_1(M)$, the projection $p(g)$ is either hyperbolic (and of infinite order) or elliptic (and of finite order), as $M$ (and thus $B$) is closed.
	Since any hyperbolic isometry has fixed points on $\partial\Hbb^2$, we know $\phi(g)=\rot(\rho(g))\in\Z$, proving the third property.
	
	Finally, we prove the last property using a diagram chase originated from the proof of \cite[Proposition 4.30]{Cal:sclbook}. 
	We have the following diagram with exact rows and columns (Section \ref{subsec:circle actions and euler class}), where all (bounded) cohomology groups are in $\R$ coefficients: 
	\begin{center}
		\begin{tikzcd}
		&																						& 0 \arrow[d] 													& \\
		0 \arrow[r]	&\mathcal{Q}(\pi_1 B)/H^1(\pi_1 B) \arrow[r, "\delta"] \arrow[d, "p^*"] 	& H^2_b(\pi_1 B) \arrow[d, "p^*", "\cong" swap] \arrow[r, "c"]	& H^2(\pi_1 B) \arrow[d, "p^*"]\\
		0 \arrow[r]	&\mathcal{Q}(\pi_1 M)/H^1(\pi_1 M) \arrow[r, "\delta"] 						& H^2_b(\pi_1 M) \arrow[d] \arrow[r]							& H^2(M)\\
		&																						& H^2_b(\Z)=0
		\end{tikzcd}
	\end{center}
	Let $\eurm(B)\defeq \eurm(\bar{\rho}_0\bar{i})$ (resp. $\eurm_b(B)\defeq \eurm_b(\bar{\rho}_0\bar{i})$) be the (resp. bounded) Euler class associated to the $\pi_1(B)$ action on $\partial \mathbb{H}^2$. 
	Then $\delta\phi=\delta \rho^*\rot= p^* \eurm_b(B)$ by equation (\ref{eqn: rot and eu}).
	
	For any $f \in \mathcal{Q}(\pi_1 M)/H^1(\pi_1 M)$, there is some $\sigma\in H^2_b(\pi_1 B)$ such that $p^*\sigma=\delta f$. 
	Since $\eurm(B)$ generates $H^2(\pi_1 B)$ and $c(\eurm_b(B))=\eurm(B)$, we can write $\sigma=\delta\psi+t\eurm_b(B)$ for some $t\in\R$ and $\psi\in \mathcal{Q}(\pi_1 B)/H^1(\pi_1 B)$. Thus
	$$\delta(f-p^*\psi-t\phi)=p^*(\delta\psi+t\eurm_b(B))-\delta p^*\psi-tp^*\eurm_b(B)=0.$$
	This shows that $f=p^*\psi-t\phi$ and completes the proof.
\end{proof}

See \cite{Scott} for a more detailed introduction to orbifolds, Seifert fibered spaces, and their relation to the eight geometries.

\subsection{Gaps from hyperbolicity}
We will need a few tools for our estimates. The first is the following spectral gap theorem for word-hyperbolic groups, which is a corollary of \cite[Theorem A$'$]{CF:sclhypgrp}.
\begin{thm}[Calegari--Fujiwara]\label{thm: scl rel gap for hyp grp}
	Let $G$ be a $\delta$-hyperbolic group with a finite generating set $S$. Fix finitely many group elements $\{a_i\}$. Then there is a constant $C=C(\delta,|S|,\{a_i\})$ such that, for any $g\in G$ satisfying
	\begin{enumerate}
		\item there is no $n\ge1$ such that $g^n$ is conjugate to $g^{-n}$ in $G$; and
		\item there is no $m,n\neq 0$ and index $i$ such that $g^n$ is conjugate to $a_i^m$,
	\end{enumerate} 
	we have $\scl_{(G,\mathcal{A})}(g)\ge C$, where $\mathcal{A}$ is the collection of cyclic groups $\langle a_i \rangle$.
\end{thm}
\begin{proof}
	Since $\{a_i\}$ is a finite collection, we automatically have a uniform bound $T$ on their translation lengths. Thus this follows from Theorem \ref{thm: relative gap restate} (\ref{item: rel gap restate 1}).
\end{proof}

Their technique also has applications to groups acting on hyperbolic spaces. The special case of action on trees is carried out carefully by Clay--Forester--Louwsma \cite{CFL16} to make the estimate explicit. We will use the following theorem from \cite[Theorem 6.11]{CFL16}. Note the potential confusion that a $K$-acylindrical action by our definition is $(K+1)$-acylindrical in \cite{CFL16}, and the statement below has been tailored for our use.

\begin{thm}[Clay--Forester--Louwsma]\label{thm: acyl action gap}
	Suppose $G$ acts $K$-acylindrically on a tree and let $N$ be the smallest integer greater than or equal to $(K+3)/2$. If $g\in G$ is hyperbolic, then either $\scl_G(g)\ge 1/12N$ or $\scl_G(g)=0$. Moreover, the latter case occurs if and only if $g$ is conjugate to $g^{-1}$.
\end{thm}

We also need a strengthened version of Calegari's gap theorem for hyperbolic manifolds (not necessarily closed) \cite[Theorem C]{Cal:sclhypmfd}. 
The only difference is that we deal with scl \emph{relative} to the peripheral subgroups.

\begin{thm}[Calegari]\label{thm: hyp rel gap}
	Let $M$ be a complete hyperbolic manifold of dimension $m$. Then for any $\kappa>0$, there is a constant $\delta(\kappa,m)>0$ such that for any hyperbolic element $g\in \pi_1(M)$ with $\scl_{(M,\partial M)}(g)<\delta$, the unique geodesic loop $\gamma$ representing $g$ has hyperbolic length no more than $\kappa$.
\end{thm}
\begin{proof}
	The proof of the original theorem \cite[Theorem C]{Cal:sclhypmfd} works without much change. We briefly go through it for completeness.
	
	Let $S$ be a relative admissible surface for $g$ of degree $n$. By simplifying $S$, we may assume $S$ to be a \emph{pleated surface} and $-\chi(S)/n=-\chi^-(S)/n$. That is, $S$ is a hyperbolic surface of finite volume with geodesic boundary, the map $f:S\to M$ takes cusps into cusps and preserves the lengths of all rectifiable curves, and each point $p\in S$ is in the interior of a straight line segment which is mapped by $f$ to a straight line segment. The nice properties that we will use are $\area(S)=-2\pi\chi(S)$ and that $f$ preserves lengths, in particular $\len(\partial S)=n\cdot \len(\gamma)$.
	
	Choose $\epsilon$ small compared to the $2$-dimensional Margulis constant and $\len(\gamma)$, and take the thin-thick decomposition of $DS$, where $DS$ is the double of $S$ along its geodesic boundaries. Let $D\Sthick$ (resp. $D\Sthin$) be the part with injectivity radius $\ge 2\epsilon$ (resp. $<2\epsilon$), and let $\Sthick$ (resp. $\Sthin$) be $S\cap D\Sthick$ (resp. $S\cap D\Sthin$). Then $\Sthin$ is a union of cusp neighborhoods, open embedded annuli around short simple geodesic loops, and open embedded rectangles between pairs of geodesic segments of $\partial S$ which are distance $<\epsilon$ apart at every point. Each such a rectangle doubles to an open annulus in $D\Sthin$. Let $r$ be the number of rectangles and $s$ be the number of annuli in $\Sthin$, then there are $2s+r$ annuli in $D\Sthin$, which are disjoint and non-isotopic. Hence $2s+r$ is no more than the maximal number of disjoint non-isotopic simple closed curves on $DS$, and thus
	$$r\le 2s+r\le -\frac{3}{2}\chi(DS)=-3\chi(S).$$
	
	For convenience, add the cusp neighborhoods and open annuli of $\Sthin$ back to $\Sthick$ so that $\Sthin$ consists of $r$ thin rectangles. By definition, the $\epsilon/2$-neighborhood of the geodesic boundary in $\Sthick$ is \emph{embedded}, and thus 
	$$-2\pi\chi(S)=\area(S)\ge \area(\Sthick)\ge \frac{\epsilon}{2}\len(\partial S\cap \Sthick).$$
	Since each component of $\Sthin$ intersects $\partial S$ in two components, there are at most $-6\chi(S)$ components of $\partial S\cap \Sthin$, while their total length 
	$$\len(\partial S\cap \Sthin)=\len(\partial S)-\len(\partial S\cap\Sthick)\ge n\cdot\len(\gamma)-\frac{-4\pi\chi(S)}{\epsilon}.$$
	Therefore, at least one component $\sigma$ of $\partial S\cap \Sthin$ satisfies 
	$$\len(\sigma)\ge \frac{n\cdot \len(\gamma)-\frac{-4\pi\chi(S)}{\epsilon}}{-6\chi(S)}=\frac{n\cdot\len(\gamma)}{-6\chi(S)}-\frac{2\pi}{3\epsilon}.$$
	
	On the other hand, $\len(\sigma)$ cannot be much longer than $\len(\gamma)$, otherwise we will have two long anti-aligned geodesic segments on $\gamma$, which according to the Margulis lemma would violate either the discreteness of $\pi_1(M)$ or the fact that there are no order-$2$ elements in $\pi_1(M)$. More precisely, choosing $\epsilon$ small in the beginning so that $4\epsilon$ is less than the $m$-dimensional Margulis constant, it is shown in the original proof of \cite[Theorem C]{Cal:sclhypmfd} that 
	$$\len(\sigma)\le 2\cdot\len(\gamma)+4\epsilon.$$
	Combining this inequality with our earlier estimates, we have
	$$2\cdot\len(\gamma)+4\epsilon\ge \frac{n\cdot\len(\gamma)}{-6\chi(S)}-\frac{2\pi}{3\epsilon},$$
	or equivalently
	$$\left(\frac{n}{-6\chi^-(S)}-2\right)\len(\gamma)\le 4\epsilon+\frac{2\pi}{3\epsilon}.$$
	The result follows since $\epsilon$ is a constant depending only on the dimension $m$.
\end{proof}

From the proof above, we deduce the following analog of \cite[Theorem D]{Cal:sclhypmfd} (and the claims made in \cite[Remark 5.6]{Cal:sclhypmfd}).
\begin{thm}\label{thm: accumulation}
	Let $M$ be a complete hyperbolic manifold of dimension $m$. Then for any $0<\delta<1/24$, there is a constant $\kappa(\delta,m)>0$ 
	such that for any hyperbolic element $g\in \pi_1(M)$ with $\scl_{(M,\partial M)}(g)<\delta$, 
	the unique geodesic loop $\gamma$ representing $g$ has hyperbolic length no more than $\kappa$.
\end{thm}
\begin{proof}
	Following the same proof as above, for a small constant $\epsilon$ depending on the dimension $m$,
	if $\scl_{(M,\partial M)}(g)<\delta$, then there is a relative admissible surface $S$ of some degree $n$ that is pleated,
	such that $-\chi^-(S)/2n < \delta$ and	
	$$\left(\frac{n}{-6\chi^-(S)}-2\right)\len(\gamma)\le 4\epsilon+\frac{2\pi}{3\epsilon}.$$
	It follows that
	$$\left(\frac{1}{12\delta}-2\right)\len(\gamma)\le 4\epsilon+\frac{2\pi}{3\epsilon}.$$
	Since $\frac{1}{12\delta}>2$ by assumption, and the right-hand side of the equation above only depends on $m$,
	we obtain an upper bound $\kappa$ of $\len(\gamma)$ relying only on $\delta$ and $m$ as desired.
\end{proof}

\subsection{Estimates of scl in $3$-manifolds}

The main theorem of this section is the following gap theorem.
\begin{thm}\label{thm:3mfdgap}
	For any $3$-manifold $M$, there is a constant $C=C(M)>0$ such that for any $g\in\pi_1(M)$, we have either $\scl_M(g)\ge C$ or $\scl_M(g)=0$.
\end{thm}

The size of $C$ does depend on $M$, but we will describe elements with scl less than $1/48$ and classify those with vanishing scl in Theorem \ref{thm: small scl in 3mfds}. 

As for concrete examples of elements with small scl, one can perform Dehn fillings on a knot complement to produce a sequence of closed hyperbolic $3$-manifolds and loops with hyperbolic lengths and stable commutator lengths both converging to $0$; See \cite[Example 2.4]{CF:sclhypgrp}. Here we give a different example among the so-called graph manifolds.

\begin{exmp}\label{exmp: small scl in 3mfd, SFS surgery}
	Consider the manifolds $M_{p,q}$ in Example \ref{exmp: surgery example} with $p=1$ and $q\in\Z_+$. The image of the loop $\tau_A$ has positive scl $1/q$ converging to $0$ as $q$ goes to infinity. In this example, the geometric decomposition is obtained by cutting along the image of the torus $T_A$, where the resulting manifolds $X_A$ and $X_B$ are both trivially Seifert fibered and admit $\mathbb{H}^2\times \mathbb{E}$ geometry of finite volume in the interior.
\end{exmp}

To show Theorem \ref{thm:3mfdgap}, by Lemma \ref{lemma: freeprod preserves gap}, it suffices to estimate scl of loops in each prime factor. 
We first focus on geometric prime manifolds and then deal with non-geometric ones.

\subsubsection{The geometric cases}
For geometric prime manifolds, we first focus on the harder case of those with $\Hbb^2\times\Ebb$ or $\PSLtwotilde\R$ geometry.
\begin{thm}\label{thm: closed Seifert fibered prime factors}
	Let $M$ be a prime $3$-manifold with $\Hbb^2\times\Ebb$ or $\PSLtwotilde\R$ geometry. So $M$ is Seifert fibered over some closed orbifold $B$ with $\chi_o(B)<0$. Then there is a constant $C=C(M)>0$ such that either $\scl_M(g)\ge C$ or $\scl_M(g)=0$ for any $g\in \pi_1(M)$. Moreover, $\scl_M(g)=0$ if and only if
	\begin{enumerate}
		\item either $B$ is nonorientable and $g^n h g^n h^{-1}$ represents a multiple of the regular fiber for some $h\in \pi_1(M)$ and $n\in \Z_+$,
		\item or $g^n h g^n h^{-1}=id$ for some $h\in \pi_1(M)$ and $n\in \Z_+$.
	\end{enumerate}
\end{thm}
\begin{proof}
	Let $z\in\pi_1(M)$ represent a regular fiber, and let $p$ be the projection in the short exact sequence (\ref{eqn:exact sequence}). By Theorem \ref{thm: relative gap restate} and the fact that $\pi_1(B)$ is word-hyperbolic, there is some $C(B)>0$ such that either $\scl_B(\bar{g})\ge C(B)$ or $\scl_B(\bar{g})=0$ for any $\bar{g}\in\pi_1(B)$. Hence by monotonicity, $\scl_M(g)\ge\scl_B(p(g))\ge C(B)$ unless $\scl_B(p(g))=0$. Moreover, the latter case occurs if and only if there is some $n\in\Z_+, m\in\Z$ and $h\in\pi_1(M)$ such that $g^n hg^n h^{-1}=z^m$. Thus it remains to analyze $\scl_M(g)$ for such exceptional $g$ satisfying $g^n hg^n h^{-1}=z^m$.
	
	If $B$ is nonorientable, then $H^2(\pi_1(B);\R)=0$ and $p$ preserves scl by Lemma \ref{lemma: p preserves scl}. Moreover, the bundle must be nonorientable and $z$ is conjugate to $z^{-1}$. 
	Thus our equation implies that $[g]=m[z]/2n=0\in H_1(M;\R)$. So we have $\scl_M(g)=\scl_B(p(g))=0$ for all such exceptional $g$.
	
	Now suppose $B$ is orientable. Consider two cases:
	\begin{enumerate}
		\item Suppose $[z]\neq 0\in H_1(M;\R)$, i.e. $M$ admits $\Hbb^2\times\Ebb$ geometry by Lemma \ref{lemma: distinguish two geometries}. Then $g^n hg^n h^{-1}=z^m$ implies $2n[g]=m[z]\in H_1(M;\R)$. If $m=0$, then $g^n$ is conjugate to $g^{-n}$ and thus $\scl_M(g)=0$. If $m\neq 0$, then $[g]\neq 0\in H_1(M;\R)$ and thus $\scl_M(g)=\infty$.
		\item Suppose $[z]=0\in H_1(M;\R)$, i.e. $M$ admits $\PSLtwotilde\R$ geometry. Let $\phi=\rho^*\rot\in\mathcal{Q}(\pi_1 M)$ be the homogeneous quasimorphism from Lemma \ref{lemma: additional rot}, where $\rho$ is the $\pi_1 M$ action on $\R$ covering the $\pi_1 B$ action on $S^1=\partial \Hbb^2$. 
		For any exceptional $g$ satisfying $g^n hg^n h^{-1}=z^m$, or equivalently $g^n=hg^{-n}h^{-1}z^m$, since $z$ is central, using Proposition \ref{prop: split hqm on comm elem} we have 
		$$n\phi(g)=\phi(g^n)=\phi(hg^{-n}h^{-1})+\phi(z^m)=-n\phi(g)+m\phi(z).$$
		As $\phi(z)\neq0$, we have $\phi(g)=0$ if and only if $m=0$, i.e. $g^n hg^n h^{-1}=id$, in which case $\scl_M(g)=0$ as shown in the previous case. 
		
		Now suppose $\phi(g)\neq0$, i.e. $m\neq 0$.
		\begin{enumerate}
			\item If $p(g)$ is hyperbolic, it acts on $\partial\Hbb^2$ with fixed points, 
			so $\phi(g)\in\Z$ by Lemma \ref{lemma: additional rot}. Thus $\phi(g)\in\Z\setminus\{0\}$. As $D(\phi)\le1$ by Lemma \ref{lemma: additional rot}, Bavard's duality implies
			$$\scl_M(g)\ge \frac{|\phi(g)|}{2D(\phi)}\ge \frac{1}{2}.$$
			\item If $p(g)$ is elliptic, then it has finite order $n\le N$, where $N$ is the largest order of cone points on $B$.
			Thus $g^n=z^m$ for some $m\neq0\in\Z$ (as $\pi_1 M$ is torsion-free). 
			Applying $\phi$, we obtain $\phi(g)=m\phi(z)/n$ and by Bavard's duality
			$$\scl_M(g)\ge \frac{|m||\phi(z)|}{2nD(\phi)}\ge \frac{|m|}{2N}\ge \frac{1}{2N}$$
			as $\phi(z)\in\Z\setminus\{0\}$ by Lemma \ref{lemma: additional rot}.
		\end{enumerate}
	\end{enumerate}
	In summary, for
	\[
	C(M)\defeq\left\{ 
	\begin{array}{ll}
		\min\left\{C(B),\frac{1}{2N}\right\} 	& \text{if }B\text{ is orientable and }[z]=0\in H_1(M;\R),\\
		C(B)						& \text{otherwise,}
	\end{array}
	\right.
	\]
	we have either $\scl_M(g)\ge C(M)$ or $\scl_M(g)=0$ for all $g\in\pi_1(M)$, and the latter case occurs exactly in the asserted scenarios.
\end{proof}

We summarize our result for the geometric cases in the theorem below.
\begin{thm}\label{thm: Gap for geometric manifolds}
	Let $M$ be a closed prime $3$-manifold that admits one of the eight geometries. Then there is $C(M)>0$ such that for all $g\in\pi_1(M)$, either $\scl_M(g)\ge C(M)$ or $\scl_M(g)=0$.
\end{thm}
\begin{proof}
	If $M$ is hyperbolic, then the result follows from Theorem \ref{thm: hyp rel gap} by choosing $\kappa$ less than the injectivity radius of $M$ and $C(M)=\delta(\kappa,3)$. In this case, $\scl_M(g)=0$ only holds for $g=id$ as $\pi_1M$ is torsion-free hyperbolic or by Theorem \ref{thm: hyp rel gap}.
	
	If $M$ has $\Hbb^2\times\Ebb$ or $\PSLtwotilde\R$ geometry, this is shown by Theorem \ref{thm: closed Seifert fibered prime factors} above. 
	For the other five geometries, $\pi_1 M$ is amenable and $\scl_M\equiv 0$, so the result holds trivially for any $C(M)>0$.
\end{proof}

\subsubsection{The non-geometric cases}
Now we consider prime manifolds with nontrivial geometric decomposition.
We will treat the fundamental group as a graph of groups with $\Z^2$ edge groups according to the geometric decomposition. 
We first establish scl gaps of vertex groups relative to adjacent edge groups, which would help us estimate $\scl_M$ in the vertex groups using Lemma \ref{lemma: basic prop of rel scl}.

\begin{lemma}\label{lemma:vertex group hyp case}
	Let $M$ be a compact $3$-manifold with tori boundary (possibly empty). Suppose the interior of $M$ is hyperbolic with finite volume. Then $\pi_1(M)$ has a strong spectral gap relative to the peripheral subgroups. Moreover, $\scl_{(M,\partial M)}(g)=0$ occurs if and only if $g$ is conjugate into some peripheral subgroup.
\end{lemma}
\begin{proof}
	If $g$ is conjugate into some peripheral subgroup, then obviously $\scl_{(M,\partial M)}(g)=0$. Assume that this is not the case, then $g$ is a hyperbolic element. 
	In the thin-thick decomposition, there are only finitely many neighborhoods of short loops in the thin part as the interior of $M$ has finite volume.
	So we may choose a constant $\kappa=\kappa(M)>0$ so that any geodesic loop in $M$ has length at least $\kappa$. 
	Then there is a constant $\delta>0$ by Theorem \ref{thm: hyp rel gap} such that $\scl_{(M,\partial M)}(g)\ge \delta$ for all such $g$.
\end{proof}

A similar result holds for Seifert fibered manifolds. Here we focus on those with nonempty boundary. See Theorem \ref{thm: closed Seifert fibered prime factors} for the case of closed Seifert fibered manifolds.

\begin{lemma}\label{lemma:vertex group Seifert case}
	Let $M$ be a compact $3$-manifold with \emph{nonempty} tori boundary, Seifert fibered over an orbifold $B$ with $\chi_o(B)<0$. 
	Then $\pi_1 M$ has a spectral gap relative to the peripheral $\Z^2$-subgroups. 
	Moreover, $\scl_{(M,\partial M)}(g)=0$ if and only if one of the following cases occurs, where $z\in\pi_1 M$ represents a regular fiber:
	\begin{enumerate}
		\item $g$ is conjugate into some peripheral subgroup;
		\item $g^n=z^m$ for some $n>0$ no greater than the maximal order of cone points on $B$ and for some $m\in \Z\setminus\{0\}$; or
		\item There is $h\in \pi_1 M$ and $m\in\Z$ such that $ghgh^{-1}=z^m$ (so $p(g)$ is conjugate to $p(g^{-1})$ in $\pi_1(B)$), and $\scl_M(g^2-z^m)=0$, where $p:\pi_1(M)\to\pi_1(B)$ is the projection map.
	\end{enumerate}
\end{lemma}
\begin{proof}
	Since $B$ has nonempty boundary, Lemma \ref{lemma: p preserves scl} implies that $H^2(\pi_1(B);\R)=0$ and $p$ preserves scl. That is, for any $g\in\pi_1(M)$, either $\scl_M(g)=\infty$ or $\scl_M(g)=\scl_B(p(g))$.
	
	If $p(g)$ is of finite order $n$, then by the exact sequence (\ref{eqn:exact sequence}), we have $z^m=g^n$ for some integer $m$, where $n$ divides the order of some cone points on $B$ by the structure of $\pi_1(B)$. We have $\scl_{(M,\partial M)}(g)=0$ since $z$ is peripheral. Moreover $m\neq0$ since $\pi_1 M$ is torsion-free.
	
	Suppose $p(g)$ is of infinite order in the sequel. Recall that $\pi_1(B)$ is a lattice in $\PSLtwo\R$ and Gromov-hyperbolic. 
	Then $p(g)$ either acts as a parabolic isometry or acts as a (possibly orientation-reversing) hyperbolic isometry. 
	\begin{enumerate}
		\item If $p(g)$ is parabolic, then it is conjugate into a peripheral $\Z$-subgroup of $\pi_1(B)$, which implies that $g$ is conjugate into a peripheral $\Z^2$ subgroup of $\pi_1(M)$ and $\scl_{(M,\partial M)}(g)=0$.
		
		\item If $p(g)$ is hyperbolic and $\bar{h}p(g^n)\bar{h}^{-1}\neq p(g^{-n})$ for any $n\neq0$ and $\bar{h}\in\pi_1(B)$. Fix a generator $g_i$ for each peripheral subgroup of $\pi_1(B)$. Since $p(g)$ acts as a hyperbolic isometry, $hp(g^n)h^{-1}\neq g_i^m$ for any $m,n\neq0$ and any $h\in\pi_1(B)$. Then Theorem \ref{thm: scl rel gap for hyp grp} implies $\scl_{(B,\partial B)}(p(g))\ge C$ for a constant $C=C(B)$, where $\scl_{(B,\partial B)}$ denotes scl of $\pi_1(B)$ relative to its peripheral $\Z$-subgroups. This implies $\scl_{(M,\partial M)}(g)\ge\scl_{(B,\partial B)}(p(g))\ge C$ by monotonicity of relative scl. 
		
		\item If $p(g)$ is hyperbolic and $\bar{h}p(g^n)\bar{h}^{-1}= p(g^{-n})$ for some $n\ge1$ and $\bar{h}\in\pi_1(B)$. 
		Then $\bar{h}$ must interchange the two endpoints of the axis of $p(g)$, and thus $\bar{h}p(g)\bar{h}^{-1}=p(g^{-1})$. 
		This implies $hgh^{-1} g=z^m$ for some $h\in\pi_1(M)$ and $m\in\Z$. 
		Hence the chain $g-\frac{m}{2}z$ is null-homologous. Since $p$ preserves scl and $p(g)$ is conjugate to its inverse, we have $\scl_{M}(g-\frac{m}{2}z)=\scl_B(p(g))=0$. Since $z$ is peripheral, we conclude that $\scl_{(M,\partial M)}(g)=0$ in this case.
	\end{enumerate}
	
	Combining the cases above, any $g\in \pi_1(M)$ either has $\scl_{(M,\partial M)}(g)\ge C$ or $\scl_{(M,\partial M)}(g)=0$. Moreover, if $\scl_{(M,\partial M)}(g)=0$, then
	\begin{enumerate}
		\item either $p(g)$ is of finite order and $g^n= z^m$ for some $m\neq0$ and $n>0$ no greater than the maximal order of cone points on $B$;
		\item or $p(g)$ is of infinite order and 
		\begin{enumerate}
			\item $p(g)$ is parabolic, which implies that $g$ is conjugate into some peripheral subgroup; or
			\item $p(g)$ is hyperbolic, $ghgh^{-1}=z^m$ for some $h\in\pi_1 M$ and $m\in\Z$, and $\scl_M(g^2-z^m)=0$.
		\end{enumerate}
	\end{enumerate}
\end{proof}
\begin{rmk}\label{rmk: explicit gap size}
	In the proof above, the gap $C$ comes from the spectral gap of scl in $2$-orbifolds relative to nonempty boundary, which can be made uniform and explicit with $C=1/24$; See the old version of this paper \cite[Theorem 9.4]{CH:sclgapgog_old}, which may appear in a separate note. 
\end{rmk}

Next we control scl of integral chains in the edge spaces in the geometric decomposition of a prime manifold (Theorem \ref{thm: edge scl for non-geometric prime factors}). We need the following two lemmas and their corollaries.
\begin{lemma}\label{lemma:edge group for hyp}
	Let $M$ be a compact $3$-manifold with boundary consisting of tori $T$ and $T_i$, $i\in I$ ($I$ could be empty). Suppose the interior of $M$ is hyperbolic with finite volume. Then there exists $C>0$ such that $\scl_{(M,\{T_i\})}(g)\ge C$ for any $g \in\pi_1(T)\setminus\{id\}$.
\end{lemma}
\begin{proof}
	By the hyperbolic Dehn filling theorem \cite[Theorem 5.8.2]{Thurstonnotes}, we can fix two different Dehn fillings of the end $T$ such that the resulting manifolds $M_1$ and $M_2$ are both hyperbolic with cusp ends $T_i$, $i\in I$. For any $g\neq id \in\pi_1(T)$, it is a nontrivial hyperbolic element in at least one of $\pi_1(M_1)$ and $\pi_1(M_2)$. Thus the result follows from Lemma \ref{lemma:vertex group hyp case} and monotonicity.
\end{proof}
\begin{cor}\label{cor:edge group for hyp}
	Let $M$ be a compact $3$-manifold with boundary tori $T_i$ ($i\in I$). Suppose the interior of $M$ is hyperbolic with finite volume. Then for any chain $c$ of the form $\sum_{i\in I} t_ig_i$ with $t_i\in\R$ and $g_i \in\pi_1(T_i)\setminus\{id\}$, we have $\scl_M(c)>0$ unless $t_i=0$ for all $i\in I$.
\end{cor}
\begin{proof}
	Suppose $t_j\neq 0$ for some $j$, then Lemma \ref{lemma:edge group for hyp} provides a constant $C_j>0$ such that
	$$\scl_M(\sum t_ig_i)\ge|t_j|\cdot\scl_{(M,\{T_i\}_{i\neq j})}(g_j)\ge |t_j|C_j>0.$$
\end{proof}

\begin{lemma}\label{lemma:edge group for Seifert}
	Let $M$ be a compact $3$-manifold with boundary consisting of tori $T$ and $T_i$, $i\in I$ ($I$ could be empty). Suppose $M$ is Seifert fibered with bundle projection $p:M\to B$ where $B$ is an orbifold with $\chi_o(B)<0$. Then there exists $C>0$ such that $\scl_{(M,\{T_i\})}(g)\ge C$ for any $g\in\pi_1(T)\backslash\ker p$.
\end{lemma}
\begin{proof}
	We have the short exact sequence (\ref{eqn:exact sequence}). Consider any $g\in\pi_1(T)\backslash\ker p$. Fix any hyperbolic structure of $B$ realizing boundaries as cusps. Then $p(g)$ is a parabolic element. We know $hp(g^n)h^{-1}\neq p(g^{-n})$ for any $h\in\pi_1(B)$ and any $n\neq0$ since otherwise $h$ must be a hyperbolic reflection, which cannot appear in $\pi_1(B)$. Since different boundary components of $B$ cannot be homotopic and $\pi_1(B)$ is $\delta$-hyperbolic, Theorem \ref{thm: scl rel gap for hyp grp} implies the existence of $C>0$ such that $$\scl_{(M,\{T_i\})}(g)\ge\scl_{(B,\{p(T_i)\})}(p(g))\ge C$$ for all $g\in\pi_1(T)\backslash\ker p$.
\end{proof}
Again in the lemma above, we can take $C=1/24$ just as commented in Remark \ref{rmk: explicit gap size}.

\begin{cor}\label{cor:edge group for Seifert}
	Let $M$ be a compact $3$-manifold with nonempty boundary tori $T_i$ ($i\in I$). Suppose $M$ is Seifert fibered with bundle projection $p:M\to B$ where $B$ is an orbifold with $\chi_o(B)<0$. Then for any chain $c$ of the form $\sum_{i\in I} t_ig_i$ with $t_i\in\R$ not all zero and $g_i\neq id\in\pi_1(T_i)$, we have $\scl_M(c)=0$ if and only if $g_i\in\ker p$ for all $i\in I$ and $[c]=0\in H_1(M,\R)$.
\end{cor}
\begin{proof}
	If $g_j\notin\ker p$ for some $j$, then Lemma \ref{lemma:edge group for Seifert} provides a constant $C_j>0$ such that
	$$\scl_M(\sum t_ig_i)\ge|t_j|\cdot\scl_{(M,\{T_i\}_{i\neq j})}(g_j)\ge |t_j|C_j>0.$$
	
	Suppose $g_i\in\ker p$ for all $i\in I$. Note that $H^2(\pi_1(B);\R)=0$ and $p$ preserves scl by Lemma \ref{lemma: p preserves scl} since $B$ has boundary. Hence if $[c]=0\in H_1(M,\R)$, then $\scl_M(c)=\scl_B(p(c))=\scl_B(0)=0$.
\end{proof}

We are now ready to prove a gap result for $\scl_M(c)$, where $c$ is an integral chain supported in the edge groups.
In order to characterize those elements $c$ in the edge group with $\scl_M(g)=0$, it is convenient to introduce the following notion.
\begin{defn}\label{def: vanishing pair}
	We say $(g,N)$ is a \emph{vanishing pair} in a prime $3$-manifold $M$ if $N$ is a piece in the (nontrivial) geometric decomposition of $M$ and $g\in\pi_1 N$ is represented by a loop on a torus boundary of $N$ such that
	\begin{enumerate}
		\item $g=id$,
		\item $N=K$ is the twisted $I$ bundle over the Klein bottle and $g$ is null-homologous,\label{item: kill by one side K}
		\item or $N$ is Seifert fibered over base $B$ with $\chi_o(B)<0$, $g=z^m$ for some $m\in\Z$ and $z$ represents the regular fiber, such that (a) $B$ is nonorientable 
		or (b) at least one boundary component of $N$ is glued to a copy of the twisted $I$-bundle $K$ such that the pair $(z',K)$ satisfies case (\ref{item: kill by one side K}) above for the image $z'$ of $z$ in $\pi_1 K$.
	\end{enumerate}
\end{defn}

\begin{thm}\label{thm: edge scl for non-geometric prime factors}
	Let $M$ be a non-geometric prime $3$-manifold. Let $\mathcal{T}$ be the collection of tori in the JSJ decomposition of $M$. Then there is a constant $C_M>0$ such that, for any integral chain $c=\sum_{T\in\mathcal{T}} g_T$ with $g_T\in \pi_1(T)\cong \Z^2$, we have either $\scl_M(c)=0$ or $\scl_M(c)\ge C_M$. Moreover, if $c$ is a single loop supported in a JSJ torus $T$, then $\scl_M(c)=0$ if and only if $T$ identifies boundary components $\partial_1\subset N_1$ and $\partial_2\subset N_2$ of (possibly the same) pieces $N_1,N_2$, and $c=ab\in\pi_1(T)$, such that under the identification of $T$ with $\partial_1$ and $\partial_2$ respectively, the images $a',b'$ of $a,b$ in $\pi_1 N_1$ and $\pi_1 N_2$, respectively, satisfy:
	\begin{enumerate}
		\item either $(a',N_1)$ and $(b',N_2)$ are both vanishing pairs,
		\item or $N_1=N_2$ is Seifert fibered with regular fiber represented by $z$ so that $a'=z^m$ and $b'=z^{-m}$ for some $m\in\Z$.
	\end{enumerate}
\end{thm}
\begin{proof}
	As we mentioned earlier, the geometric decomposition endows $\pi_1(M)$ with the structure of a graph of groups, where the vertex groups are the fundamental groups of geometric pieces $\mathcal{N}$ and the edge groups are the fundamental groups of those tori $\mathcal{T}$ we cut along.
	
	For each $N\in \mathcal{N}$, let $V_N\defeq\bigoplus_{T\subset\partial N}H_1(T;\R)$ and equip it with the degenerate norm $\|(h_T)_{T\in \partial N}\|_N \defeq \scl_N(\sum_{T\subset \partial N} h_T)$ where $h_T\in H_1(T)$. Then $\bigoplus_{N\in\mathcal{N}} V_N$ is naturally equipped with $\|\cdot \|_1$, the $\ell^1$-product norm of all $\|\cdot \|_N$. As we observed in Section \ref{sec: scl vertex and edge group}, $\bigoplus_{N\in\mathcal{N}} V_N=\bigoplus_{T\in\mathcal{T}} [H_1(T;\R)\oplus H_1(T;\R)]$. Let $V_\mathcal{T}\defeq\bigoplus_{T\in\mathcal{T}} H_1(T;\R)$. 
	Piecing together the addition maps $H_1(T;\R)\oplus H_1(T;\R)\overset{+}{\to} H_1(T;\R)$ over all $T\in\mathcal{T}$, we get a projection $\pi:\bigoplus_{N\in\mathcal{N}} V_N\to V_\mathcal{T}$.
	Equips $V_\mathcal{T}$ with the quotient norm $\|\cdot\|$ of $\|\cdot \|_1$. 
	By Corollary \ref{cor: scl edge compute} and Remark \ref{rmk: adjust finite dimensional}, we have $\|(h_T)_{T\in\mathcal{T}}\|=\scl_M(\sum h_T)$ for any $h_T\in H_1(T)$, $T\in\mathcal{T}$. If $N$ is the twisted $I$-bundle over the Klein bottle, then a loop on its boundary has vanishing scl if and only if it is null-homologous since the fundamental group of the Klein bottle is virtually abelian. Combining this with Corollary \ref{cor:edge group for hyp} and Corollary \ref{cor:edge group for Seifert}, we note that the vanishing locus of $\|\cdot \|_N$ on $V_N$ is rational for each $N\in\mathcal{N}$. Thus the vanishing locus of $\|\cdot\|_1$ is also rational since it is the direct sum over all $N$ of the vanishing locus of $\|\cdot \|_N$ on $V_N$. Then its image under the projection $\pi$ is rational, which is exactly the vanishing locus of $\|\cdot\|$. Hence by Lemma \ref{lemma: auto gap of norm}, the desired constant $C_M$ exists.
	
	Suppose $c$ is a single loop in some edge space with $\scl_M(c)=0$. Then there is some $(v_N)\in \bigoplus_{N\in\mathcal{N}} V_N$ such that $\pi(v_N)=c$ and $\|v_N\|_N=0$ for all $N$. By Corollary \ref{cor:edge group for hyp}, we have $v_N=0$ for all hyperbolic pieces $N$. By Corollary \ref{cor:edge group for Seifert}, each $v_N$ is a sum of fibers on the boundary components of $N$ with $[v_N]=0\in H_1(N;\R)$ for each Seifert fibered piece $N$. Finally, $v_N$ is a null-homologous loop in $N$ if $N$ is the twisted $I$-bundle over the Klein bottle. Combining these facts, to annihilate the scl of a nontrivial loop, hyperbolic pieces make no contribution, and Seifert fibered pieces away from $T$ cannot contribute either by the minimality of JSJ decomposition. Also note that no two twisted $I$ bundles over the Klein bottle can be glued together as $M$ is non-geometric (see Theorem \ref{thm: 3-mfd decomp}). We obtain the desired classification of such loops $c$ via a case-by-case study.
\end{proof}

The size of $C_M$ is not explicit in Theorem \ref{thm: edge scl for non-geometric prime factors}. We notice from Example \ref{exmp: small scl in 3mfd, SFS surgery} that $C_M$ could be very small and depends on how the geometric pieces are glued together. We record the following special case for later use.

\begin{cor}\label{cor: fiber vanishing}
	Let $M$ be a non-geometric prime $3$-manifold. Suppose $N$ is a Seifert fibered geometric piece in $M$, and $z\in\pi_1 N\le\pi_1 M$ represents a regular fiber. If $\scl_M(z)=0$, then $(z,N)$ is a vanishing pair.
\end{cor}
\begin{proof}
	Represent $z$ by a loop supported on a JSJ torus $T$. We have a decomposition $z=ab$ as in Theorem \ref{thm: edge scl for non-geometric prime factors}.
	Without loss of generality $N=N_1$, then $a$ must be a power of $z$, and so is $b$.
	We cannot have $N_1=N_2=N$ since $z$ is already in the fiber direction and the JSJ decomposition is minimal.
	Minimality also implies that $N_2\neq N$ cannot be Seifert fibered unless $b=id$.
	So either $a=z$ or $N_2$ is the twisted $I$ bundle over a Klein bottle. In either case, $(z,N)$ must be a vanishing pair.
\end{proof}

Now we estimate $\scl_M(g)$ for elements $g$ in vertex groups.

\begin{thm}\label{thm: vertex scl for non-geometric prime factors}
	Let $M$ be a non-geometric prime $3$-manifold. Then for each geometric piece $N$ in the JSJ decomposition of $M$, there is a constant $C_N>0$ such that for any $g$ representing a loop in $N$, we have either $\scl_M(g)=0$ or $\scl_M(g)\ge C_N$.
\end{thm}
\begin{proof}
	Endow $\pi_1(M)$ with the structure of a graph of groups from the geometric decomposition, then $g$ is conjugate into the vertex group $\pi_1(N)$. The boundary of $N$ consists of a nonempty collection of tori. Let $C_M$ be the bound for integral chains supported in the edge groups from Theorem \ref{thm: edge scl for non-geometric prime factors}. There are three cases:
	\begin{enumerate}
		\item The interior of $N$ is hyperbolic with finite volume. If $g$ is conjugate into a peripheral subgroup of $\pi_1(N)$, then our control on edge groups shows that either $\scl_M(g)=0$ or $\scl_M(g)\ge C_M$. If $g$ is not conjugate into any peripheral subgroup of $\pi_1(N)$, then by Lemmas \ref{lemma: simple estimate of scl in vertex groups} and \ref{lemma:vertex group hyp case}, there exists $C>0$ such that $\scl_M(g)\ge\scl_{(N,\partial N)}(g)\ge C$ for all such $g$. Thus the conclusion holds with $C_N\defeq\min(C,C_M)$ in this case.
		
		\item $N$ is Seifert fibered over an orbifold $B$ such that $\chi_o(B)<0$. By Lemma \ref{lemma:vertex group Seifert case}, there is a constant $C=C(N)$ such that either $\scl_{(N,\partial N)}(g)\ge C$ or $\scl_{(N,\partial N)}(g)=0$. Moreover, $\scl_{(N,\partial N)}(g)=0$ only occurs in two cases:
		\begin{enumerate}
			\item either $g^n$ is conjugate into some peripheral subgroup of $\pi_1(N)$ for some $n>0$ not exceeding the maximal order $O_N$ of cone points on $B$, then by our control of $\scl_M$ on the edge groups, either $\scl_M(g)=0$ or $\scl_M(g)\ge C_M/O_N$;
			\item or $\scl_N(g^{2}-z^m)=0$ for some integer $m$, where $z$ is a generator of $\ker p\cong \Z$ and $p:\pi_1(N)\to \pi_1(B)$ is the projection map. In this case, by monotonicity, we have $\scl_M(g^{2}-z^m)=0$ and $\scl_M(g)=\frac{|m|}{2}\scl_M(z)$. If $m=0$ or $\scl_M(z)=0$, then $\scl_M(g)=0$; otherwise, $\scl_M(g)\ge \scl_M(z)/2\ge C_M/2$. 
		\end{enumerate} 
		In summary, we may choose $C_N\defeq \min(C,C_M/O_N,C_M/2)$ in this case.
		
		\item $N$ is homeomorphic to the regular neighborhood $K$ of a one-sided Klein bottle. Then $g^2$ is conjugate to an edge group element, and thus either $\scl_M(g)=0$ or $\scl_M(g)\ge C_N\defeq C_M/2$.
	\end{enumerate}
\end{proof}

Finally, we control $\scl_M(g)$ for hyperbolic elements $g$ using Theorem \ref{thm: acyl action gap}.
\begin{lemma}\label{lemma: scl gap of hyp elem in 3mfd}
	For any non-geometric prime $3$-manifold $M$ and any $g\in \pi_1(M)$ that is hyperbolic for the action on the JSJ-tree, either $\scl_M(g)\ge 1/48$ or $g$ is conjugate to its inverse, in which case $\scl_M(g)=0$.
\end{lemma}
\begin{proof}
	By Lemma \ref{lemma: acyl JSJ}, the action of $\pi_1(M)$ on the JSJ-tree is $4$-acylindrical, thus the result follows from Theorem \ref{thm: acyl action gap} with $K=4$ and $N=4$.
\end{proof}

The bound can be improved by Theorem \ref{thm: $n$-RTF gap, weak version} if certain pieces do not appear in the geometric decomposition. This is done by verifying the $3$-RTF condition using geometry.
\begin{lemma}\label{lemma:boundary inclusion is 3-RTF}
	Let $M$ be a compact $3$-manifold with tori boundary and let $T$ be a boundary component. Suppose either the interior of $M$ is hyperbolic with finite volume, or $M$ is Seifert fibered over an orbifold $B$ such that $\chi_o(B)<0$ and $B$ has no cone points of even order. Then $\pi_1(T)$ is $3$-RTF in $\pi_1(M)$.
\end{lemma}
\begin{proof}
	We focus on the case where $M$ is Seifert fibered. We will use the exact sequence (\ref{eqn:exact sequence}) again. Note that $\pi_1(B)$ embeds in $\mathrm{PGL}_2\R\cong\Isom(\Hbb^2)$ as a lattice and, up to a conjugation, $H\defeq p(\pi_1(T))$ is a subgroup of 
	$$P\defeq \left\{\begin{bmatrix} 1 &x\\ 0&1 \end{bmatrix}:x\in\R\right\}\cong (\R,+)$$ 
	in $\mathrm{PGL}_2\R$.
	As a result, each $h\in H$ has a unique square root $\sqrt{h}\in P$, i.e. one that satisfies $(\sqrt{h})^2=h$. Also note that $P\cap\pi_1(B)=H$.
	
	To show that $\pi_1(T)$ is $3$-RTF in $\pi_1(M)$, it suffices to show that $H$ is $3$-RTF in $\pi_1(B)$. Suppose $g\in\pi_1(B)$ satisfies $gh_1gh_2=id$ for some $h_1,h_2\in H$. We need to show $g\in H$. Let $h^*\defeq\sqrt{h_1h_2}\in P$ and $g^*\defeq gh^*\in \textrm{Isom}(\mathbb{H}^2)$. Since $P$ is abelian, we note that $(h^*)^{-1}h_2(h^*)^{-1}h_1=id$, i.e. $(h^*)^{-1}h_2=[(h^*)^{-1}h_1]^{-1}$. Thus we have $$(g^*)^{-1}=(h^*)^{-1}g^{-1}=(h^*)^{-1}h_1 g h_2=(h^*)^{-1}h_1 g^* (h^*)^{-1} h_2=((h^*)^{-1}h_1)g^*((h^*)^{-1}h_1)^{-1}.$$
	We have three cases:
	\begin{enumerate}
		\item $g^*$ fixes some point in $\mathbb{H}^2$. Then the fixed point set is a geodesic subspace $X$ in $\Hbb^2$, which must be preserved by $(h^*)^{-1}h_1$. Since $(h^*)^{-1}h_1$ lies in $P$, this is impossible unless
		\begin{enumerate}
			\item[(1a)] $(h^*)^{-1}h_1=id$; or
			\item[(1b)] $X=\Hbb^2$.
		\end{enumerate}
		In the first subcase, we get $h^*=h_1=h_2$ and $g^*=(g^*)^{-1}$, but now $g^*=gh^*=gh_1$ is an element of $\pi_1(B)$, which contains no $2$-torsion since $B$ has no cone points of even order. So $g^*=id$ and $g=h_1^{-1}\in H$. In the second subcase, we have $g^*=id$, i.e. $g=(h^*)^{-1}$ which lies in $P\cap \pi_1(B)=H$.
		
		\item $g^*$ is parabolic. Then $g^*$ fixes a unique point on $\partial\Hbb^2$, which is also fixed by $(h^*)^{-1}h_1$. So either
		\begin{enumerate}
			\item[(2a)] $(h^*)^{-1}h_1=id$; or
			\item[(2b)] $g^*$ fixes the unique fixed point of $P$.
		\end{enumerate}
		The first subcase (2a) is similar to (1a). In the second subcase, we have $g^*\in P$, and thus $g=g^*(h^*)^{-1}\in P\cap\pi_1(B)=H$.
		
		\item $g^*$ is hyperbolic, possibly further composed with a reflection across the axis of translation. Then $(h^*)^{-1}h_1$ must switch the two unique points on $\partial\Hbb^2$ fixed by $g^*$, which is impossible since $(h^*)^{-1}h_1$ is parabolic.
	\end{enumerate}

	For the case where the interior of $M$ is hyperbolic with finite volume, a similar argument works by replacing $\textrm{Isom}(\Hbb^2)$ by $\textrm{Isom}^+(\mathbb{H}^3)$. This is even easier since $M$ is orientable and $\pi_1(M)$ is torsion-free, and thus we omit the proof.
\end{proof}

We do not know if the result above is optimal in the hyperbolic case.
\begin{quest}
	Is there some $n>3$ such that for any hyperbolic $3$-manifold $M$ of finite volume with tori cusps, every peripheral subgroup of $\pi_1(M)$ is $n$-RTF?
\end{quest}

It is essential in Lemma \ref{lemma:boundary inclusion is 3-RTF} to assume that the orbifold $B$ has no cone points of even order: If $z_0$ represents the singular fiber over a cone point of order $2m$, then the square of $z_0^m$ represents a regular fiber, which lies in peripheral subgroups. The analogous statement of Lemma \ref{lemma:boundary inclusion is 3-RTF} does not hold for the twisted $I$-bundle $K$ over a Klein bottle with respect to its torus boundary, since the peripheral $\Z^2$ subgroup has index two in $\pi_1(K)$ which is the fundamental group of the core Klein bottle.

Excluding these pieces in a non-geometric prime $3$-manifold, the bound in Lemma \ref{lemma: scl gap of hyp elem in 3mfd} can be improved to $1/6$.

\begin{lemma}\label{lemma: 1/6 bound for 3-mfds}
	Suppose $M$ is a non-geometric prime $3$-manifold where none of the pieces in its geometric decomposition is the twisted $I$-bundle over a Klein bottle or Seifert fibered with a fiber of even multiplicity. Then any $g\in \pi_1(M)$ acting hyperbolically on the JSJ-tree has $\scl_M(g)\ge 1/6$.
\end{lemma}
\begin{proof}
	This is simply a combination of Lemma \ref{lemma:boundary inclusion is 3-RTF} and Theorem \ref{thm: $n$-RTF gap, weak version}.
\end{proof}

Now we are in a place to prove a gap theorem for non-geometric prime $3$-manifolds. 
\begin{thm}\label{thm:non-geometric prime factors}
	Let $M$ be a non-geometric prime $3$-manifold. Then there is a constant $C=C(M)>0$ such that either $\scl_M(g)\ge C$ or $\scl_M(g)=0$ for any $g\in \pi_1(M)$. Moreover, if $\scl_G(g)<1/48$, then either $g$ is conjugate to its inverse or it is represented by a loop supported in a single piece of the geometric decomposition of $M$.
\end{thm}
\begin{proof}
	Endow $\pi_1(M)$ with the structure of a graph of groups according to the geometric decomposition, where the vertex groups are the fundamental groups of geometric pieces $\mathcal{N}$ and the edge groups are the fundamental groups of those tori $\mathcal{T}$ we cut along. By Lemma \ref{lemma: scl gap of hyp elem in 3mfd}, if $g$ is hyperbolic, then either $g$ is conjugate to its inverse and $\scl_G(g)=0$, or $\scl_G(g)\ge1/48$.
	
	On the other hand, by Theorem \ref{thm: vertex scl for non-geometric prime factors}, there is a constant $C_N>0$ for each geometric piece $N$ such that, if $g$ is conjugate into $\pi_1(N)$ then either $\scl_M(g)\ge C_N$ or $\scl_M(g)=0$.
	
	Combining the two parts above, for any $g\in\pi_1(M)$, we have either $\scl_M(g)=0$ or $\scl_M(g)\ge C$, where $C:=\min\{1/48,C_N\}>0$ with $N$ ranging over all geometric pieces of $M$.
\end{proof}

\subsubsection{Proofs of main results of this section}
We can now prove Theorem \ref{thm:3mfdgap}.
\begin{proof}[Proof of Theorem \ref{thm:3mfdgap}]
	The prime decomposition splits $\pi_1(M)$ as a free product. By Lemma \ref{lemma: freeprod preserves gap}, we have either $\scl_M(g)\ge 1/12$ or $\scl_M(g)=0$, unless $g$ is represented by a loop supported in a single prime factor. Thus it suffices to prove the result for any prime manifold. Depending on whether the prime manifold $M$ admits one of the eight geometries, the result follows from either Theorem \ref{thm: Gap for geometric manifolds} or Theorem \ref{thm:non-geometric prime factors}.
\end{proof}

Moreover, following the proof above, we can classify elements whose scl vanishes and list the sources of elements with scl less than $1/48$.
\begin{thm}\label{thm: small scl in 3mfds}
	For any $3$-manifold $M$, if a null-homologous element $g\in\pi_1(M)$ represented by a loop $\gamma$ has $\scl_M(g)<1/48$, then one of the following cases occurs:
	\begin{enumerate}
		\item there is some $n\in \Z_+$ and $h\in \pi_1(M)$ such that $g^n h g^n h^{-1}=id$;\label{item: mirror case}
		\item $\gamma$ up to homotopy is supported in a prime factor of $M$ that admits $\mathbb{S}^3, \mathbb{E}^3, \mathbb{S}^2\times\mathbb{E}$, $Nil$, or $Sol$ geometry of finite volume;\label{item: 5 geom case}
		\item $\gamma$ up to homotopy is supported in a piece with $\Hbb^3$ geometry of finite volume in the (possibly trivial) geometric decomposition of a prime factor of $M$, such that the geodesic length of $\gamma$ is less than a universal constant $C$;\label{item: short curve in hyp case}
		\item $\gamma$ up to homotopy is supported in a piece $N$ with $\Hbb^2\times\Ebb$ or $\PSLtwotilde\R$ geometry of finite volume in the (possibly trivial) geometric decomposition of a prime factor of $M$, such that $\scl_{(B,\partial B)}(p(g))<1/48$, where $p:\pi_1(N)\to \pi_1(B)$ is the projection induced by the Seifert fibration of $N$ over an orbifold $B$;\label{item: scl small in orbifold Seifert base case}
		\item $g^2$ is represented by a loop in a torus of the JSJ decomposition of a prime factor of $M$;\label{item: edge group case}
	\end{enumerate}
	Moreover, $\scl_M(g)=0$ if and only if we have cases (\ref{item: mirror case}), (\ref{item: 5 geom case}), or the following two special cases of (\ref{item: scl small in orbifold Seifert base case}) and (\ref{item: edge group case}):
	\begin{enumerate}
		\item[(\ref{item: scl small in orbifold Seifert base case}*)] $g^nh g^n h^{-1}=z^\ell$ for some $n>0$, $\ell\in\Z$ and $h\in \pi_1(N)$, where $z$ represents a regular fiber, so that $(z,N)$ is a vanishing pair (see Definition \ref{def: vanishing pair});

		\item[(\ref{item: edge group case}*)] $g^2$ is represented by a loop in a JSJ torus $T$ that identifies boundary components $\partial_1\subset N_1$ and $\partial_2\subset N_2$ of (possibly the same) pieces $N_1,N_2$, and $g^2=ab\in\pi_1(T)$, such that under the identification of $T$ with $\partial_1$ and $\partial_2$ respectively, the images $a',b'$ of $a,b$ in $\pi_1 N_1$ and $\pi_1 N_2$ satisfy: (a) either $(a',N_1)$ and $(b',N_2)$ are both vanishing pairs, (b) or $N_1=N_2$ is Seifert fibered with regular fiber represented by $z$ so that $a'=z^m$ and $b'=z^{-m}$ for some $m\in\Z$.
		
	\end{enumerate}
	
\end{thm}
\begin{proof}
	Suppose $\scl_M(g)<1/48$. If $g^n h g^n h^{-1}= id$ for some $h\in \pi_1(M)$ and $n\in \Z_+$, then $\scl_M(g)=0$. \emph{Assume this is not the case in the sequel}. Then $\gamma$ up to homotopy is supported in a prime factor $M_0$ by Lemma \ref{lemma: freeprod preserves gap}. Now we focus on $M_0$ and $\scl_{M_0}(g)=\scl_M(g)<1/48$. 
	
	If $M_0$ has one of the five geometries in case (\ref{item: 5 geom case}), then $\scl_{M_0}(g)=0$ since $\pi_1(M_0)$ is amenable.
	
	If $M_0$ has hyperbolic geometry, take $\delta=1/48$, $m=3$ and $C=\kappa(\delta,m)$ in Theorem \ref{thm: accumulation},
	then $\scl_{M_0}(g)<1/48$ implies that $\gamma$ has geodesic length no more than $C$. Moreover, $\scl_{M_0}(g)>0$ unless $g=id$ by Theorem \ref{thm: hyp rel gap}.
	
	If $M_0$ has $\Hbb^2\times\Ebb$ or $\PSLtwotilde\R$ geometry, then $M_0$ is Seifert fibered over a closed orbifold $B$ with $\chi_o(B)<0$. The monotonicity of scl implies that $\scl_B(p(g))<1/48$ where $p:\pi_1(M_0)\to \pi_1(B)$ is the projection induced by the Seifert fibration. Moreover, by Theorem \ref{thm: closed Seifert fibered prime factors}, if $\scl_{M_0}(g)=0$, then we are in case (\ref{item: scl small in orbifold Seifert base case}*) since we assumed above that case (\ref{item: mirror case}) does not occur.
	
	Now suppose $M_0$ has a nontrivial geometric decomposition. By Theorem \ref{thm:non-geometric prime factors} and the assumption above, $\gamma$ up to homotopy is supported in some piece $N$ of the JSJ decomposition of $M_0$. Suppose $\gamma$ is homotopic to a loop in a JSJ torus $T$, then by Theorem \ref{thm: edge scl for non-geometric prime factors}, $\scl_{M_0}(g)=0$ only occurs if we are in 
	case (\ref{item: edge group case}*).
	
	Suppose $\gamma$ is not homotopic to a loop in a JSJ torus. Then depending on the type of $N$ we have three cases:
	\begin{enumerate}[(a)]
		\item If $N$ has hyperbolic geometry in the interior, then by Theorem \ref{thm: accumulation}, $\gamma$ has geodesic length less than the universal constant $C$ mentioned above in the closed hyperbolic case, and $\scl_{M_0}(g)>0$ by Lemma \ref{lemma:vertex group hyp case} since $\gamma$ is not boundary parallel.
		
		\item If $N$ is Seifert fibered over an orbifold $B$ with $\chi_o(B)<0$, then $p: \pi_1(N)\to \pi_1(B)$ preserves scl since $B$ has boundary, and we have $\scl_{(B,\partial B)}(p(g))=\scl_{(N,\partial N)}(g)\le \scl_{M_0}(g)<1/48$ by Lemma~\ref{lemma: simple estimate of scl in vertex groups}. Moreover, if $\scl_{M_0}(g)=0$, then we have $\scl_{(N,\partial N)}(g)=0$.
		Since $g$ is not conjugate into any peripheral subgroup, the classification in Lemma \ref{lemma:vertex group Seifert case} 
		implies that, either $g^n=z^m$ for some $n>0$ and $m\in\Z\setminus\{0\}$, or there is $h\in \pi_1 N$ and $m\in\Z$ such that $ghgh^{-1}=z^m$ and $\scl_N(g^2-z^m)=0$.
		In the former case, $g^nhg^n h^{-1}=z^{2m}$ with $h=id$, and $\scl_{M_0}(z)=n\scl_{M_0}(g)/|m|=0$, so we are in case (\ref{item: scl small in orbifold Seifert base case}*) by Corollary \ref{cor: fiber vanishing}.
		For the other case, we must have $m\neq 0$ by our assumption at the beginning of the proof. 
		Now $\scl_N(g^2-z^m)=0$ implies $\scl_{M_0}(g^2-z^m)=0$ by monotonicity. Thus $\scl_{M_0}(z^m)=\scl_{M_0}(g^2)=0$, from which we get $\scl_{M_0}(z)=0$ as $m\neq 0$. So we again have case (\ref{item: scl small in orbifold Seifert base case}*).

		
		\item If $N$ is the twisted $I$-bundle over the Klein bottle, then $g^2$ is supported in $\partial N$. Moreover, $\scl_{M_0}(g)=0$ implies $\scl_{M_0}(g^2)=0$, which happens only if we have case (\ref{item: edge group case}*).
	\end{enumerate}
\end{proof}

Note that there are few types of conjugacy classes as above with scl strictly between $0$ and $1/48$, so one may expect many $3$-manifolds to have only finitely many such conjugacy classes. For example, Michael Hull suggested the following statement in personal communications.

\begin{cor}\label{cor: finitely many conj class with small scl}
	Let $M$ be a prime $3$-manifold with only hyperbolic pieces in its (possibly trivial) geometric decomposition, then $\scl_M(g)>0$ for all $g\neq id$ and there are only finitely many conjugacy classes $g$ with $\scl_M(g)<1/48$.
\end{cor}
\begin{proof}
	Let $g\neq id$. If $g$ is not conjugate into any vertex group, then $\scl_M(g)\ge 1/6$ by Lemma \ref{lemma: 1/6 bound for 3-mfds}. If $g$ lies in a hyperbolic piece $N$ and is not conjugate into any peripheral subgroup, then $\scl_M(g)\ge \scl_{(N,\partial N)}(g)>0$ by Lemma \ref{lemma:vertex group hyp case}. If $g$ is conjugate into some JSJ torus, then the proof of Theorem \ref{thm: edge scl for non-geometric prime factors} shows that $\scl_M$ restricted to the edge groups can be computed by a degenerate norm $\|\cdot \|$. Having only hyperbolic pieces implies that $\|\cdot \|$ has trivial vanishing locus by Corollary \ref{cor:edge group for hyp}, and thus $\scl_M(g)>0$ for all such $g$, and there are only finitely many integer points with norm less than $1/48$.
	
	Suppose $\scl_M(g)<1/48$. Then by our assumption, $g$ must fall into cases (\ref{item: short curve in hyp case}) or (\ref{item: edge group case}) in Theorem \ref{thm: small scl in 3mfds}. We have discussed the case where $g$ is conjugate into an edge group above. As for the other case, there are only finitely many conjugacy classes supported in a hyperbolic piece of $M$ with bounded geodesic length since each piece has finite volume; See for example \cite{Canary-Leininger}.
\end{proof}

One should not expect a similar result in general if we allow Seifert fibered pieces, for the norm $\|\cdot \|$ may have nontrivial vanishing locus. For example, let $M_1$ be a hyperbolic $3$-manifold with one cusp so that a loop $\gamma$ on the boundary has small positive scl. Let $M_2$ be a Seifert fibered $3$-manifold with a nonorientable base space and one torus boundary. Glue $M_1$ and $M_2$ along their boundary to obtain $M$ so that $\gamma$ is not identified with the fiber direction of $M_2$. Then all elements of the form $gz^n$ have the same small positive scl value in $M$, where $g$ and $z$ represent the image of $\gamma$ and the fiber direction of $M_2$ respectively.

\bibliographystyle{alpha}
\bibliography{sclgaps}

\begin{thebibliography}{Che18b}

\bibitem[AD{\v{S}}18]{LRC}
Yago Antol\'{i}n, Warren Dicks, and Zoran {\v{S}}uni\'{c}.
\newblock Left relativity convex subgroups.
\newblock In {\em Topological methods in group theory}, volume 451 of {\em
  London Math. Soc. Lecture Note Ser.}, pages 1--18. Cambridge Univ. Press,
  Cambridge, 2018.

\bibitem[Bav91]{bavard}
Christophe Bavard.
\newblock Longueur stable des commutateurs.
\newblock {\em Enseign. Math. (2)}, 37(1-2):109--150, 1991.

\bibitem[BBF16]{BBF}
Mladen Bestvina, Ken Bromberg, and Koji Fujiwara.
\newblock Stable commutator length on mapping class groups.
\newblock {\em Ann. Inst. Fourier (Grenoble)}, 66(3):871--898, 2016.

\bibitem[BFH16]{cicle_quasimorph_modern}
Michelle Bucher, Roberto Frigerio, and Tobias Hartnick.
\newblock A note on semi-conjugacy for circle actions.
\newblock {\em Enseign. Math.}, 62(3-4):317--360, 2016.

\bibitem[BM99]{BurgerMonod}
M.~Burger and N.~Monod.
\newblock Bounded cohomology of lattices in higher rank {L}ie groups.
\newblock {\em J. Eur. Math. Soc. (JEMS)}, 1(2):199--235, 1999.

\bibitem[BM02]{BurgerMonod02}
M.~Burger and N.~Monod.
\newblock Continuous bounded cohomology and applications to rigidity theory.
\newblock {\em Geom. Funct. Anal.}, 12(2):219--280, 2002.

\bibitem[Bou95]{Bouarich}
Abdessalam Bouarich.
\newblock Suites exactes en cohomologie born\'ee r\'eelle des groupes discrets.
\newblock {\em C. R. Acad. Sci. Paris S\'er. I Math.}, 320(11):1355--1359,
  1995.

\bibitem[Cal08]{Cal:sclhypmfd}
Danny Calegari.
\newblock Length and stable length.
\newblock {\em Geom. Funct. Anal.}, 18(1):50--76, 2008.

\bibitem[Cal09a]{calegari:extremal}
Danny Calegari.
\newblock Faces of the scl norm ball.
\newblock {\em Geom. Topol.}, 13(3):1313--1336, 2009.

\bibitem[Cal09b]{Cal:sclbook}
Danny Calegari.
\newblock {\em scl}, volume~20 of {\em MSJ Memoirs}.
\newblock Mathematical Society of Japan, Tokyo, 2009.

\bibitem[CF10]{CF:sclhypgrp}
Danny Calegari and Koji Fujiwara.
\newblock Stable commutator length in word-hyperbolic groups.
\newblock {\em Groups Geom. Dyn.}, 4(1):59--90, 2010.

\bibitem[CFL16]{CFL16}
Matt Clay, Max Forester, and Joel Louwsma.
\newblock Stable commutator length in {B}aumslag-{S}olitar groups and
  quasimorphisms for tree actions.
\newblock {\em Trans. Amer. Math. Soc.}, 368(7):4751--4785, 2016.

\bibitem[CH20]{CH:sclgapgog_old}
Lvzhou Chen and Nicolaus Heuer.
\newblock Spectral gap of scl in graphs of groups and $3$-manifolds, 2020.
\newblock arXiv preprint 1910.14146v2.

\bibitem[CH23]{CH:RAAG_chain}
Lvzhou Chen and Nicolaus Heuer.
\newblock Stable commutator length in right-angled {A}rtin and {C}oxeter
  groups.
\newblock {\em J. Lond. Math. Soc. (2)}, 107(1):1--60, 2023.

\bibitem[Che18a]{Chen:sclfp}
Lvzhou Chen.
\newblock Scl in free products.
\newblock {\em Algebr. Geom. Topol.}, 18(6):3279--3313, 2018.

\bibitem[Che18b]{Chen:sclfpgap}
Lvzhou Chen.
\newblock Spectral gap of scl in free products.
\newblock {\em Proc. Amer. Math. Soc.}, 146(7):3143--3151, 2018.

\bibitem[Che20]{Chen:sclBS}
Lvzhou Chen.
\newblock Scl in graphs of groups.
\newblock {\em Invent. Math.}, 221(2):329--396, 2020.

\bibitem[CL07]{Canary-Leininger}
Richard~D. Canary and Christopher~J. Leininger.
\newblock Kleinian groups with discrete length spectrum.
\newblock {\em Bull. Lond. Math. Soc.}, 39(2):189--193, 2007.

\bibitem[DH91]{DH91}
Andrew~J. Duncan and James Howie.
\newblock The genus problem for one-relator products of locally indicable
  groups.
\newblock {\em Math. Z.}, 208(2):225--237, 1991.

\bibitem[DM84]{davis1984}
Michael~W. Davis and John~W. Morgan.
\newblock Finite group actions on homotopy {$3$}-spheres.
\newblock In {\em The {S}mith conjecture ({N}ew {Y}ork, 1979)}, volume 112 of
  {\em Pure Appl. Math.}, pages 181--225. Academic Press, Orlando, FL, 1984.

\bibitem[FFT19]{RAAGgap1}
Talia Fern\'{o}s, Max Forester, and Jing Tao.
\newblock Effective quasimorphisms on right-angled {A}rtin groups.
\newblock {\em Ann. Inst. Fourier (Grenoble)}, 69(4):1575--1626, 2019.

\bibitem[FM98]{homrigidityFM}
Benson Farb and Howard Masur.
\newblock Superrigidity and mapping class groups.
\newblock {\em Topology}, 37(6):1169--1176, 1998.

\bibitem[Fri17]{Frigerio}
Roberto Frigerio.
\newblock {\em Bounded cohomology of discrete groups}, volume 227 of {\em
  Mathematical Surveys and Monographs}.
\newblock American Mathematical Society, Providence, RI, 2017.

\bibitem[FSTar]{RAAGgap2}
Max Forester, Ignat Soroko, and Jing Tao.
\newblock Genus bounds in right-angled artin groups.
\newblock {\em Publ. Mat.}, To appear.

\bibitem[Ghy87]{ghys}
\'{E}tienne Ghys.
\newblock Groupes d'hom\'{e}omorphismes du cercle et cohomologie born\'{e}e.
\newblock In {\em The {L}efschetz centennial conference, {P}art {III} ({M}exico
  {C}ity, 1984)}, volume~58 of {\em Contemp. Math.}, pages 81--106. Amer. Math.
  Soc., Providence, RI, 1987.

\bibitem[Ghy01]{Ghys:circle}
\'{E}tienne Ghys.
\newblock Groups acting on the circle.
\newblock {\em Enseign. Math. (2)}, 47(3-4):329--407, 2001.

\bibitem[Gre90]{graphproductthesis}
Elisabeth~Ruth Green.
\newblock {\em Graph products of groups}.
\newblock PhD thesis, University of Leeds, 1990.

\bibitem[Heu19]{Heuer}
Nicolaus Heuer.
\newblock Gaps in {SCL} for amalgamated free products and {RAAG}s.
\newblock {\em Geom. Funct. Anal.}, 29(1):198--237, 2019.

\bibitem[Hol75]{Convanal}
Richard~B. Holmes.
\newblock {\em Geometric functional analysis and its applications}.
\newblock Springer-Verlag, New York-Heidelberg, 1975.
\newblock Graduate Texts in Mathematics, No. 24.

\bibitem[IK18]{IK}
Sergei~V. Ivanov and Anton~A. Klyachko.
\newblock Quasiperiodic and mixed commutator factorizations in free products of
  groups.
\newblock {\em Bull. Lond. Math. Soc.}, 50(5):832--844, 2018.

\bibitem[IMT19]{gentorsion}
Tetsuya Ito, Kimihiko Motegi, and Masakazu Teragaito.
\newblock Generalized torsion and decomposition of {$3$}--manifolds.
\newblock {\em Proc. Amer. Math. Soc.}, 147(11):4999--5008, 2019.

\bibitem[Joh79]{JSJ2}
Klaus Johannson.
\newblock {\em Homotopy equivalences of {$3$}-manifolds with boundaries},
  volume 761 of {\em Lecture Notes in Mathematics}.
\newblock Springer, Berlin, 1979.

\bibitem[JS79]{JSJ1}
William Jaco and Peter~B. Shalen.
\newblock Seifert fibered spaces in {$3$}-manifolds.
\newblock In {\em Geometric topology ({P}roc. {G}eorgia {T}opology {C}onf.,
  {A}thens, {G}a., 1977)}, pages 91--99. Academic Press, New York-London, 1979.

\bibitem[KM96]{homrigidityKM}
Vadim~A. Kaimanovich and Howard Masur.
\newblock The {P}oisson boundary of the mapping class group.
\newblock {\em Invent. Math.}, 125(2):221--264, 1996.

\bibitem[LS77]{LyndonSchupp}
Roger~C. Lyndon and Paul~E. Schupp.
\newblock {\em Combinatorial group theory}.
\newblock Springer-Verlag, Berlin-New York, 1977.
\newblock Ergebnisse der Mathematik und ihrer Grenzgebiete, Band 89.

\bibitem[Orl72]{Orlik:Seifert}
Peter Orlik.
\newblock {\em Seifert manifolds}, volume Vol. 291 of {\em Lecture Notes in
  Mathematics}.
\newblock Springer-Verlag, Berlin-New York, 1972.

\bibitem[Per02]{Perelman1}
Grisha Perelman.
\newblock The entropy formula for the ricci flow and its geometric
  applications, 2002.
\newblock arXiv preprint math/0211159.

\bibitem[Per03a]{Perelman3}
Grisha Perelman.
\newblock Finite extinction time for the solutions to the ricci flow on certain
  three-manifolds, 2003.
\newblock arXiv preprint math/0307245.

\bibitem[Per03b]{Perelman2}
Grisha Perelman.
\newblock Ricci flow with surgery on three-manifolds, 2003.
\newblock arXiv preprint math/0303109.

\bibitem[Sco83]{Scott}
Peter Scott.
\newblock The geometries of {$3$}-manifolds.
\newblock {\em Bull. London Math. Soc.}, 15(5):401--487, 1983.

\bibitem[Ser80]{Serre}
Jean-Pierre Serre.
\newblock {\em Trees}.
\newblock Springer-Verlag, Berlin-New York, 1980.
\newblock Translated from the French by John Stillwell.

\bibitem[Thu78]{Thurstonnotes}
William Thurston.
\newblock The geometry and topology of 3-manifolds.
\newblock {\em Lecture note}, 1978.

\bibitem[Thu82]{Thurston}
William~P. Thurston.
\newblock Three-dimensional manifolds, {K}leinian groups and hyperbolic
  geometry.
\newblock {\em Bull. Amer. Math. Soc. (N.S.)}, 6(3):357--381, 1982.

\bibitem[Thu97]{Thurstonbook}
William~P. Thurston.
\newblock {\em Three-dimensional geometry and topology. {V}ol. 1}, volume~35 of
  {\em Princeton Mathematical Series}.
\newblock Princeton University Press, Princeton, NJ, 1997.
\newblock Edited by Silvio Levy.

\bibitem[WZ10]{Acyl}
Henry Wilton and Pavel Zalesskii.
\newblock Profinite properties of graph manifolds.
\newblock {\em Geom. Dedicata}, 147:29--45, 2010.

\end{thebibliography}

\end{document}